\documentclass[10pt]{amsart}
\usepackage{amsthm,amsmath,amssymb,amscd,graphics,enumerate, stmaryrd,xspace,verbatim, epic, eepic,url}

\usepackage[usenames,dvipsnames]{color}

\usepackage[colorlinks=true]{hyperref}

\usepackage[active]{srcltx}
\usepackage[all]{xypic}
\SelectTips{cm}{}

{
   \newtheorem{theorem}[subsubsection]{Theorem}
      \newtheorem*{theorem*}{Theorem}
   \newtheorem{proposition}[subsubsection]{Proposition}

   \newtheorem{lemma}[subsubsection]{Lemma}

   \newtheorem{corollary}[subsubsection]{Corollary}
   
   \newtheorem*{conjecture*}{Conjecture}

}
{\theoremstyle{definition}
          \newtheorem*{exercise*}{Exercise}
   
   \newtheorem{example}[subsubsection]{Example}
   \newtheorem*{example*}{Example}

   \newtheorem*{definition*}{Definition}
   
   \newtheorem{remark}[subsubsection]{Remark}

}
\newcommand{\RR}{{\mathbb{R}}}

\newcommand{\QQ}{{\mathbb{Q}}}
\newcommand{\NN}{{\mathbb{N}}}

\newcommand{\ZZ}{{\mathbb{Z}}}

\newcommand{\GG}{{\mathbb{G}}}

\renewcommand{\AA}{{\mathbb{A}}}

\newcommand{\bB}{{\mathbf{B}}}

\def\logdim{{\rm logrk}}
\def\toto{{\rightrightarrows}}
\def\hatcO{{\widehat\cO}}

\def\hatcJ{{\widehat\cJ}}

\def\ucI{{\underline\cI}}
\def\ucS{{\underline\cS}}
\def\ucP{{\underline\cP}}
\def\ucJ{{\underline\cJ}}
\def\ucC{{\underline\cC}}

\def\hatX{{\widehat X}}
\def\hati{{\widehat i}}

\def\hatZ{{\widehat Z}}
\def\hatV{{\widehat V}}
\def\ket{_{\rm k\acute et}}

\newcommand{\ff}{{\mathfrak{f}}}

\newcommand{\fX}{{\mathfrak{X}}}
\newcommand{\fY}{{\mathfrak{Y}}}
\newcommand{\fy}{{\mathfrak{y}}}
\newcommand{\fZ}{{\mathfrak{Z}}}
\newcommand{\fz}{{\mathfrak{z}}}

\newcommand{\fSp}{{\mathfrak{Sp}}}

\newcommand{\cA}{{\mathcal A}}
\newcommand{\cB}{{\mathcal B}}
\newcommand{\cC}{{\mathcal C}}
\renewcommand{\cD}{{\mathcal D}}

\newcommand{\cF}{{\mathcal F}}

\newcommand{\cI}{{\mathcal I}}
\newcommand{\cJ}{{\mathcal J}}

\newcommand{\cM}{{\mathcal M}}

\newcommand{\cN}{{\mathcal N}}
\newcommand{\cO}{{\mathcal O}}
\newcommand{\cP}{{\mathcal P}}

\renewcommand{\cR}{{\mathcal R}}
\newcommand{\cS}{{\mathcal S}}
\newcommand{\cT}{{\mathcal T}}

\newcommand{\cW}{{\mathcal W}}
\newcommand{\cX}{{\mathcal X}}
\newcommand{\cY}{{\mathcal Y}}
\newcommand{\cZ}{{\mathcal Z}}

\newcommand{\fB}{{\mathfrak B}}

\newcommand{\ophi}{{\overline{\phi}}}

\newcommand{\oM}{{\overline{M}}}
\newcommand{\oP}{{\overline{P}}}
\newcommand{\oQ}{{\overline{Q}}}
\newcommand{\oR}{{\overline{R}}}
\newcommand{\of}{{\overline{f}}}

\newcommand{\oa}{{\overline{a}}}
\newcommand{\ob}{{\overline{b}}}
\newcommand{\oc}{{\overline{c}}}
\newcommand{\oy}{{\overline{y}}}

\newcommand{\cs}{{\rm cs}}

\def\<{\langle}
\def\>{\rangle}

\newcommand{\GL}{\operatorname{GL}}

\newcommand{\Spec}{\operatorname{Spec}}

\newcommand{\Spf}{\operatorname{Spf}}
\newcommand{\Proj}{\operatorname{Proj}}
\newcommand{\cProj}{{{\mathcal P}roj}}
\newcommand{\colim}{\operatorname{colim}}

\newcommand{\RZ}{{\operatorname{RZ}}}

\newcommand{\Val}{{\operatorname{Val}}}

\newcommand{\Hom}{{\operatorname{Hom}}}

\newcommand{\Ker}{{\operatorname{Ker}}}
\newcommand{\Coker}{{\operatorname{Coker}}}

\newcommand{\sh}{{\operatorname{sh}}}

\newcommand{\bfA}{{\mathbf A}}

\newcommand{\ocI}{\overline{{\mathcal I}}}

\newcommand{\osigma}{{\overline{\sigma}}}
\newcommand{\oC}{{\overline{C}}}
\newcommand{\oA}{{\overline{A}}}
\newcommand{\oD}{{\overline{D}}}
\newcommand{\oX}{{\overline{X}}}

\newcommand{\ox}{{\overline{x}}}
\newcommand{\ot}{{\overline{t}}}

\newcommand{\st}{{\operatorname{st}}}

\newcommand{\sing}{{\operatorname{sing}}}

\newcommand{\lcm}{{\operatorname{lcm}}}

\def\:{{\colon}}
\def\.{{,\dots,}}
\def\dim{{\rm dim}}

\def\inv{{\rm inv}}
\def\cln{{\rm cln}}
\def\Inv{{\rm Inv}}

\def\tilY{{\widetilde Y}}

\newcommand{\double}{\genfrac..{0pt}1
{\raise -1pt\hbox{$\scriptstyle\longrightarrow$}}{\raise 3pt\hbox
{$\scriptstyle\longrightarrow$}}}

\renewcommand{\setminus}{\smallsetminus}

\def\sat{{\rm sat}}

\def\nor{{\rm nor}}
\def\int{{\rm int}}
\def\triv{{\rm triv}}
\def\trdeg{{\rm tr.deg.}}

\def\et{_{\rm \acute et}}

\def\tototi{\mathbin{\mathop{\otimes}\limits^{\raise-1pt\hbox
{$\scriptscriptstyle {\rm L}$}}}}

\def\indlim{\mathop{\vrule width0pt height7pt depth
4pt\smash{\lim\limits_{\raise 1pt\hbox to 14.5pt
{\rightarrowfill}}}}}
\def\projlim{\mathop{\vrule width0pt height7pt depth
4pt\smash{\lim\limits_{\raise 1pt\hbox to 14.5pt
{\leftarrowfill}}}}}

\newcommand\displaceamount{3pt}

\newcommand{\doubledown}{\ar@<\displaceamount>[d]\ar@<-\displaceamount>[d]}

\newcommand{\doubleup}{\ar@<\displaceamount>[u]\ar@<-\displaceamount>[u]}

\newcommand{\doubleright}{\ar@<\displaceamount>[r]\ar@<-\displaceamount>[r]}

\newcommand{\res}{{\operatorname{res}}}

\newcommand{\rk}{\operatorname{rk}}

\def\into{\hookrightarrow}
\def\onto{\twoheadrightarrow}

\def\gp{\text{gp}}

\newcommand{\tr}{{\operatorname{tr}}}
\newcommand{\logord}{{\operatorname{logord}}}

\newcommand{\Der}{{\cD er}}
\newcommand{\ord}{{\operatorname{ord}}}
\newcommand{\Log}{{\mathbf{Log}}}

\def\GGm{{\mathbb G}_m}

\def\st{\text{st}}

\def\tilcJ{{\widetilde\cJ}}
\def\tilcI{{\widetilde\cI}}
\def\tilcR{{\widetilde\cR}}

\def\tilucI{{\widetilde\ucI}}
\def\tilucJ{{\widetilde\ucJ}}

\def\tilt{{\widetilde t}}
\def\tils{{\widetilde s}}

\def\tilm{{\widetilde m}}

\def\tilZ{{\widetilde Z}}
\def\tilQ{{\widetilde Q}}
\def\tilP{{\widetilde P}}
\def\tilq{{\widetilde q}}

\def\tilz{{\widetilde z}}

\def\tilf{{\widetilde f}}
\def\tilg{{\widetilde g}}
\def\tilg{{\widetilde r}}
\def\tili{{\widetilde i}}
\def\tilh{{\widetilde h}}
\def\tilX{{\widetilde X}}
\def\tilU{{\widetilde U}}
\def\tilB{{\widetilde B}}
\def\tilH{{\widetilde H}}

\def\toisom{\xrightarrow{{_\sim}}}
\def\hatA{{\widehat A}}
\def\hatC{{\widehat C}}

\def\hatB{{\widehat B}}

\def\bfD{{\mathbf{D}}}
\def\wtimes{{\widehat\otimes}}
\def\hatO{{\widehat O}}
\def\hatphi{{\widehat\phi}}

\def\reg{{\operatorname{reg}}}
\def\supp{{\operatorname{supp}}}

\def\Frac{{\rm Frac}}

\begin{document}
\title{Relative desingularization and principalization of ideals}

\author[D. Abramovich]{Dan Abramovich}
\address{Department of Mathematics, Box 1917, Brown University,
Providence, RI, 02912, U.S.A}
\email{abrmovic@math.brown.edu}

\author[M. Temkin]{Michael Temkin}
\address{Einstein Institute of Mathematics\\
               The Hebrew University of Jerusalem\\
                Edmond J. Safra Campus, Giv'at Ram, Jerusalem, 91904, Israel}
\email{temkin@math.huji.ac.il}

\author[J. W{\l}odarczyk] {Jaros{\l}aw W{\l}odarczyk}
\address{Department of Mathematics, Purdue University\\
150 N. University Street,\\ West Lafayette, IN 47907-2067}
\email{wlodar@math.purdue.edu}

\thanks{This research is supported by  BSF grants 2014365 and 2018193, ERC Consolidator Grant 770922 - BirNonArchGeom, and NSF grant DMS-1759514.}

\date{\today}

\begin{abstract}
In characteristic zero, we construct relative principalization of ideals for logarithmically regular morphisms of logarithmic schemes, and use it to construct logarithmically regular desingularization of morphisms. These constructions are relatively canonical and even functorial with respect to logarithmically regular morphisms and arbitrary base changes. {\em Relative canonicity} means, that the principalization requires a fine enough non-canonical modification of the base, and once it is chosen the process is canonical. As a consequence we deduce the semistable reduction theorem over arbitrary valuation rings.
\end{abstract}
\maketitle

\section{Introduction}
If not said to the contrary, all schemes and stacks considered in this paper are assumed to be qe (quasi-excellent), noetherian and of characteristic zero, though we recall these assumptions when formulating the main results.

\subsection{Motivation}
Comparing to desingularization of schemes and varieties, the theory of resolution of morphisms or families has a much shorter history. The classical example of such problem is stable reduction of curves. This case is extremely important and the solution is relatively canonical, and even holds in mixed characteristics, but it is a low-dimensional exception. With this exception there were two main achievements in the characteristic zero case, and both were based on breakthroughs in the absolute desingularization theory.

First, Mumford et al. observed in \cite{KKMS} that Hironaka's theorem implies that if $R$ is a DVR of pure characteristic 0 and $X$ is flat, of finite type and with a smooth generic fiber $X_\eta$ over $S=\Spec(R)$, then there exists a modification $X'\to X$ such that $X'_\eta=X_\eta$ and $X'\to S$ is logarithmically smooth (or toroidal). This implication is in fact obvious and the main goal of \cite{KKMS} is to prove that $X'$ can be made semistable, after base change, by combinatorial methods. By later works Hironaka's method can be made canonical, and then the modification $X'\to X$ becomes canonical, but it heavily depends on $R$ and changes completely after any ramified extension of $R$.

Second, Abramovich and Karu used in \cite{AK} de Jong's method of alterations to prove that any dominant morphism of finite type between integral schemes $Z\to B$ can be made logarithmically smooth after modifying $Z$ and $B$. This was the only known result applying to any dimension of the base and of the fibers, but the method is non-canonical and even a smooth generic fiber $Z_\eta$ can be modified by it.

In addition, there were a series of works by Cutkosky, see \cite{Cutkosky-m32} and \cite{Cutkosky-toroidalization}, where toroidalization of a dominant morphism $Z\to B$ with $\dim(Z)\le 3$ of characteristic zero was achieved. These works have different goals: one starts with smooth $Z$ and $B$ ---  not a real restriction due to Hironaka --- and the main challenge is to achieve torodalization by only blowing up $Z$ and $B$ at smooth centers, with the added flexibility of working locally. The restriction to blowings up  makes the problem much subtler. It leads to applications somewhat  different from ours --- with a stronger emphasis on strong factorization of birational maps rather than families and moduli.

The following questions were open until now: Can one resolve $Z/B$ relatively canonically? Can one at least achieve that the smooth locus of $Z/B$ is kept unchanged? And, finally, an even more specific but famous question extending the semistable reduction to non-discrete valuation rings $R$: given a smooth proper variety $X_\eta$ over ${\rm Frac}(R)$, can one extend it to a proper logarithmically smooth $R$-scheme $X$?

In this paper we answer all these questions affirmatively (in characteristic zero). Moreover, we construct a \emph{relatively canonical} resolution of morphisms compatible with arbitrary base changes $B'\to B$ and logarithmically smooth morphisms $Z'\to Z$.\footnote{\emph{Relatively canonical} is given precise meaning below.} Our construction is strongly based on ideas and methods of the classical desingularization, but involves completely new ingredients, such as extensive use of logarithmic geometry and non-representable modifications. To the best of our knowledge, this is the first relative desingularization algorithm for arbitrary dimensions. We are also working on another paper, where a \emph{canonical} resolution of proper morphisms will be achieved.

Finally, in the end of our project Belotto Da Silva and Bierstone informed us about their work \cite{BB}. Its main result provides \emph{local monomialization} of dominant morphisms in characteristic zero in arbitrary dimensions. A careful comparison of our methods should be done in the future, here we only note that there are definitely some similar ideas, such as the use of relative logarithmic differentials and the centers they define.  As in the works of Cutkosky, the paper \cite{BB} does not aim at a global functorial procedure, and on the other hand, they restrict the class of modifications of the base to the simplest possible ones. In comparison, in our work we are mainly concerned with the relative method, and pay much less care to modifications of the base. Another major difference is that we use arbitrary logarithmically smooth morphisms and allow to blow up fractional powers of monomials ideals. Our resulting algorithm is faster and more functorial, but it has to work with (mild) stacks.

\subsection{Statement of main results}\label{mainres}

\subsubsection{Overview}
In this paper we accomplish our program on functorial desingularization (or semistable reduction) of morphisms in characteristic zero. Oversimplifying, the main idea is to take the classical desingularization method of Hironaka, Giraud, Bierstone and Milman, Villamayor, W{\l}odarczyk, Kollar, and others and adjust all its ingredients to logarithmic and relative settings. In particular, one should first obtain a logarithmic relative principalization of ideals on logarithmically smooth morphisms, and deduce desingularization from it. Half of this work was done in \cite{ATW-principalization}, where we constructed logarithmically functorial desingularization of logarithmic varieties. It turned out that the stronger functoriality forces one to use non-representable stack theoretic modifications, that we called Kummer blow ups, but the algorithm's structure  became simpler than in the classical case.

In this paper, we show that the algorithm constructed in \cite{ATW-principalization} applies to morphisms $Z\to B$ once one adjusts all definitions to the relative setting. Developing plenty of foundations in the relative setting is one of main technical contributions of the paper. However, there is also a principally new relative phenomenon which required a new ingredient: one might need to modify the base $B$ before --- or in the process of --- running the desingularization algorithm. Technically this happens because relative differential saturation does not have to produce a monomial ideal. To overcome this issue we prove a surprisingly hard monomialization theorem \ref{monomialization}, which ``repairs'' the differential saturation by a modification of the base.

An additional challenge we took on in this paper is to deal with morphisms not necessarily of finite type. In particular, we study principalization for logarithmically regular morphisms (defined and studied in a separate paper \cite{Molho-Temkin}) rather than for logarithmically smooth morphisms. This has the advantage of applications to other categories, such as formal schemes and analytic spaces. This direction is new even for classical desingularization, since \cite{Temkin-embedded} does not construct principalization for qe schemes.

\subsubsection{Absolute desingularization}
Our first application of working with morphisms not of finite type  generalizes \cite[Theorem~8.3.4]{ATW-principalization}  to qe schemes, see \S\ref{abscase}:

\begin{theorem}\label{absdesingth}
For any formally equidimensional generically logarithmically regular qe noetherian logarithmic scheme $Z$ of characteristic zero, there exists a modification $\cF(Z)\:Z_\res\to Z$ such that $Z_\res$ is logarithmically regular. Moreover, one can achieve that the process assigning the morphism $\cF(Z)$ to $Z$ is functorial for logarithmically regular quasi-saturated morphisms $Z'\to Z$. In particular, $\cF(Z)$ is an isomorphism over logarithmic regularity locus $Z_\reg\subseteq Z$. Finally, each $\cF(Z)$ can be realized as a blow up whose center is disjoint from $Z_\reg$.
\end{theorem}

See  \cite[Section~8.2.2]{ATW-principalization} for the notion of \emph{quasi-saturated morphisms} used in this statement.
\begin{remark}
(i) It might be the case that the center of the blow up can be chosen functorially with respect to surjective logarithmically regular morphisms, but we do not pursue this direction.

(ii) By functoriality, the morphism $\cF(Z)\:Z_\res\to Z$ is defined even without $Z$ noetherian, since the local constructions glue.

(iii) Any normal qe scheme is formally equidimensional. So desingularization can be achieved without assuming $Z$ is formally equidimensional, by taking the composition $Z'_\res\to Z'\to Z$, where $Z'$ is the normalization of $Z$.
\end{remark}

\subsubsection{Functorial relative principalization}\label{Assumption B}
Throughout this paper, we say that a morphism $f\:X\to B$ of qe noetherian fs logarithmic DM stacks of characteristic zero is a {\em relative logarithmic orbifold}\footnote{When the morphism $X \to B$ is specified we might just call it a \emph{logarithmic orbifold}, suppressing the word \emph{relative}.} if $X \to B$ is logarithmically regular and has enough derivations, see \S\ref{logregsec}--\ref{logorbsec}. We will work with fs logarithmic structures, hence only saturated base changes $X'=(X\times_BB')^\sat$ will be considered and we may omit the superscript ``sat". See also \cite[\S1.2.1]{ATW-principalization}. In addition, we have to impose some restrictions on $B$: we assume throughout that
\begin{itemize}
\item[($\bB$)] $B$ belongs to the class $\bB$ of logarithmically regular logarithmic DM stacks.\end{itemize}

See Remark \ref{Reb:B} for implications beyond $B\in \bB$.

A Kummer ideal $\cJ\subseteq\cO_{X\ket}$ is called {\em submonomial} if it is the sum of a suborbifold ideal and a Kummer monomial ideal, see \S\ref{kumidsec}, and the submonomial Kummer blow up along $\cJ$ is a stacky Proj of the corresponding Rees algebra $h\:X'=\cProj_X(\oplus_n\pi_*(\cI^n))\to X$, where $\pi\:X\ket\to X\et$ is the natural morphism, see \S\ref{submonomblowsec}. Note that $h^{-1}\cJ$ is an invertible ideal. By a {\em principalization method} we mean a rule $\cF$ which obtains a relative logarithmic orbifold $f\:X\to B$ with $B\in\bB$ and an ideal $\cI\subseteq\cO_X$ and outputs either $\cF(f,\cI)=\emptyset$, which means ``$\cF$ fails over the given $B$, blow it up first'', or a sequence of submonomial Kummer blowings up $\cF(f,\cI)\:X'=X_n\to\dots\to X_0=X$ such that each $X_i\to B$ is a relative logarithmic orbifold, the center of each $X_{i+1}\to X_i$ is supported on $V(\cI\cO_{X_i})$, and $\cI\cO_{X'}$ is a monomial ideal.

Our first main result states that there exists a principalization method that does not fail after a large enough modification of the base and satisfies three functoriality conditions as follows:

\begin{theorem}[Principalization]\label{Proj:principalization}
There exists a principalization method in characteristic zero $\cF$ satisfying the following properties:

(i) Existence: let $f\:X\to B$ be a relative logarithmic orbifold with an ideal $\cI\subseteq\cO_X$, and assume that $B\in\bB$ and either $\dim(B)\le 1$ or $f$ has abundance of derivations. Then there exists a blow up $g\:B'\to B$ with saturated pullback $f'\:X'\to B'$ and $\cI'=\cI\cO_{X'}$ such that $\cF(f',\cI')\neq\emptyset$. Moreover, if $B$ is a scheme, then one can also achieve that the center of $g$ is monomial away from the closure $\overline{f(V(\cI))}\subset B$.

(ii) Base change functoriality: if $\cF(f,\cI)\neq\emptyset$ and $B'\to B$ is any morphism of logarithmic stacks from $\bB$ with noetherian saturated base change $f'\:X'\to B'$ and $\cI'=\cI\cO_{X'}$, then the sequence $\cF(f',\cI')$ is obtained from the saturated pullback sequence $\cF(f,\cI)\times_BB'$ by removing Kummer blowings up with empty centers.

(iii) Logarithmically regular functoriality: if $\cF(f,\cI)\neq\emptyset$ and $\cI'=\cI\cO_{X'}$ for a logarithmically regular morphism $X'\to X$ such that $X'\to B$ is a relative logarithmic orbifold (i.e. has enough derivations), then $\cF(f',\cI')$ is obtained from the saturated pullback sequence $\cF(f,\cI)\times_XX'$ by removing Kummer blowings up with empty centers.

(iv) Compatibility with closed embeddings: if $\of\:\oX\to B$ is another relative logarithmic orbifold and $i\:X\into\oX$ is a $B$-suborbifold embedding of pure codimension, then the pushforward sequence $i_*(\cF(f,\cI))$ coincides with $\cF(\of, \ocI)$, where $\ocI :=  (i^*)^{-1}\cI$ and $i^*: \cO_{\oX} \to \cO_X$ is the reduction homomorphism.
\end{theorem}

\begin{remark}
First, a few comments on the obtained principalization.

(i) In brief, the main result is that there exists a relative principalization algorithm $\cF$, which works after a fine enough modification of $B$ assuming $B$ is reasonable and $X/B$ has abundance of derivations, for example, $X\to B$ is logarithmically smooth.

(ii) When $\dim(B)\le 1$ the abundance condition is not needed. In particular, the algorithm automatically succeeds if $B=\Spec(R)$ for a field or a DVR $R$ with the logarithmic structure $R\setminus\{0\}$, because the logarithmic scheme $B$ has no non-trivial logarithmic modifications.

(iii) Once $B$ is fine enough, our method is constructive and applies beyond algebraic geometry. However, the result on existence of the modification $B' \to B$ is purely existential and might not apply as broadly, see \S\ref{othersec}. This is the main reason why our formulation separates the canonical construction of $X'\to X$ from the non-canonical refinement $B'\to B$.

(iv) Ideally, one would like $g$ in \ref{Proj:principalization}(i) to be $T$-supported for $T=\overline{f(V(\cI))}$. At least when $B$ is a scheme we achieve just a bit less -- $g$ is a logarithmic blow up over $U=X\setminus T$. In particular, the triviality locus $U_\tr$ of the logarithmic structure is kept unchanged.
\end{remark}

\begin{remark}
Next, let us compare the relative principalization to its predecessor from \cite{ATW-principalization}.

(i) The functoriality property (iii) is the analogue of the functoriality in \cite{ATW-principalization}. Property (iv) strengthens the re-embedding principle in \cite{ATW-principalization}. Finally, property (ii) only makes sense in the relative situation, and it is necessary since the algorithm often fails without base change.

(ii) The algorithm in \cite{ATW-principalization} is the particular case of the relative principalization when $B=\Spec(k)$ and $f$ is of finite type. The only new feature of the algorithm is that when $\dim(B)>1$ it may fail in the initial cleaning stage (see \S\ref{relalgsec}), which means that a modification of $B$ is needed. The justification of the algorithm in this paper differs in one aspect: independence of the choice of a maximal contact is proven using equivalence of marked ideals. In particular, we prove a slightly stronger claim, see Theorem~\ref{mainth}(iv).
\end{remark}

\begin{remark}\label{Reb:B}
Finally, let us discuss what happens if $B\notin\bB$.

(i) In claim (i) of the Theorem one can take any qe $B$ and make it logarithmically regular by Theorem~\ref{absdesingth}. The only weakening of the formulation in this case is that the modification locus will also contain $B_\sing$.

(ii)  So far, our definition of submonomial Kummer blow ups is developed only for a logarithmically regular $B$. Probably, this assumption can be eliminated, but in this paper the notion of principalization does not make sense otherwise. In addition, logarithmic regularity of $B$ is used a couple of times in the proofs. On the other hand, we definitely know that the class $\bB$ is not the widest possible, see Remark~\ref{Brem}.
\end{remark}

\subsubsection{Order reduction}
As in other cases, the principalization theorem is a particular case of an order reduction theorem for marked ideals on $X$. The formulation is similar to Theorem \ref{Proj:principalization} and will be given in \ref{mainth}. In a sense, this is the main theorem of the paper, since principalization and desingularization are its corollaries.

\subsubsection{Functorial relative desingularization}
Similarly to the absolute case, the main application of relative principalization is the following relative desingularization theorem that was the main motivation for our project started at \cite{ATW-principalization}.

In \S\ref{desingmethodsec} we introduce a class of morphisms locally embeddable into relative logarithmic orbifolds. In particular, this class includes all morphisms of finite type. By a {\em relative desingularization method} we mean a rule as follows:
\begin{itemize}
\item {\sc Its input}  is a generically logarithmically regular morphism $g\:Z\to B$ of logarithmic DM stacks such that $Z$ is formally equidimensional, $B\in\bB$ and $g$ is locally embeddable;
\item {\sc Its ouput}  is
  either the empty value (``sorry, please blow up $B$ first"), or a stack theoretic modification $Z_\res\to Z$ such that $g_\res\:Z_\res\to B$ is logarithmically regular, $Z_\res\to Z$ is a composition of stack-theoretic $\Proj$ constructions (in particular, $I_{Z_\res/Z}$ is finite and diagonalizable), and $I_{Z_\res/Z}$ acts trivially on the stalks of $\oM_{Z_\res}$.
  \end{itemize}
  The following theorem will be deduced from Theorem \ref{Proj:principalization} in \S\ref{desingsec}.

\begin{theorem}[Relative desingularization]\label{Proj:toroidalization}
There exists a relative desingularization method $\cR$ in characteristic zero such that

(i) Existence: assume that $g\:Z\to B$ is a generically logarithmically regular morphism of qe noetherian logarithmic DM stacks of characteristic 0, $B\in\bB$, $Z$ is formally equidimensional, and $g$ is embeddable into a logarithmic orbifold $f\:X\to B$, so that either $\dim(B)\le 1$ or $f$ has abundance of derivations. Then there exists a blow up $h\:B'\to B$ with saturated base change $g'\:Z'\to B'$ such that $\cR(g')$ does not fail. Moreover, if $B$ is a scheme, then one can also achieve that the center of $h$ is monomial over any open $U\subseteq B$ such that $Z\times_BU\to U$ is logarithmically regular.

(ii) Base change functoriality: if $\cR(g)$ does not fail and $B'\to B$ is a morphism of logarithmic stacks from $\bB$ with noetherian saturated base change $g'\:Z'\to B'$, then $\cR(g')$ is the saturated pullback of $\cR(g)$.

(iii) Logarithmically regular functoriality: if $\cR(g)$ does not fail and $Z'\to Z$ is a logarithmically regular morphism such that $g'\:Z'\to B$ is locally embeddable into a relative logarithmic orbifold, then $\cR(g')$ is the saturated pullback of $\cR(g)$.
\end{theorem}

\subsubsection{Destackification}
Starting with a scheme, for some applications it is important  to obtain a scheme-theoretic desingularization. As in \cite{ATW-principalization}, it can be obtained from the stack-theoretic desingularization by applying a destackification functor. Moreover, almost all nice properties of the desingularization are preserved, but logarithmic functoriality is weakened. Unfortunately, properties of the destackification functor were studied in \cite[Secion~4]{ATW-destackification} only in the absolute setting. In appendix~\ref{destackapp} we briefly indicate how they are extended to the relative case.

By a {\em representable relative desingularization method} we mean a construction assigning to a locally embeddable morphism of logarithmic schemes $Z\to B$ a blow up $h\:Z_\res\to Z$ such that $g_\res\:Z_\res\to B$ is logarithmically regular and the center of $h$ is disjoint from the logarithmic regularity locus of $Z\to B$.

\begin{theorem}\label{Proj:schemedesing}
Composing $\cR$ from Theorem~\ref{Proj:toroidalization} with destackification and taking coarse moduli space, one obtains a representable relative desingularization method $\cR'$ in characteristic zero which satisfies properties (i) and (ii) of Theorem~\ref{Proj:toroidalization}, while (iii) is only satisfied for logarithmically regular quasi-saturated morphisms $Y\to Z$.
\end{theorem}

Again, see  \cite[Section~8.2.2]{ATW-principalization} for the notion of \emph{quasi-saturated morphisms}.

\begin{proof}
Assume that $g\:Z\to B$ is locally embeddable into {a qe relative} logarithmic orbifold and $\cR(g)\:Z'\to Z$ is its desingularization. Let $T$ denote the complement of the locus on which $Z\to B$ is logarithmically regular. Let $Z''=\cD_Z(Z')$ be the relative destackification of $Z'$ over $Z$ (see \S\ref{destackapp}). Then $Z''_{\cs/Z}\to Z$ is the composition of $T$-supported blow ups $Z''_{\cs/Z}\to Z'_{cs/Z}\to Z$, hence is a $T$-supported blow up. In addition, $Z''_{\cs/Z}\to B$ is logarithmically regular and the destackification is compatible with base change $B' \to B$ by Theorem~\ref{destackth1}. Functoriality of the construction with respect to quasi-saturated morphisms $Y\to Z$ is checked in the same way as in the proof of \cite[Theorem~8.3.4]{ATW-principalization}.
\end{proof}

\subsubsection{Weak semistable reduction}\label{Sec:weak-semistable} In \cite{Molcho} Molcho shows that, assuming $Z\to B$ is a proper surjective logarithmically smooth morphism of toroidal embeddings, there is a universal stack theoretic toroidal modification $\cB \to B$ with toroidal pullback $\cZ \to \cB$ which is \emph{weakly semistable}.  Molcho's work is a refinement of \cite{AK, ADK}. In his treatment a logarithmically smooth morphism is \emph{weakly semistable} if it is an integral and saturated morphism of logarithmic structures - this comes down to requiring the morphism to be flat with reduced fibers. This notion is clearly stable under logarithmic pullbacks. It is more flexible than the terminology of \cite{AK, ADK} which also requires the base to be regular. Molcho's stack theoretic modification $\cB \to B$ is naturally the composition of a representable modification and a generalized root stack.

We note that Molcho's  toroidal modification  $\cB \to B$ is necessarily functorial, in the following sense. Assume $B' \to B$ is a logarithmic morphism of toroidal embeddings, let $Z' \to B'$ be the pullback of $Z \to B$ in the fs category, so that $Z'$ is a toroidal embedding and $Z' \to B' $ logarithmically smooth. Molcho provides a universal $\cB' \to B'$, and we claim that $\cB' = \cB \times_B B'$, the pullback in the fs category.

 Indeed let $T \to B'$ is a dominant morphism of toroidal embeddings  such that $Z_T \to T$ is weakly semistable. Then the composite morphism $T \to B$ factors uniquely through $\cB \to \cB$ by its universal property. Hence $T \to B'$ factors uniquely through  the fs pullback $\cB \times_B B'$, as needed.

A similar argument shows that Molcho's modification   $\cB \to B$ is functorial for proper, surjective, logarithmically smooth $Z' \to Z$: if we denote by $\cZ' \to \cB'$ the modification associated to $Z' \to Z$ then  $Z \times_B \cB' \to \cB'$ is necessarily integral and saturated, since it is dominated by $\cZ'$. Thus $\cB' \to B$ factors uniquely through $\cB$. If, moreover, $Z'\to Z$ is smooth and strict, then in fact $\cB' = \cB$.

Molcho's work implies that, if we allow the base to inherit a stack structure and insist on the morphism $Z \to B$ to be a proper morphism of varieties, the results of our earlier theorems can be made weakly semistable, with functoriality properties preserved.

We note that, in another manuscript, we show that when $X \to B$ in Theorem \ref{Proj:principalization} is proper, the modification $B' \to B$ in part (i) of the theorem can be made functorial as well. Together with Molcho's result this allows weak semistable reduction of proper morphisms of varieties to be functorial also for base change. It would be interesting to understand the extent to which functoriality can be retained when $Z \to B$ is not necessarily proper.

Molcho's constructions are entirely combinatorial, and certainly extend beyond varieties. Note also that \cite[Section~2.3]{semistable} achieves weak semistable reduction for general logarithmically smooth morphisms of logarithmic schemes. It only uses scheme-theoretic modifications of the base, but achieves a weaker form of functoriality.

Using Kawamata's trick, the stack theoretic modification $\cB \to B$ can be replaced by an alteration $B_1 \to B$, see  \cite[Section 5]{AK}. Unfortunately Kawamata's trick requires choices and cannot be done in a functorial way. For example, if $B=\Spec(O)$ with $O=k[t]_{(t)})$ is the localization of $\AA^1_k$ at 0 and $X=\Spec(O[x,y]/(x^2y-t))$, then $B_1\to B$ must have even ramification degree. For any $a\in k^\times$,  sending $t$ and $y$ to $t/(t-a)$ and $y/(t-a)$ defines an automorphism of $X\to B$. This automorphism translates any ramified cover of $B_1\to B$ to a non-isomorphic one.

\subsubsection{Semistable reduction}

It is proved in \cite{semistable} that by toroidal methods $g_\res$ can be even made semistable. A logarithmically regular morphism $Z \to B$ is \emph{semistable} if it is integral and saturated, and both $Z$ and $B$ are regular and logarithmically regular. In local monomial coordinates, the morphism is given by a system of monomial equations of the form

\begin{eqnarray*}
 t_1 & = & y_1 \cdots y_{k_1}\\
 \vdots & & \vdots \\
 t_\ell & = & y_{k_{\ell-1} + 1}  \cdots y_{k_\ell},
 \end{eqnarray*}
  in other words it is, locally, a product of $\ell$ one-parameter semistable families. Combining the main theorem of \cite{semistable} with Molcho's formalism\S\ref{Sec:weak-semistable} and Theorem \ref{Proj:principalization} we obtain:

\begin{theorem}\label{Th:semistable}
Let $g:Z \to B$ be a morphism as in Theorem~\ref{Proj:toroidalization}. There exists a logarithmic alteration $B_1 \to B$ and a logarithmic modification $Z_1 \to B_1 \times_B Z$ of the saturated pullback,  such that $Z_1 \to B_1$ is semistable. The construction is compatible with regular surjective morphisms $Z' \to Z$, in particular with arbitrary actions of groups on $Z$ over $B$.
\end{theorem}

Semistable reduction is inherently not stable under base change.  A version which is stable under base change is \emph{polystability}, a notion which goes back to de Jong and Berkovich, which is defined in local monomial coordinates by equations of the form

\begin{eqnarray*}
 m_1 & = & y_1 \cdots y_{k_1}\\
 \vdots & & \vdots \\
 m_\ell & = & y_{k_{\ell-1} + 1}  \cdots y_{k_\ell},
 \end{eqnarray*}
 with $m_i$ arbitrary monomials on the base.

 The main theorem of \cite{semistable} of course implies the existence of polystable reduction. It would be interesting to provide a functor giving polystable reduction. This must necessarily allow the base to be a stack. The combinatorial methods of \cite{semistable} suggest that this should be possible.
%
%
%
%

\subsubsection{Morphisms of finite presentation}
For morphisms $f\:X\to B$ of finite presentation one can use the noetherian approximation theory from \cite[${\rm IV}_3$, \S8]{ega} to substantially weaken assumptions on $B$. We assume the stacks are quasi-compact quasi-separated (abbreviated qcqs). By an {\em approximation} of $f$ we mean a morphism $\tilf\:\tilX\to\tilB$ such that $\tilf$ is of finite type, $\tilB$ is of finite type over $\QQ$ fitting in a base-change diagram:
$$\xymatrix{ X\ar[r]\ar[d]_f & \tilX \ar[d]^{\tilf} \\ B \ar[r] & \tilB.}$$
Approximation of a pair $(f,\cI\subseteq\cO_X)$ consists also of an ideal $\tilcI\subseteq\cO_{\tilX}$ such that $\cI=\tilcI\cO_X$. An approximation is called a {\em $\bB$-approximation} if, in addition, $\tilB\in\bB$.

If $f$ possesses a $\bB$-approximation $\tilf$ such that the desingularization of $\tilf$ succeeds, then its pullback is a desingularization of $f$. We will show in \S\ref{fpsec} that the latter is independent of the choice of the $\bB$-approximation and deduce that the desingularization algorithms extend to finitely presented morphisms $f$ which possess $\bB$-approximation \'etale-locally. A similar result will be obtained for principalizations. Combining this with the standard approximation technique we will obtain the following result. Here in part (2) we say that a morphism $g\:Z\to B$ of qcqs logarithmic DM stacks is {\em nice} if it is generically logarithmically smooth and any connected component of $Z$ has equidimensional generic fibers over $B$.

\begin{theorem}\label{finpres}
In the case of finitely presented morphisms Theorems \ref{Proj:principalization}, \ref{Proj:toroidalization} and \ref{Proj:schemedesing} extend to non-noetherian qcqs DM logarithmic stacks of characteristic zero as follows:

(1) The principalization algorithm of Theorem \ref{Proj:principalization} extends to a principalization $\cF$ of ideals $\cI$ on logarithmically smooth morphisms $f\:X\to B$ such that the pair possesses a $\bB$-approximation \'etale-locally on $X$ and $B$. It satisfies properties (ii), (iii) and (iv) for arbitrary base changes $B'\to B$ and log smooth morphisms $X'\to X$ and $\oX\to B$. In addition, for any log smooth $f\:X\to B$ and $\cI\subseteq\cO_X$ with integral qcqs $B$ whose logarithmic structure is saturated quasi-coherent and generically trivial, there exists a blow up $B'\to B$ such that the saturated base change $f'$ of $f$ and $\cI'=\cI\cO_{X'}$ possess a $\bB$-approximation \'etale locally and $\cF(f',\cI')$ succeeds.

(2) The desingularization algorithm of Theorems \ref{Proj:toroidalization} and \ref{Proj:schemedesing} extend to nice morphisms $g\:Z\to B$ which possesses a $\bB$-approximation \'etale-locally on $Z$ and $B$. They satisfy properties (ii) and (iii) for arbitrary base changes $B'\to B$ and log smooth morphisms $Z'\to Z$. In addition, for any nice $g\:Z\to B$ with an integral stack $B$ whose logarithmic structure is quasi-coherent, saturated and generically trivial, there exists a blow up $B'\to B$ such that the saturated base change $g'$ of $g$ possesses a $\bB$-approximation \'etale locally and $\cR(g')$ and $\cR'(g')$ succeed.
\end{theorem}

In particular, this covers the case when $B$ is the spectrum of a valuation ring $R$ and $M_B=R\setminus \{0\}$.

\subsubsection{Logarithmically smooth compactification theorem}
As another application we obtain the following compactification result.

\begin{theorem}\label{Th:compactification}
Let $f\:X\to B$ be a separated logarithmically smooth morphism of qcqs logarithmic schemes. Assume that $B\in\bB$ and the logarithmic structures are Zariski. Then there exists a logarithmic blow up $g\:B' \to B$, a proper logarithmically smooth morphism $\oX\to B'$ and an open immersion $X_{B'}\into\oX$ of logarithmic schemes over $B'$. In addition, if $U\subset B$ is an open subscheme such that $X_U\to U$ is proper, then one can achieve that the center of $g$ is monomial on $U$
\end{theorem}
\begin{proof}
First, by Nagata compactification theorem (e.g. see \cite{ConradNagata}) there exists a compactification of $f$ to a proper morphism of schemes $\of_0\:\oX_0\to B$. Find a finite covering $X=\cup_iX_i$ whose elements possess charts $X_i\to\Spec(\QQ[P_i])$ and let $\{p_{ij}\}$ be finite sets of generators of $P_i$. Then the log structure on $X_i$ is given by the Cartier divisors $D_{ij}=V(u^{p_{ij}})$ and blowing up the schematic closures $\oD_{ij}$ of $D_{ij}$ in $\oX_0$ we obtain a finer compactification of $X$, still denoted $\oX_0$ for shortness, such that each $\oD_{ij}$ is a Cartier divisor. Then the divisors $\oD_{ij}$ define a fine log structure on $\oX_0$ which extends the log structure on $X$, hence we obtain a logarithmic compactification $\of_0\:\oX_0\to B$ of $f$. Replacing $\oX_0$ by its saturation we can assume that $\oX_0$ is fs, and then by Theorem~\ref{Proj:schemedesing} there exists a relative desingularization $\oX\to B'$ of $\oX_0 \to B$. Since $\of_0$ is logarithmically smooth on $X$ and over $U$, the center of $B'\to B$ is monomial on $U$ and $\oX\to\oX_0\times_BB'$ is an isomorphism over $X_{B'}$. Thus, $\of\:\oX\to B'$ is as required.
\end{proof}

Examples such as the  Whitney umbrella punctured at the origin show that one cannot remove the assumption that the logarithmic structures are Zariski in Theorem~\ref{Th:compactification}.

\subsubsection{Extension to other geometries}\label{othersec}
Part of our motivation stems from profound questions arising in non-archimedean geometry, see \cite[Page 364]{BLR}. Being functorial for regular morphisms, our relative principalization and desingularization methods extend to these categories immediately, and existence would be equivalent to extending the monomialization theorem to these contexts. Remark~\ref{intromonomrem} below shows that at the very least, one should either restrict the class of objects one resolves or allow non-proper base changes. This is a separate question not related much to the content of this paper, so we will not study it here. On the positive side, we expect that our main results extend to other categories verbatim when the base is at most one-dimensional, and this is explained in \S\ref{lastsec}.

\subsection{Methods}\label{methodssec}
Now, let us describe our methods and discuss certain choices we make and possible alternatives.

\subsubsection{Ambient spaces}
We resolve morphisms $Z\to B$ by embedding them into logarithmically regular morphisms $X\to B$ and principializing $\cI_Z$ on $X/B$. Similarly to \cite{ATW-principalization}, principalization is done by a sequence of non-representable modifications $\sigma\:X'\to X$ with finite diagonalizable inertia $I_{X'/X}$, which are called {\em submonomial Kummer blow ups} in this paper. This forces us to work with stacks.

\begin{remark}
(i) Similarly to \cite{ATW-principalization} one could consider only stacks with finite diagonalizable inertia, but we decided to extend the generality to logarithmic DM stacks because this does not change a single argument.

(ii) Modifications of this type were called Kummer blow ups in \cite{ATW-principalization}. We prefer to change the terminology because one can, more generally, construct Kummer blow ups along arbitrary centers using the same stack theoretic $\cProj$ formula as in \S\ref{submonomblowsec}.
\end{remark}

\subsubsection{Derivations}
As in \cite{ATW-principalization}, all constructions of our method are expressed in terms of sheaves of logarithmic derivations and differential operators of order at most $a$, but this time we use the relative sheaves $\cD^{(\le a)}_{X/B}$. In particular, the {\em logarithmic order} $a$ of $\cI$ is the minimal number such that $\cD^{(\le a)}_{X/B}(\cI)=\cO_X$, the ideal $\cI$ is {\em clean} if $a<\infty$, and the {\em (hypersurface of) maximal contact} to $\cI$ is $H=V(t)$, where $t$ is an element of logarithmic order one in $\cD^{(\le a-1)}_{X/B}(\cI)$.

\subsubsection{Marked ideals}
Order reduction operates with marked (or weighted) ideals denoted $\ucI=(\cI,a)$ throughout the paper. We define products, equivalence and domination as usual, but replace sums with the homogenized sums, see \S\ref{basicfactssec} and Remark~\ref{sumrem}. Also, we use the {\em homogenized coefficient ideals} $\ucC=\sum_{i=0}^{a-1}\cD_{X/B}^{(\le i)}\ucI$ defined by use of homogenized sums. In fact, this is precisely the logarithmic analogue of Kollar's tuning ideal $W_{a!}(\ucI)$, see \cite{Kollar}.

\subsubsection{Modules of derivations}
Certain submodules of $\cD_X$ suffice to run various parts of the principalization algorithms, for example, see \cite[Exercise~4.4]{Bierstone-Milman-funct}. This was essentially used in \cite{BMT} to show that the algorithm of Bierstone-Milman depends only on the (huge) module of absolute derivations and hence resolutions of varieties over different fields are compatible. Submodules of $\cD_X$ were also used in \cite{ATW-principalization}.

For an arbitrary logarithmically regular morphism the module $\cD_{X/B}$ can be very large or very small, even 0. Also it is may be not quasi-coherent and does not satisfy good functoriality properties, see \S\ref{logdifsec}. Therefore, we prefer to work with concrete $\cO_X$-submodules $\cF\subseteq\cD_{X/B}$ and their transforms. In particular, in Proposition~\ref{Flem} and Theorem~\ref{coefprop} we obtain equalities for transforms of derivations and coefficient ideals, while in the classical situation and in \cite{ATW-principalization} one only proved inclusions. This equality, which only holds for homogenized coefficient ideals, allows us to simplify the formalism of \cite{ATW-principalization} -- we do not have to consider normal closures of marked ideals anymore, see \cite[Proposition~6.1.3]{ATW-principalization} for comparison. {Note that} we still have to consider normal closures of powers of admissible centers, see Remark~\ref{norrem} and Lemma~\ref{operlem}(ii). {This technicality is avoided in \cite{ATW-weighted} by systematically  using valuative $\QQ$-ideals.}

\subsubsection{Relative logarithmic orbifolds}
At each step of the algorithm we precisely describe the conditions {that} modules of derivations $\cF$ should satisfy. We say that $\cF$ is {\em separating} (resp. {\em logarithmically separating}) if it separates all parameters (resp. {parameters} and logarithmic parameters) on $X/B$. It turns out that almost all stages of the algorithm only require that $\cF$ is separating, for example, this provides enough derivations to compute the logarithmic order and define the maximal contact. However, only the property of being logarithmically separating is stable under blow ups, so we call a logarithmically regular morphism $X\to B$ a {\em relative logarithmic orbifold} if $\cD_{X/B}$ is logarithmically separating, and we construct relative principalization for arbitrary relative logarithmic orbifolds.

\subsubsection{Base change}
Loosely speaking everything described in \S\ref{methodssec} up to now is obtained from the method of \cite{ATW-principalization} by passing to the relative and logarithmically regular setting and making some (relatively minor) improvements. The really new issue is that one has to take base changes into account. First, they provide a new form of functoriality which one should establish. This task is easy. Second, {one has to allow that the algorithm might fail without prior base-change, and to resolve this by modifying} the base first and running the algorithm after that. Here is the simplest example.

\begin{example}\label{stuckexample}
(i) We start with a ``non-trivial'' example. Take $B=\AA^2_k=\Spec(k[x,y])$ and $X=\AA^3_k=\Spec(k[x,y,t])$ with the trivial logarithmic structures, and let $\cI=(x,y,t)$. A minimal submonomial center $\cJ$ containing $P=V(\cI)$ is two-dimensional, for example $(t)$, hence there exists no $\cJ$ with support at $P$, and $\cI$ cannot be principalized. The situation can be solved by modifiying $B$. In this case, the most natural way is to blow up $(x,y)$ and enrich the logarithmic structure with the exceptional divisor, but one can even simply enlarge the logarithmic structure to $x^\NN\times y^\NN$ so that $\cI=(t)+(x,y)$ becomes submonomial.

(ii) Precisely the same example works when $B=\Spec(k[x])$, but it is less illuminating since blowing up $(x)$ and increasing the logarithmic structure to $x^\NN$ have the same effect.

(iii) In fact, even in the ``most trivial'' case of $X=B=\AA^2_k$ or $X=B=\AA^1_k$, the only way to principalize $\cI=(x,y)$ or $\cI=(x)$ is by base change.
\end{example}

\subsubsection{Monomialization}\label{intromonomsec}
For a monomial ideal $\cI$ we denote its saturation by $\cI^\sat$, see \cite[Remark 3.2.1(1)]{ATW-principalization}. Similarly to the absolute case in \cite{ATW-principalization}, our relative order reduction of $\ucI=(\cI,a)$ starts with blowing up the weighted monomial saturation $\cW(\ucI):=(\cM(\cI)^\sat)^{1/a}$, as this is the best possible attempt to get rid of the monomial part. If the transformed ideal is clean, one proceeds precisely as in the absolute case. However, the main new issue one has to deal with is that the new ideal may be not clean, and then the algorithm fails. The reason is that, unlike the absolute case, the differential saturation $\cD_{X/B}^{\infty}(\cI)$ can be strictly smaller than the monomial saturation $\cM(\cI)$. For example this is what happens in Example~\ref{stuckexample}(iii). Our main result in this direction is that if $\dim(B)\le 1$ or $X/B$ has abundance of derivations (for example, $X\to B$ is logarithmically smooth, see \S\ref{abundsec}), then the situation can be restored by a preliminary blow up of $B$, which monomializes $\cD_{X/B}^{\infty}(\cI)$. This is proved in a surprisingly difficult  monomialization theorem \ref{monomialization}. The subtlety of this result is also indicated by the fact that its naive generalizations fail:

\begin{remark}\label{intromonomrem}
(i) Monomialization does not hold for complex analytic spaces. For example, take $X\into B$ an open immersion and $\cI\subseteq\cO_X$ an ideal which cannot be extended to $\cO_B$. Then $\cI$ cannot be monomialized by modifications of $B$. Similar ``non-trivial'' examples exist when $X\to B$ factors through an open immersion. This indicates that monomialization requires some restrictions: either restrict the class of objects one resolves, for example, by imposing a properness assumption, or allow non-proper base changes, for example, open covers of $B$. We plan to study this question elsewhere.

(ii) In view of the above remark, it also makes sense to study whether Theorem \ref{monomialization} can be extended to more general relative logarithmic orbifolds by enlarging the class of base changes.
\end{remark}

\subsubsection{The relative order reduction algorithm}\label{relalgsec}
As usual, the principalization algorithm of $\cI$ makes an order reduction of $(\cI,1)$. Our order reduction algorithm for a marked ideal $\ucI=(\cI,a)$ {follows the steps} of the algorithm from \cite{ATW-principalization}. Here is an outline, see \S\ref{princsec} and \S\ref{ordersec} for details. The algorithm runs by induction on the relative dimension of $X/B$ and consists of three steps.

Step 1. {\em Initial cleaning}. Blow up the weighted monomial saturation $\cW(\ucI)$, and output the fail value if the controlled transform is not clean. This step certainly successes if $\cD_{X/B}^{\infty}(\cI)$ is monomial.

Step 2. {\em Reducing the logarithmic order of the clean part below $a$.} Throughout this step the ideal is balanced, that is, $\cI=\cM\cdot\cI^\cln$ with an invertible monomial $\cM=\cM(\cI)$ and a clean $\cI^\cln$ of order $b$. By induction, one simply performs the order reduction of the maximal order marked ideal $\ucI^\cln=(\cI^\cln,b)$, which automatically results in reducing the order of the clean part of the transform. The order reduction of $\ucI^\cln$ is constructed by finding \'etale-locally a maximal contact $H$ to $\ucI^\cln$ and descending to $X$ the order reduction of the equivalent marked ideal $(\cC(\cI^\cln)|_H,b!)$ on $H$.

Step 3. {\em Final cleaning.} At this stage the clean part is resolved, so one simply blows up $\cM(\cI)^{1/a}$.

\subsubsection{Equivalence classes}
Similarly to other principalization algorithms, the most subtle thing is to show that the process is independent of the choice of $H$ and hence descends to a Kummer blow up sequence of $X$. As in \cite{ATW-principalization}, one possible way would be to prove that $\cC(\cI^\cln)|_H$ is unique up to a formal isomorphism, where we use that our coefficient ideal is homogenized. However, we decided to use equivalence classes instead, in a manner modeled on \cite{Bierstone-Milman-funct}. This route is a bit more sophisticated, but it also proves the stronger claim that the order reduction only depends on the functorial equivalence class of $\ucI$ as defined in \S\ref{equivsec} and \S\ref{equivsubsec}. We explain in Remark~\ref{equivrem} why it is more natural to consider functorial logarithmic equivalence, and how it sheds a new light on the classical definitions of Hironaka and Bierstone-Milman. In addition, we prove in Theorem~\ref{equivth} that the functorial equivalence class determines the main invariants of $\ucI$ used in the order reduction -- the weighted logarithmic order and the weighted monomial part. This is a logarithmic version of \emph{Hironaka's trick,} and the proof is more straightforward than in the classical case, see \S\ref{modelsec}.

\subsubsection{Induction on length}
All proofs in this paper are designed so that claims about blow up sequences are proved by simple and formal induction on the length of the sequence. In this way we achieve that any coordinate-dependent check is done for a single Kummer blow up, and there is no need to consider charts of sequences, which might be unpleasant due to non-representability of Kummer blow ups. This becomes possible because we use transforms of derivations in Theorem~\ref{coefprop}. In particular, we manage to restrict the use of explicit Taylor series (or another form of a Weierstrass preparation) only to the proof of Proposition~\ref{liftlem}, which deals with lift of admissibility of a single Kummer center and not a whole blow up sequence.

\subsubsection{The relative desingularization method}
The relative desingularization is obtained by embedding into relative logarithmic orbifolds and principalizing ideals there, and the main issue is to prove that there is a unique minimal embedding. In \cite{ATW-principalization} we did this in \'etale topology. In this paper we use formal topology instead. This is possible because principalization was constructed for arbitrary logarithmically regular morphisms, and this even leads to simpler proofs because uniqueness of minimal formal embedding is less technically demanding.


\subsection{Conventiones}\label{convsec}
Unless said to the contrary, all schemes and stacks we consider are noetherian of characteristic zero. In particular, we will only consider {\em noetherian base changes} of a morphism $f\:Y\to Z$, that is, base changes $f'\:Y'=Y\times_ZZ'\to Z'$ with noetherian $Y'$ and $Z'$.

Unless said to the contrary, all logarithmic schemes and logarithmic stacks we consider are assumed to be fs. In particular, fiber products are taken in the fs category. All logarithmic structures are defined in the \'etale topology. In particular, it will be usually convenient to work \'etale locally and use geometric points. By convention, we only consider minimal such points $\ox\to X$, that is, $\cO_\ox$ is the strict henselization $\cO_x^{sh}$ of the local ring of the underlying Zariski point $x\in X$. Zariski points will be sometimes used when a choice of a geometric point above is not important.

\setcounter{tocdepth}{1}


\tableofcontents

\section{Relative logarithmic orbifolds}\label{basicsec}
\addtocontents{toc}
{\noindent Recollections of logarithmic structures. Logarithmical regularity. Logarithmic differential operators. Relative logarithmic orbifolds.}
Our principalization algorithm works with ideals on a {\em relative logarithmic orbifold} $f\:X\to B$ of characteristic zero. Throughout this paper this means that $X$ is a logarithmic DM stack, $f$ is logarithmically regular and the sheaf of relative logarithmic derivations $\cD_{X/B}$ is logarithmically separating. Section~\ref{basicsec} is devoted to defining and studying these properties.

\subsection{Charts}\label{chafrtssec}
In this subsection we fix notation and recall basic facts. Especially important will be the notion of neat charts.

\subsubsection{Monoids}
Unless the opposite is clear from the context, all monoids are assumed to be fs, and by $\rk_\QQ(P)$ we denote the number $\dim_\QQ(P^\gp\otimes\QQ)$.

\subsubsection{Charts of logarithmic schemes}\label{Sec:charts}
By $\bfA$ we denote the natural functor from monoids to affine schemes: $\bfA_P=\Spec(\QQ[P])$. If $P$ is a group, we will also use the notation $\bfD_P=\bfA_P$ to stress that it is a diagonalizable group, Cartier-dual to $P$.

We will use conventiones of \cite[II.2.3.1]{Ogus-logbook}: A chart $Z\to\bfA_P$ is {\em exact} at a geometric point $z\to Z$ if $\oP=\oM_z$, and this happens if and only if the homomorphism $P\to M_z$ is local. Thus, the chart is exact at $z$ if and only if $z$ is mapped to the closed logarithmic stratum of the chart. Moreover, a chart is {\em neat} if it is sharp (\S\ref{Sec:sharp-chart} below) and exact, that is, $P=\oM_z$. Neat charts always exist; in a sense, they are the minimal charts at $z$.

\begin{example}
Take $P=\NN$ and let $Z=\bfA_P$, $Z'=\bfA_{P^\gp}$ and $z\in Z'$ be any point. Then the tautological chart $Z=\bfA_P$ is not exact at $z$, the chart $Z'=\bfA_{P^\gp}$ is exact but not neat at $z$, and the trivial chart $Z_0\to\bfA_0=\Spec(\QQ)$ is neat at $z$.
\end{example}

\subsubsection{Base change of monoids}
Assume that $A$ is a ring, $P$ is a monoid and $u\:P\to(A,\cdot)$ a homomorphism, which will be written exponentially: $p\mapsto u^p$. For any homomorphism of monoids $P\to Q$ we will use the notation $A_P[Q]=A\otimes_{\ZZ[P]}\ZZ[Q]$. If $(A,m)$ is local and $Q$ is sharp, then the completion of $A_P[Q]$ with respect to the maximal ideal generated by $m$ and $Q_+$ will be denoted $\hatA_P\llbracket Q\rrbracket$.

\subsubsection{Relative toric schemes}
If $Z\to\bfA_P$ is a chart and $P\to Q$ is a homomorphism of monoids, we let $Z_P[Q]=Z\times_{\bfA_P}\bfA_Q$ denote the associated relative toric scheme.

\begin{remark}\label{strictfactorrem}
Any chart $Z\to\bfA_P$, $Y\to\bfA_Q$, $\phi\:P\to Q$ of a morphism $f\:Y\to Z$ factors $f$ into a composition $Y\stackrel{h}\to Z_P[Q]\to Z$ of a strict morphism $h$ and a relative toric scheme, the pullback of the toric morphism $\bfA_Q\to\bfA_P$. For brevity, we will usually encode a chart by writing down $Y\to Z_P[Q]$.
\end{remark}

\subsubsection{Relative characteristic}
Given a geometric point $y\to Y$ with $z=f(y)$ we will use notation $\oM^\gp_{y/z}=\Coker(\oM^\gp_z\to\oM^\gp_y)$. In fact, $\oM^\gp_{y/z}$ is the stalk at $y$ of the sheaf $M_{Y/Z}$ (see \cite[p. 187]{Ogus-logbook}), though we will not need this fact.


\subsubsection{Exact charts}
We say that a chart $Y\to Z_P[Q]$ of $f$ is {\em exact at} $y$ if both $Y\to\bfA_Q$ and $Z\to\bfA_P$ are exact. We say that a chart is {\em exact} if it is exact at some geometric point $y\to Y$. This happens if and only if $\oP=\Gamma(\oM_Z)$ and $\oQ=\Gamma(\oM_Y)$. In fact, any chart can be made exact by appropriate localizations of $P$ and $Q$: simply invert all elements that become units in $M_Z$ and $M_Y$.

\subsubsection{Sharp charts}\label{Sec:sharp-chart}
We say that $f$ is {\em sharp} at $y$ if the homomorphism $\phi_y\:\oM_{z}\to\oM_y$ is injective. If this is the case, then a chart of $f$ is called {\em sharp at} $y$ if it is modeled on the homomorphism $\phi_y\:\oM_{z}\to\oM_y$, where $z=f(y)$. We will see below that any sharp morphism possesses a sharp chart.

The quintessential logarithmically smooth morphism which is not sharp is the blowing up of the affine plane at the origin, with the toric logarithmic structure.

\subsubsection{Neat chart}
We say that a chart $Y\to Z_P[Q]$ is {\em neat at} $y$ if $P^\gp\into Q^\gp$ and $Q^\gp/P^\gp=\oM^\gp_{y/z}$. This notion was introduced by Kato and studied in detail in \cite{Ogus-logbook}.

\begin{remark}
(i) The meaning of the definition is that once the chart $Z\to\bfA_P$ is fixed, one chooses a minimal chart for $Y$ such that $P\into Q$. In particular, if $f$ is sharp, then a chart is sharp if and only if it is neat and $P$ is sharp.

(ii) In general, neat charts provide a natural generalization of the notion of sharp charts to the case of non-sharp morphisms. In order to work with neat charts, one has to consider a non-sharp $Q$ even when $P$ is sharp. Furthermore, in order to work with compositions one should also study the case where $P$ itself is not sharp.
\end{remark}

Working with neat charts often simplifies arguments, and fortunately they do exist in the case of interest. By \cite[Theorem~III.1.2.7]{Ogus-logbook} we have:

\begin{theorem}\label{neatth}
Assume that $f\:Y\to Z$ is a morphism of logarithmic schemes of characteristic zero,\footnote{This is a point where characteristic 0 is used for a technical purpose. We believe that the only characteristic 0 assumpition that cannot be circumvented in this paper is the use of maximal contact hypersurfaces --- the usual obstruction for procedures following \cite{Hironaka}.} $y\to Y$ is a geometric point, and $z=f(y)$. Then locally at $y$ and $z$ there exists a neat chart of $f$. Moreover, any local chart of $Z$ at $z$ can extended to a local neat chart of $f$ at $y$.
\end{theorem}

\subsubsection{Composition of neat charts}\label{compossec}
Let $g\:X\to Y$ and $f\:Y\to Z$ be two morphisms of logarithmic schemes, $h=f\circ g$ the composition, $x\to X$ a geometric point, $y=g(x)$, $z=h(x)$. Any pair of charts $\alpha\:X\to Y_Q[R]$ and $\beta\:Y\to Z_P[Q]$ induce a chart $\gamma\:X\to Z_P[R]$ composed from $\alpha$ and the base change $\beta_Q[R]$ of $\beta$:
$$
\xymatrix{
X\ar[r]^\alpha\ar[dr] & Y_Q[R]\ar[r]^{\beta_Q[R]}{\ar[d]\ar@{}[dr] | {\Box}} & Z_P[R]\ar[d]\\
 & Y \ar[dr]\ar[r]^\beta & Z_P[Q]\ar[d]\\
 & & Z
}
$$

Assume now that $\alpha $ and $\beta$ are neat at $x$ and $y$, respectively. If $\oM_{Y,y}\stackrel\phi\to\oM_{X,x}$ is injective, then $\gamma$ is neat at $x$ by \cite[Corollary II.2.4.7]{Ogus-logbook}. Here is a simple case, when $\gamma$ is not neat.

\begin{example}
Take $P=0$, $Q=\NN^2$ with basis $v,w$, and $R=\NN^2$ with basis $v,w-v$ and consider the toric schemes $Z=z=\bfA_P$, $Y=\bfA_Q$ and $X=\bfA_R$ with the tautological charts. In particular, $X$ is the $v$-chart of the logarithmic blow up of $Y$ at the origin $y$. Note that $P=\oM_z$, $Q=\oM_y$, and the chart $Y\toisom Z_P[Q]$ is neat at $y$.

Let $x$ be any point of the exceptional divisor of $X$ except the origin, that is, $x\in V(v)\setminus V(\frac{w}{v})$. Then the chart $X=\bfA_R$ is not exact at $x$ as $\frac{w}{v}\in\cO^\times_x$. In fact, $R\to M_x$ factors through the local homomorphism $R_x\to M_{X,x}$, where $R_x=R[v-w]\toisom\NN\cdot v+\ZZ\cdot(v-w)$, and hence $\oM_x=\oR_x=\NN$ with basis $v$. Note that $X_x:=\bfA_{R_x}$ is a neighborhood of $x$ and its tautological chart is exact at $x$. Since $\oM_y\onto\oM_x$, both charts $X\to Y_Q[R]$ and $X_x\to Y_Q[R_x]$ are neat at $x$, and the latter is even exact. On the other hand, the composed charts $X\to Z_P[R]$ and $X_x\to Z_P[R_x]$ are not neat at $x$ because $R^\gp/P^\gp=R^\gp_x/P^\gp=\ZZ^2$, while $\oM_{x/z}^\gp=\ZZ$. 
\end{example}




\subsection{Logarithmic fibers}

\subsubsection{Relative toric stacks}
If $Z\to\bfA_P$ is a chart and $\phi\:P\to Q$ is a homomorphism of monoids, then the multiplicative group $\bfD_\phi:=\bfA_N$, where $N=\Coker(P^\gp\to Q^\gp)$, acts on the $Z$-scheme $Z_P[Q]$. Indeed, $\bfD_\phi$ is the kernel of the homomorphism $\bfD_{Q^\gp}\to \bfD_{P^\gp}$, hence it acts on $\bfA_Q$ in an $\bfA_P$-equivariant way, and by base change it also acts on $Z_P[Q]$ in a $Z$-equivariant way.

The relative toric stack is defined to be the stack theoretic quotient $\cZ_P[Q]=[Z_P[Q]/\bfD_\phi]$. It is easy to see that the $Z$-stack $\cZ_P[Q]$ depends only on the homomorphisms $\oP\to\Gamma(\oM_Z)$ and $\oP\to\oQ$, for example, use \cite[Lemma~2.2.4]{Molho-Temkin}. In particular, given a morphism $Y\to Z$, the stack $\cZ_P[Q]$ is the same for all exact charts of $f$.

\begin{remark}\label{torstackrem}
Olsson introduced in \cite{Olsson-logarithmic} a stack $\Log(Z)$ parameterizing logarithmic structures on $Z$-schemes. For a morphism of logarithmic schemes $f\:Y\to Z$, the morphism $l_f\:Y\to\Log(Z)$ corresponding to the logarithmic structure of $Y$ can be used to define for $f$ logarithmic analogues of usual properties of morphisms, such as flatness, smoothness, etc. In fact, working \'etale locally, one can assume that a global chart exists and then the morphism $Y\to\cZ_P[Q]$ can often be used as a replacement of $l_f$ because $\cZ_P[Q]\to\Log(Z)$ is \'etale. Moreover, invariance of $\cZ_P[Q]$ can be used for an alternative ad hoc approach to defining various logarithmic properties in a rather elementary way -- one works with $Z\to Z_P[Q]$ and takes the action of $\bfD_\phi$ into account.
\end{remark}

\subsubsection{Neat charts}
Now we can provide a very geometric characterization of neat charts.

\begin{lemma}\label{neatchartlem}
Assume that a morphism of logarithmic schemes $f\:Y\to Z$ possesses a chart $Y\to Z_P[Q]$ modeled on an injective homomorphism $\phi\:P\into Q$, and $y\to Y$ is a geometric point with image $z'\in Z_P[Q]$. Let $Q_y$ be obtained by inverting all elements of $Q$ whose images in $\cO_y$ are units. Then,

(i) The stabilizer of $z'$ under the action of $\bfD_\phi$ is $\bfD_L$, where $$L=Q^\gp/(Q_y^\times+P^\gp).$$

(ii) The chart is neat at $y$ if and only if $z'$ is a single orbit of $\bfD_\phi$. The subfield $k(z')\subseteq k(y)$ is the same for all neat charts.
\end{lemma}
\begin{proof}
(i) Note that $\bfD_{Q^\gp}$ acts on $Z_P[Q]$ through its action on $\bfA_Q$, and the stabilizer of $z'$ is $\bfD_{Q^\gp/Q_y^\times}$. Restricting to the action of $\bfD_\phi$, we obtain that the stabilizer is $\bfD_L$.

(ii) The chart is neat at $y$ if and only if the homomorphism $Q^\gp/P^\gp\to\oM^\gp_y/\oM^\gp_x$ is injective. Since the homomorphisms $Q^\gp\to\oM_y^\gp$ and $P^\gp\to\oM_x^\gp$ are surjective, this happens if and only if $\Ker(Q^\gp\to\oM_y^\gp)=Q_y^\times$ is contained in $P^\gp$. The latter means that $L=Q^\gp/P^\gp$ or $\bfD_L=\bfD_\phi$. By (i), this happens if and only if the stabilizer of $Z'$ is the whole $\bfD_\phi$.

The second claim follows from the fact that $z'/\bfD_\phi$ is the stacky point of $\cZ_P[Q]$, which is the image of $y$. In particular, the morphism $y\to z'/\bfD_\phi$ is independent of the chart.
\end{proof}

\subsubsection{Sharp factorization}\label{sharpfactorsec}
Recall that a homomorphism of monoids $\phi\:P\to Q$ is {\em exact} if $P$ is the preimage of $Q$ under $\phi^\gp\:P^\gp\to Q^\gp$. In general, setting $\tilP=(\phi^\gp)^{-1}(Q)$ we obtain a canonical factorization of $\phi$ into the homomorphism $\psi\:P\to\tilP$, such that $\psi^\gp$ is an isomorphism, and the exact homomorphism $\widetilde\phi\:\tilP\to Q$. In particular, if $\phi$ is injective, then the sharpening of $\widetilde\phi$ is injective.

Given a chart $Y\to Z_P[Q]$ modelled on $\phi\:P\to Q$ we obtain a canonical {\em sharp factorization} of $f$ into the composition $Y\stackrel\tilf\to\tilZ=Z_P[\tilP]\stackrel {f_0}\to Z$, where $\tilf$ is sharp (even exact) at any point of $Y$, and $f_0$ is logarithmically \'etale (even a chart of a logarithmic blow up). For any point $y\to Y$ with $z=f(y)$ we will also use the notation $\tilz=\tilf(y)$.

\begin{remark}\label{tilzrem}
(i) Sharp factorization isolates all pathologies related to non-sharp\-ness of $f$ in the logarithmically \'etale morphism $f_0$, which is often more convenient to deal with.

(ii) If $z'$ is the image of $y$ in $Z_P[Q]$, then it is easy to see that $k(\tilz)=k(z')$, and hence $k(\tilz)\subseteq k(y)$ is independent of the chart.

(iii) It may happen that the inclusion $k(z)\subseteq k(\tilz)$ is not an equality even when $\bfD_\phi$ is trivial, and no stack-theoretic issues show up. For example, let $P=\NN^2$ with basis $(v,w)$ and $Q=\NN^2$ with basis $(v,w-v)$. Then $Z:=\bfA_P$ is a plane, $Y:=\bfA_Q$ is the $v$-chart of the blow up, $\bfD_\phi$ is trivial, and $Y=Z_P[Q]=\cZ_P[Q]=\tilZ$. If $y$ is the generic point of the exceptional divisor, then $k(y)=k(\tilz)=\QQ(\frac{w}{v})$ and $k(z)=\QQ$. Note also that the logarithmic fibers --- defined below ---  of $Y\to Z$ are trivial, because they are the fibers of $Y\to Z_P[Q]$.
\end{remark}

\subsubsection{Logarithmic fibers}
Given a morphism of logarithmic schemes $f\:Y\to Z$, by its {\em logarithmic fibers} we mean connected components of the fibers of the morphism $Y\to\Log(Z)$. In particular, if $f$ possesses a chart $Y\to Z_P[Q]$, then the logarithmic fibers are also the connected components of the fibers of $Y\to\cZ_P[Q]$. Furthermore, the fibers of $Z_P[Q]\to\cZ_P[Q]$ are the orbits of $\bfD_\phi$, hence the logarithmic fibers of $f$ are the connected components of the preimages of the orbits of $\bfD_\phi$ under the morphism $Y\to Z_P[Q]$. The logarithmic fiber of a point $y$ --- geometric or Zariski --- will be denoted $S_y$.

\begin{remark}
The description with orbits can be used as an alternative definition of the logarithmic fibers, as independence of choices can be checked by comparing the charts via mutual refinements. This approach does not need to use stacks at all.
\end{remark}

\subsubsection{Another description of logarithmic fibers.}
Combining the above description of logarithmic fibers with Lemma~\ref{neatchartlem}(ii), we obtain the following result, which will be our main tool when working with logarithmic fibers.

\begin{lemma}\label{logfibchartlem}
Assume that $f\:Y\to Z$ is a morphism of logarithmic schemes and $Y\stackrel h\to Z_P[Q]$ is a chart which is neat at $y\in Y$. Then locally at $y$ the logarithmic fiber $S_y$ is the connected component of $y$ in the corresponding fiber of $h$.
\end{lemma}

Logarithmic fibers also possess the following geometric interpretation. It is obtained similarly but will not be used in the paper.

\begin{remark}
(i) The logarithmic fiber $S_y$ is the logarithmic stratum of $y$ in the fiber of $\tilf\:Y\to\tilZ$ over $\tilz$. If $f$ is sharp at $y$, then $S_y$ is a logarithmic stratum of the fiber of $f$ itself.

(ii) Note also that $\tilf$ restricts to regular morphisms between logarithmic strata of $Y$ and $\tilZ$, and the logarithmic fibers of $f$ are the fibers of these morphisms between logarithmic strata. In particular, from this description it is obvious that the logarithmic fibers are regular.
\end{remark}

\subsubsection{Functoriality}
We will also need the following basic properties of logarithmic fibers.

\begin{lemma}\label{strictlogfib}
Let $g\:X\to Y$ be a strict morphism. Then for each logarithmic fiber $C$ of $f\:Y\to Z$, the connected components of its pullback $C\times_YX$ are the logarithmic fibers of $h=f\circ g$.
\end{lemma}
\begin{proof}
The morphism $X\to\Log(Z)$ is the composition of $g$ and $Y\to\Log(Z)$ because $g$ is strict. The lemma follows easily.
\end{proof}

\begin{lemma}\label{fibchangelem}
Assume that $f\:Y\to Z$ and $g\:Z'\to Z$ are morphisms of logarithmic schemes and $f'\:Y'\to Z'$ is the base change. Then each logarithmic fiber of $f'$ is an open subscheme of a ground field extension of a logarithmic fiber of $f$.
\end{lemma}

\begin{proof}
Starting with the strict morphism $Y\to\Log(Z)$ and changing the base from $Z$ to $Z'$ we obtain a strict morphism $\alpha\:Y'\to\Log(Z)\times_ZZ'$ of logarithmic stacks over $Z'$. Since $\alpha$ is strict, the tautological morphisms $\beta\:Y'\to\Log(Z')$ and $\gamma\:\Log(Z)\times_ZZ'\to\Log(Z')$ are compatible: $\beta=\gamma\circ\alpha$. By Lemma \ref{logmaplem} below $\gamma$ is \'etale and hence the log fibers of $f'$, which are the fibers of $\beta$, are open in the fibers of $\alpha$, which are ground field extensions of logarithmic fibers of $f$.
\end{proof}

\begin{lemma}\label{logmaplem}
Let $f\:X\to Y$ be a morphism of logarithmic schemes and let $\gamma\:\Log(Y)\times_YX\to\Log(X)$ be the tautological morphism corresponding to the saturated base change $\Log(Y)\times_YX\to X$. Then $\gamma$ is \'etale.
\end{lemma}
\begin{proof}
The claim is \'etale-local hence it suffices to consider the case when $f$ possesses a global chart $X\to Y_P[Q]$. Then $\Log(Y)$ possesses an \'etale presentation $\coprod_{P\to P'}\cY_P[P']$, see \cite[Corollary~5.25]{Olsson-logarithmic}, and similarly for $\Log(X)$. The map $\gamma$ then lifts to the open immersion of \'etale presentations $$\coprod_{P\to P'}(\cY_P[P']\times_XY)^\sat=\coprod_{P\to P'}\cX_Q[(Q\otimes_PP')^\sat]\into\coprod_{Q\to Q'}\cX_Q[Q'].$$
\end{proof}

\begin{remark}
(i) We are grateful to Sam Molcho for the idea of proving Lemma \ref{logmaplem}. Note that it also holds for non-saturated base changes and the same argument applies with $*^\sat$ replaced by $*^\mathrm{int}$.

(ii) Lemma \ref{logmaplem} also follows from results of \cite{Olsson-LCC}. Here is a short argument provided by Martin Olsson (we consider the non-saturated case). Define $L$ to be the stack over $X$ classifying commutative squares of log structures
$$
\xymatrix{
f^*M_Y\ar[d]\ar[r]& M_X\ar[d]\\
M_2\ar[r]& M_3}
$$
By \cite[Lemma~3.12 (i)]{Olsson-LCC}
$\Log(Y)\times_YX\rightarrow L$ sending $f^*M_Y\rightarrow M_2$ to the pushout square is an open immersion. It remains to notice that $\gamma$ is the composition of this open immersion with the map $q_1$ in \cite[Example~2.9]{Olsson-LCC} base-changed to $X$.
\end{remark}

\subsection{Logarithmically regular morphisms}\label{logregsec}
The theory of logarithmically regular morphisms was developed in \cite{Molho-Temkin} as a generalization of logarithmically smooth morphisms. We  recall some relevant points.

\subsubsection{The definition}
A morphism $f\:Y\to Z$ of logarithmic schemes is called {\em logarithmically regular} if the morphism $\Log(f)\:\Log(Y)\to\Log(Z)$ is {\em regular}, that is, it is flat and has regular geometric fibers. It is proved in \cite{Molho-Temkin} that this happens if and only if $Y\to\Log(Z)$ is regular. Also it is proved there that a morphism is logarithmically smooth if and only if it is logarithmically regular and locally of finite presentation.

\subsubsection{Functoriality}
Basic properties of logarithmic regularity are proved in \cite[Corollary 5.1.3 and \S5.3]{Molho-Temkin}, but here we only recall those that we will use. Recall also that logarithmic regularity of logarithmic schemes was introduced by Kato for Zariski logarithmic structures, see \cite[Definition~2.1]{Kato-toric}, but the same definition works in general, see \cite[Definition~2.2]{Niziol}.

\begin{lemma}\label{logreglem}
(i) Logarithmic regularity is flat local on the base and \'etale local on the source: if $f\:Y\to Z$, $a\:Z'\to Z$ and $b\:Y'\to Y\times_ZZ'$ are morphisms of noetherian logarithmic schemes, $a$ is flat, $b$ is \'etale and both are strict and surjective, then $f$ is logarithmically regular if and only if $Y'\to Z'$ is logarithmically regular.

(ii) Logarithmically regular morphisms are preserved by compositions and noetherian base changes.

(iii) A logarithmic scheme $Y$ of characteristic 0 is logarithmically regular if and only if $Y \to \Spec(\QQ)$ is logarithmically regular, where  $\Spec(\QQ)$ is endowed with the trivial logarithmic structure.
\end{lemma}

\subsubsection{Parameters}
Now we will introduce the notion of local parameters of $f\:Y\to Z$ at $y$. In principle, these are the parameters of the regular morphism $Y\to\Log(Z)$. So, there are parameters corresponding to the tangent space and to the extension of the residue fields, that we call regular and constant parameters, respectively. In addition, there is a family of parameters corresponding to the stabilizer of the image of $y$, which is $\bfD_{\oM^\gp_{y/z}}$. They can be naturally interpreted as monomial parameters.

\subsubsection{Regular parameters}\label{paramsec}
Set $S=S_y$ for shortness. By a {\em family of regular parameters} of $f$ at $y$ we mean any set $t=(t_1\.t_d)\subset\cO_y$, whose image is a family of regular parameters of the regular ring $\cO_{S,y}$, that is, its image is a basis of the cotangent space $T^*_{S,y}=m_{S,y}/m^2_{S,y}$ to $S$ at $y$.

\subsubsection{Monomial parameters}\label{logparam}
By a {\em family of monomial parameters} at a geometric point $y$ we mean any set of monomials $q=(q_1\.q_n)\subset M_y\subset\cO_y$ such that its image in the vector space $\oM^\gp_{y/z}\otimes\QQ$ is a basis.

\subsubsection{Constant parameters}\label{constparam}
Recall that the subfield $k(\tilz)\subseteq k(y)$ is independent of the choice of chart at $y$. By a {\em family of constant parameters} at $y$ we mean any set $u=\{u_i\}_i\subset\cO_y$, whose image  $u(y)\subset k(y)$ is a transcendence basis over $k(\tilz)$.

\subsubsection{Full family of parameters}
The full family of parameters $(t,q,u)\subset\cO_y$ at $y$ consists of families of parameters of the three types above. Any element (resp. subset) of such a family will be called a {\em parameter at $y$} (resp. a {\em partial family} of parameters at $y$), and we will specify the type when needed.

\subsubsection{Relative logarithmic ranks}\label{dimsec}
The sizes of families of parameters of $f$ at $y$ are invariants that we denote
$$
r_y=|t|=\dim_y(S_y),\ \
r'_y=|u|=\trdeg(k(y)/k(\tilz)),\ \
r''_y=|q|=\rk_\QQ(\oM^\gp_{y/z}).$$
The number $\logdim_y(Y/Z):=r_y+r''_y$ is finite, and we will use the number $\logdim(Y/Z)=\max_{y\in Y}\logdim_y(Y/Z)$ to run induction on ``dimension''. If $r'_y$ is finite, then $r_y+r'_y+r''_y$ is another reasonable invariant, which sometimes behaves better. For example, if $f$ is logarithmically smooth, then the morphism $Y\to\Log(Z)$ is smooth and its relative dimension at $y$ is precisely $r_y+r'_y+r''_y$.

\subsubsection{Charts}
Our work with logarithmically regular morphisms will be based on the following result proved in \cite[Theorem 5.2.5]{Molho-Temkin}.

\begin{theorem}\label{logregth}
Assume that $f\:Y\to Z$ is a morphism of logarithmic schemes with a chart $h\:Y\to Z_P[Q]$, and $y\to Y$ is a geometric point such that the chart is neat at $y$. Then $f$ is logarithmically regular at $y$ if and only if the morphism $h$ is regular at $y$.
\end{theorem}


\begin{corollary}\label{logfibcor}
If $f\:Y\to Z$ is a logarithmically regular morphism, then its logarithmic fibers are regular schemes.
\end{corollary}
\begin{proof}
We can work locally at a geometric point $y\to Y$. By Theorem~\ref{neatth} we can assume that there exists a chart $h\:Y\to Z_P[Q]$ neat at $y$. Then the logarithmic fiber $S_y$ is the fiber of $h$ by Lemma~\ref{logfibchartlem}, and hence $S_y$ is regular by Theorem~\ref{logregth}.
\end{proof}

\subsubsection{Formal description}
As another corollary, we can describe logarithmically regular morphisms on formal completions.

\begin{lemma}\label{logregchart}
Let $f\:Y\to Z$ be a logarithmically regular morphism of noetherian schemes of characteristic zero, $y\to Y$ a geometric point and $z\in Z$ its image. Assume that $f$ is sharp at $y$ and the logarithmic structure at $z$ is Zariski. Then any choice of a regular family of parameters $t_1\. t_n\in\hatcO_y$, a sharp chart $P=\oM_z\to\cO_z$, $Q=\oM_y\to\cO_y$ of $f$ at $y$ and compatible fields of coefficients $k(z)\into\hatcO_z$ and $k(y)\into\hatcO_y$ induces an isomorphism $$(\hatcO_z)_P\llbracket Q\rrbracket\llbracket t_1\. t_n\rrbracket\wtimes_{k(z)}k(y)\toisom\hatcO_y.$$
\end{lemma}
\begin{proof}
Since the logarithmic structure at $z$ is Zariski, locally at $y$ and $z$ there exists a sharp chart. It is automatically neat since $f$ is sharp at $y$. By Theorem~\ref{logregth} the homomorphism $(\cO_z)_P[Q]\to\cO_y$ is regular. Thanks to the quasi-excellence assumption its completion is also regular, for example, see \cite[Corollary~2.4.5]{Temkin-qe}. By the classical formal description of regular morphisms in characteristic zero (e.g. see \cite[Remark~2.2.12]{ATLuna}), we obtain the asserted isomorphism.
\end{proof}

\subsubsection{Submanifolds}
Given a logarithmically regular morphism $f\:Y\to Z$, by a {\em $Z$-submanifold} of $Y$ we mean any strict closed logarithmic subscheme $Y'\into Y$ such that the induced morphism $f'\:Y'\to Z$ is logarithmically regular. An ideal $I\subseteq\cO_Y$ will be called a {\em $Z$-submanifold ideal} if the strict closed logarithmic subscheme it defines is a $Z$-submanifold.

\begin{lemma}\label{submanlem}
Let $Y\to Z$ be a logarithmically regular morphism and $I\subseteq\cO_Y$ an ideal. Then $I$ is a $Z$-submanifold ideal at $y\in Y$ if and only if it is generated by a partial family of regular parameters at $y$.
\end{lemma}
\begin{proof}
If the logarithmic structures are trivial, then the morphisms are regular and the claim is classical and recalled in Lemma~\ref{regularlem}(ii)). We will deduce the lemma from this case. The first condition is \'etale local on $Y$ by Lemma \ref{logreglem}(i). By Nakayama's lemma, the second condition is equivalent to injectivity of the map $I/m_yI\to T^*_{S,y}$, where $S=S_y$. Therefore it is \'etale local too, and we can assume that $f$ possesses a chart $h\:Y\to Z_P[Q]$ neat at $y$. Since the closed immersion $Y'=V(I)\into Y$ is strict, the composition $h'\:Y'\to T\to Z_P[Q]$ is a chart of $Y'$. By Theorem~\ref{logregth}, $h$ is regular at $y$, and $I$ is a submanifold ideal if and only if $h'$ is regular at $y'$. Since locally at $y$ the logarithmic fiber $S=S_y$ is the fiber of $h$, the claim reduces to the classical case discussed in the above paragraph.
\end{proof}

\subsubsection{Increasing the logarithmic structure}
A standard tool in the theory of logarithmically regular schemes is to increase or decrease the logarithmic structure. For example, it is used in \cite[Sections 3.4, 3.5]{AT1} to study toroidal actions. We will only need a relative version of the increasing operation.

\begin{lemma}\label{incrloglem}
Assume that $f\:Y\to Z$ is a logarithmically regular morphism and $t_1\.t_l$ a partial family of regular parameters at a point $y\in Y$. Let $W$ be a neighborhood of $y$ where $t_i$ are defined and consider the logarithmic structure $M'_W$ obtained from $M_W$ by adding $t_1\.t_l$, that is, $M'_W$ is associated with the prelogarithmic structure $M_W\oplus\NN^l\to\cO_W$, where the basis elements of $\NN^l$ are sent to $t_i$. Then the morphism $(W,M'_W)\to Z$ is logarithmically regular at $y$ and $\oM'_y=\oM_y\oplus\NN^l$.
\end{lemma}
\begin{proof}
Working \'etale locally one reduces to the case, when $Y=W$ and $f$ possesses a chart $h\:Y\to Z_P[Q]$ neat at $y$. Let $Y'=(Y,M'_Y)$ denote the logarithmic scheme with the enlarged logarithmic structure. Sending the basis of the second summand of $Q'=Q\oplus\NN^l$ to $t_i$ gives rise to a chart $h\:Y'\to Z_P[Q']=Z_P[Q]\times\bfA^l$ of $f'\:Y'\to Z$ which is neat at $y$. Since $h$ is regular at $y$ by Theorem~\ref{logregth} and $t_1\. t_l$ form a partial family of regular parameters of $h$ at $y$, the morphism $h'$ is regular at $y$ by Lemma~\ref{regularlem}(i). 
By Theorem~\ref{logregth} $f'$ is logarithmically regular at $y$, as claimed.
\end{proof}



\subsection{Logarithmic derivations and differential operators}\label{logdifsec}

\subsubsection{Logarithmic derivations}
Let $f\:Y\to Z$ be a morphism of logarithmic schemes. Recall that a {\em logarithmic $\cO_Z$-derivations} on $\cO_Y$ with values in an $\cO_Y$-module $L$ consists of an $\cO_Z$-derivation $\partial\:\cO_Y\to L$ and an (additive) homomorphism $\delta\:M_Y\to L$ such that $\partial(u^q)=u^q\delta(q)$ for any monomial $u^q$, see \cite[Definition~IV.1.2.1]{Ogus-logbook}. For a $Y$-scheme $U$ let $\cD_{Y/Z}(U)$ denote the set of logarithmic $\cO_Z$-derivations $\cO_Y\to\cO_U$. This defines a presheaf of $\cO_{Y\et}$-modules, which is easily seen to be a sheaf even for the fpqc topology. We call it the sheaf of logarithmic $Z$-derivations on $Y$.

\begin{example}
Assume that the underlying scheme of $Y$ is locally integral, for example, $Y$ is logarithmically regular. Then the data $(\partial,\delta)$ is equivalent to the data of an $\cO_Z$-derivation $\partial\:\cO_Y\to\cO_U$ which preserves monomial ideals: $\partial(u^q)\in u^q\cO_U$. In particular, $\cD_{Y/Z}$ is a subsheaf of the usual sheaf of relative derivations $\Der_{Y/Z}$.
\end{example}

\subsubsection{Logarithmic differentials}
As in the classical theory, there exists a universal logarithmic derivation $d_{Y/Z}\:\cO_Y\to\Omega^{\rm log}_{Y/Z}$, whose target is the {\em logarithmic module of differentials}. In particular, $\cD_{Y/Z}=\mathcal{H}om_{\cO_Z}(\Omega^{\rm log}_{Y/Z},\cO_Y)$, yielding another explanation why it is a sheaf. As opposed to $\cD_{Y/Z}$, the module $\Omega^{\rm log}_{Y/Z}$ is always quasi-coherent and possesses good functoriality properties.

\begin{lemma}
If $f\:Y\to Z$ is a morphism of logarithmic schemes, then $\cD_{Y/Z}=\Der_{Y/\Log(Z)}$ and $\Omega^{\rm log}_{Y/Z}=\Omega_{Y/\Log(Z)}$.
\end{lemma}
\begin{proof}
It suffices to construct the first isomorphism. The claim is \'etale local, hence we can assume that there exists a global chart $h\:Y\to Z_P[Q]$ modeled on $\phi\:P\to Q$. Since $\cZ_P[Q]\to\Log(Z)$ is \'etale, $\Der_{Y/\Log(Z)}=\Der_{Y/\cZ_P[Q]}$. We have the exact sequence $$0\to \Der_{Y/Z_P[Q]}\to\Der_{Y/\cZ_P[Q]}\to\Der_{Z_P[Q]/\cZ_P[Q]}(Y)$$ associated to the morphisms of stacks $Y\to Z_P[Q]\to\cZ_P[Q]$ and
$$0\to\cD_{Y/Z_P[Q]}\to\cD_{Y/Z}\to\cD_{Z_P[Q]/Z}(Y)$$ associated to the morphisms of logarithmic schemes $Y\to Z_P[Q]\to Z$. It remains to note that $\Der_{Y/Z_P[Q]}=\cD_{Y/Z_P[Q]}$ because $h$ is strict, and both $\Der_{Z_P[Q]/\cZ_P[Q]}(Y)$ and $\cD_{Z_P[Q]/Z}(Y)$ are naturally isomorphic to $\Hom(\Coker(\phi^\gp),\cO_Y)$ (in the first case, one uses that $Z_P[Q]\to\cZ_P[Q]$ is a $\bfD_\phi$-torsor).
\end{proof}

As an immediate corollary we obtain the following result, which is also a consequence of Kato's computation of $\Omega^{\rm log}_{Y/Z}$.

\begin{corollary}\label{derfreelem}
If $f$ is logarithmically smooth, then $\Omega^{\rm log}_{Y/Z}$ and $\cD_{Y/Z}$ are locally free, and the rank at a point $y$ equals to the relative dimension of $Y\to\Log(Z)$ at $y$.
\end{corollary}

\subsubsection{Pullback}\label{difpullbacksec}
Let $g\:X\to Y$ be a morphism and $h=f\circ g$. By $g^!\cD_{Y/Z}$ we denote the restriction of the fpqc sheaf $\cD_{Y/Z}$ on $X\et$, that is, $g^!\cD_{Y/Z}(U)=\cD_{Y/Z}(U)$ for an \'etale $X$-scheme $U$. Clearly, $g^!\cD_{Y/Z}$ is an $\cO_{X\et}$-module, which can also be characterized as $\mathcal{H}om_{h^{-1}\cO_Z}(g^{-1}\Omega^{\rm log}_{Y/Z},\cO_X)$. There exists a natural homomorphism of $\cO_X$-modules $g^*\cD_{Y/Z}\to g^!\cD_{Y/Z}$, but in general $g^*\cD_{Y/Z}$ can be much smaller than $g^!\cD_{Y/Z}$, similarly to an infinite direct sum embedded into the direct product. If $f$ is logarithmically smooth, then $\Omega_{Y/Z}$ is locally free of finite rank, and hence $g^*\cD_{Y/Z}=g^!\cD_{Y/Z}$.

By a slight abuse of language, for any $\cO_Y$-submodule $\cF\subseteq\cD_{Y/Z}$ the image of $g^*\cF$ in $g^!\cD_{Y/Z}$ will be denoted by $g^*\cF$.

\begin{remark}
We had to introduce $g^!$ because it shows up in base changes and the fundamental sequence. However, distinguishing $g^!\cD_{Y/Z}$ and $g^*\cD_{Y/Z}$ plays a purely technical role and can be essentially ignored by the reader.
\end{remark}

\subsubsection{Base change}
Base change is compatible with logarithmic derivations via $g^!$.

\begin{lemma}\label{basechangelem}
Let $f\:Y\to Z$ and $g\:Z'\to Z$ be morphisms of logarithmic schemes with base changes $f'\:Y'\to Z'$ and $g'\:Y'\to Y$. Then there is a natural isomorphism $g'^!(\cD_{Y/Z})=\cD_{Y'/Z'}$.
\end{lemma}
\begin{proof}
The inverse isomorphisms are given by the restriction map and by extending logarithmic $\cO_Z$-derivations $\cO_Y\to\cO_{Y'}$ to logarithmic $\cO_{Z'}$-derivations $\cO_{Y'}\to\cO_{Y'}$ by $\cO_{Z'}$-linearity (see also \cite[Proposition~IV.1.2.3(2a)]{Ogus-logbook}).
\end{proof}

\subsubsection{The first fundamental sequence}
Logarithmic derivations and compositions are related via the following exact sequence:

\begin{lemma}\label{firstseqlem}
Given morphisms of logarithmic schemes $g\:X\to Y$ and $f\:Y\to Z$, there exist natural exact sequences $$0\to\cD_{X/Y}\to\cD_{X/Z}\stackrel\phi\to g^!(\cD_{Y/Z}),\ \ \ g^*(\Omega^{\rm log}_{Y/Z})\stackrel\psi\to\Omega^{\rm log}_{X/Z}\to\Omega^{\rm log}_{X/Y} \to 0.$$
In addition, if $g$ is logarithmically smooth, then $\psi$ is injective and $\phi$ is surjective. 
In particular, if $g$ is logarithmically \'etale then $\cD_{X/Z}=g^!(\cD_{Y/Z})$.
\end{lemma}
\begin{proof}
The left exact sequence is tautological: any logarithmic $Y$-derivation on $X$ is also a logarithmic $Z$-derivation, and any logarithmic $Z$-derivation on $X$ can be restricted to $Y$. As a corollary one obtains the first fundamental sequence of logarithmic differentials (see also \cite[Proposition~IV.2.3.1]{Ogus-logbook}). Note that the left sequence is obtained from the right one by applying the functor $\mathcal{H}om_{\cO_X}(\cdot,\cO_X)$. If $g$ is logarithmically smooth, then $\Omega^{\rm log}_{X/Y}$ is a locally free module of finite rank and $\psi$ is injective. Therefore $\phi$ is surjective in this case.
\end{proof}

\subsubsection{Differential operators}
In this subsection we essentially use that the characteristic is zero. By the algebra of {\em logarithmic $\cO_Z$-differential operators} $\cD^{\infty}_{Y/Z}$ on $\cO_Y$ we mean the algebra of $\cO_Z$-linear operators on $\cO_Y$ generated by $\cD_{Y/Z}$. It has a natural filtration by modules of operators $\cD^{(\le d)}_{Y/Z}$ of degree at most $d$. In particular, $\cD^{(\le 1)}_{Y/Z}$ is a direct sum of $\cO_Y$ and $\cD_{Y/Z}$. For any $\cO_Y$-submodule $\cF\subseteq\cD_{Y/Z}$ we denote by $\cF^{\infty}$ the $\cO_{Z}$-subalgebra of $\cD_{Y/Z}^{\infty}$ generated by $\cF$ and set $\cF^{(\le i)}=\cF^{\infty}\cap\cD_{Y/Z}^{(\le i)}$.

\begin{remark}
We will not need this, but $\cD^{\infty}_{Y/Z}$ is the sheaf of differential operators of $Y/\Log(Z)$.
\end{remark}

\subsection{Abundance of derivations}
In this section, $f\:Y\to Z$ is a logarithmically regular morphisms, $\cD=\cD_{Y/Z}$, $y\to Y$ a geometric point with image $z\in Z$, $S=S_y$ the logarithmic fiber of $y$. Our aim is to study the stalk $\cD_y$ and the ``fiber'' $\cD(y)=\cD(k(y))$. The restriction map induces a natural embedding of the usual fiber $i\:\cD\otimes k(y)\into\cD(y)$. Note that $\cD(y)$ can be viewed as the tangent space at $y$ to the fiber of $Y\to\Log(Z)$ or the relative logarithmic tangent space, though this is precise only when $f$ is of finite type and hence $i$ is an isomorphism.

\subsubsection{Formal description}\label{fsec}
We start with studying the formal picture, so assume for now that $f$ is sharp at $y$. Let $O=\hatcO_z$ and $A=\hatcO_y$ with the induced logarithmic structures, and let $t_1\.t_n$ be regular parameters at $y$, $P=\oM_z$, $Q=\oM_y$, $k=k(z)$ and $l=k(y)$. In particular, $\oM^\gp_{y/z}=Q^\gp/P^\gp$. By Lemma~\ref{logregchart} there exists an isomorphism $A=O_P\llbracket Q\rrbracket \llbracket t_1\. t_n\rrbracket\wtimes_kl$ depending on the choice of charts and fields of definition. We will work with the usual module of logarithmic derivations $\cD_{A/O}$ because any derivation $\partial$ on $A$ takes $m_A^n$ to $m_A^{n-1}$ and hence is continuous. In particular, this implies that a derivation $\partial\in\cD_{A/O}$ is determined by its action on $l$, $Q$ and $t_1\.t_n$.

Consider the following three types of derivations, whose uniqueness follows from the above and existence is an easy exercise:

\begin{itemize}
\item[(1)] Ordinary derivations: $\partial_i$ is defined by $\partial_i(t_j)=\delta_{ij}$, $\partial_i(l)=\partial(u^Q)=0$.

\item[(2)] Monomial derivations: for any $\phi\in\Hom_\ZZ(\oM^\gp_{y/z},A)$ $\partial_\phi$ is defined by $\partial_\phi(u^q)=\phi(q)u^q$, $\partial_\phi(t_j)=\partial_\phi(l)=0$.

\item[(3)] Constant derivations: any $\partial_l\in\cD_{l/k}$ uniquely extends to $\cD_{A/O}$ so that  $\partial_l(t_j)=\partial_l(u^Q)=0$.
\end{itemize}

\begin{lemma}\label{formderlem}
Keep the above notation. Then
$$\cD_{A/O}=(\cD_{l/k}\wtimes_lA)\oplus\Hom_\ZZ(\oM^\gp_{y/z},A)\oplus(\oplus_{i=1}^n A\partial_{i}).$$ In particular, if $m=\trdeg(l/k)<\infty$, then $\cD_{A/O}$ is free of rank $\rk_\QQ(\oM^\gp_{y/z})+m+n$.
\end{lemma}
\begin{proof}
Restricting derivations onto $l$, $Q$ and $t_1\.t_n$ one obtains projections of $\cD_{A/O}$ to the direct summands. The induced map to the direct sum is injective because derivations are determined by the restrictions. The surjectivity follows because any element of the direct sum lifts to the corresponding linear combination of derivations $\partial_i,\partial_\phi$ and $\partial_l$.
\end{proof}

\subsubsection{A filtration on $\cD(y)$}
The isomorphism in Lemma~\ref{formderlem} depends on choices. The only natural local structure on $\cD$ is described by the following three homomorphisms and induced filtration on $\cD(y)$

\begin{itemize}
\item[(1)] $\phi_1\:\cD(y)\to T_{S,y}=\Hom_{k(y)}(m_{S,y}/m^2_{S,y},k(y))$, where $\phi_1(\partial)(a)=\partial(a)$.
\item[(2)] $\phi_2\:\cD'(y)\to \cD_{k(y)/k(\tilz)}$, where $\cD'(y)=\Ker(\phi_1)$ and $\phi_2$ is the restriction map.
\item[(3)] $\phi_3\:\cD''(y)\to \Hom(\oM^\gp_{y/z},k(y))$, where $\cD''(y)=\Ker(\phi_2)$ and $\phi_3(\partial)(q)=\delta(q)=q^{-1}\partial(q)$.
\end{itemize}
Indeed, any $\partial\in\cD'(y)$ satisfies $\partial(m_y)=0$ because $m_y$ is generated by the maximal monomial ideal and the preimage of the maximal ideal of the logarithmic fiber, and both are taken to 0 by $\partial$. Therefore, $\partial(a)(y)$ depends only on $a(y)$ and the restriction map $\phi_2$ is well defined. Note also that we have only defined $k(\tilz)$ using charts, hence one should work \'etale locally on $Z$, but this poses no problems because replacing $k(\tilz)$ by a finite (separable) extension does not modify the module of derivations $\cD_{k(y)/k(\tilz)}$.

Similarly, $\phi_3$ is independent of the choice of a representative $q\in M^\gp_y$ in a class $\overline{q}\in\oM^\gp_{y/z}$, because for another choice $q'=qu$ with $u\in\cO_y^\times$ we have that $\delta(q')=\delta(q)+\delta(u)$, and $\delta(u)=0$ for any element $\partial\in\cD''(y)$.

\begin{lemma}\label{filtrlem}
The homomorphisms $\phi_i$ induce embeddings of the graded pieces $$\cD(y)/\cD'(y)\into T_{S,y},\ \ \cD'(y)/\cD''(y)\into\cD_{k(y)/k(\tilz)},\ \ \cD''(y)\into\Hom_\ZZ(\oM^\gp_{y/z},k(y)).$$
\end{lemma}
\begin{proof}
The first two claims are obvious, so we should only prove that $\phi_3$ is injective. It suffices to consider the case when $f$ possesses a chart and hence a sharp factorization $Y\to\tilZ\to Z$. By Lemma~\ref{firstseqlem} $\cD_{Y/\tilZ}=\cD_{Y/Z}$, hence we can replace $Z$ by $\tilZ$ and assume that $f$ is sharp at $y$. Any $\cO_z$-logarithmic derivation $\partial\:\cO_y\to k(y)$ extends to an $\hatcO_z$-logarithmic derivation $\partial\:\hatcO_y\to k(y)$ by continuity. Fix a description of $\hatcO_y$ as in Lemma~\ref{logregchart}, then $\partial$ is completely determined by its values on a field of coefficients $k(y)$, parameters $t_i$ and the monoid $Q=\oM_y$. Any element in $\Ker(\phi_3)$ vanishes on them all and is therefore zero.
\end{proof}

\subsubsection{Abundance of derivations}\label{abundsec}
If $Y\to Z$ is of finite type, then the sheaf $\cD_{Y/Z}$ is quasi-coherent, but in general $\cD_{Y/Z}$ can be very bad and may even have zero stalks and fibers: see \cite[Example~2.3.5(ii)]{Temkin-survey} for a pathological example already in the case of  trivial logarithmic structures. Our algorithm may only work when $\cD_{Y/Z}$ has large enough stalks. We formalize this below. For technical reasons which will become clearer later, we prefer to work with an arbitrary $\cO_Y$-submodule $\cF\subseteq\cD_{Y/Z}$.

The homomorphism $\cF_y\subseteq\cD_y\to\cD(y)$ induces homomorphisms $\psi_1\:\cF_y\to T_{S,y}$, $\psi_2\:\Ker(\psi_1)\to\cD_{k(y)/k(\tilz)}$ and $\psi_3\:\Ker(\psi_2)\to\Hom(\oM^\gp_{y/z},k(y))$. We say that $\cF$ is {\em separating at} $y$ if $\psi_1$ is onto, {\em logarithmically separating at} $y$ if $\psi_1$ and $\psi_3$ are onto, {\em abundant at} $y$ if in addition $\psi_2$ has a dense image with respect to the weak topology. We say that $\cF$ is {\em separating}, {\em logarithmically separating} or {\em abundant} if it is so at all points of $Y$.

\begin{remark}
(i) Informally speaking, $\cF$ is separating if it distinguishes regular parameters, it is logarithmically separating if it also distinguishes monomials, and it is abundant if it distinguishes arbitrary finite families of parameters at $y$.

(ii) We will not need this, but it is easy to see that $\cD_{Y/Z}$ has abundance of derivations at $y$ if and only if the natural map $\cD_{Y/Z,y}\to\cD_{\hatcO_y/\hatcO_z}$ has a dense image.
\end{remark}

\subsubsection{Derivations and parameters}
Given a partial family of parameters $S\subset\cO_y$, by a {\em dual family of derivations} we mean a set of derivations $\partial_a\in\cD_{Y/Z,y}$ indexed by elements of $S$ such that $\partial_a(b)=0$ if $a\neq b$, $\partial_a(a)=1$ if $a$ is not a monomial, and $\partial_a(u^a)=u^a$ if $u^a$ is a monomial parameter.

\begin{lemma}\label{standardderiv}
Let $f\:Y\to Z$ be a logarithmically regular morphism and $y\to Y$ a geometric point $z=f(y)$. Let $(t,v,q)\subset\cO_y$ be a family of parameters at $y$ and let $\cF\subseteq\cD_{Y/Z}$ be an $\cO_Y$-submodule. Then

(i) $\cF$ is separating at $y$ if and only if $t$ possesses a dual family of derivations from $\cF_y$.

(ii) $\cF$ is logarithmically separating at $y$ if and only if $\{t,q\}$ possesses a dual family of derivations from $\cF_y$.

(iii) $\cF$ is abundant at $y$ if and only if $\{t,v_0,q\}$ possesses a dual family of derivations from $\cF_y$ for any finite subset $v_0\subseteq v$.
\end{lemma}
\begin{proof}
All three claims are proved by the same argument, so we stick to (i) for simplicity of notation. Let $t=\{t_1\.t_n\}$. Since $\cF$ is separating, there exist derivations $\partial'_1\.\partial'_n\in\cF_y$ such that $\partial'_kt_j\in\delta_{kj}+m_y$. In particular, the square matrix $D$ with entries $d_{kj}=\partial'_kt_j$ lies in $\GL_n(\cO_y)$. We claim that for each $i\in\{1\.n\}$ there exist $f_1\.f_n\in\cO_y$ such that $\partial_i=\sum_{k=1}^nf_k\partial'_k$ is as required. Indeed, this is equivalent to solving the system $\sum_{k=1}^nf_kd_{kj}=\delta_{ij}$, which is possible since $D$ is invertible.
\end{proof}

\subsubsection{Relative logarithmic manifolds}\label{logmansec}
Let $f\:Y\to Z$ be a logarithmically regular morphism of logarithmic schemes. If $\cD_{Y/Z}$ is logarithmically separating, then we say that $f$ is a {\em relative logarithmic manifold} or $Y$ is a {\em logarithmic $Z$-manifold}. If, moreover, $\cD_{Y/Z}$ is abundant, then we say that $f$ has {\em abundance of derivations}. These notions are very important for what follows, so let us make some comments.

\begin{remark}\label{relmanrem}
(i) In the particular case of schemes with the trivial log structure over $\QQ$, we say that a regular $Y$ is a {\em manifold} if $\cD_{Y/\QQ}$ is separating. Manifolds were studied by Matsumura in detail; he called them schemes satisfying the weak Jacobian condition (WJ), see \cite[\S30]{Matsumura-ringtheory}. Matsumura proved, that if $X,Y$ are regular, $X\to Y$ is of finite type, and $Y$ is a manifold, then $X$ is a manifold. Also, he proved that it suffices to check the weak Jacobian condition (WJ) only at maximal ideals, and deduced that many natural examples of excellent regular schemes (spectra of convergent power series, etc.) are manifolds.

(ii) We postpone a detailed study of the notion of relative logarithmic manifold to another paper, and will only prove the necessary minimum -- Theorem~\ref{manifoldth} and Lemma~\ref{logmanlem}. However, it seems that at least if $\dim(Z)\le 1$, then most natural logarithmically regular morphisms $Y\to Z$ arising from formal or analytic geometries are logarithmic $Z$-manifolds. On the other hand, it seems that the class of abundant morphisms is rather small.

(iii) Intuitively, tangent spaces on a manifold are glued into a vector bundle $\cD_{Y/Z}$, and this is what happens in the logarithmically smooth case. In general, the condition of being logarithmically separating is the minimal assumption needed to guarantee that $\cD_{Y/Z}$ is large enough to be useful for us. Note also that it would not be enough to work with derivations of the localizations $\cD_{\cO_y/\cO_{f(y)}}$, since it can be larger than the stalk $\cD_{Y/Z,y}$ when $\cD_{Y/Z}$ is not quasi-coherent. So, certain coherence condition is built into the definition of logarithmic manifolds.

(iv) In fact, most constructions of our principalization algorithm only need that $\cD_{Y/Z}$ is separating. However, this property is not stable in the following sense: there might exist logarithmically smooth morphisms $Y'\to Y$, even logarithmic blow ups, such that $\cD_{Y'/Z}$ is not separating. The technical reason why this happens is that monomial parameters can be ``transformed'' to regular parameters as we will see in Proposition~\ref{composlem}. It seems that logarithmically separating is the weakest stable property that covers our needs.
\end{remark}

\subsubsection{The logarithmically smooth case}
If $f$ is logarithmically smooth, then $\cD$ is free and its rank at $y$ equals to the size of a family of parameters $(t,q,u)$ at $y$. Clearly, $|t|=\dim(T_{S,y})$, and using that the characteristic is zero, we also have that $$|q|=\rk_\QQ(\oM^\gp_{y/z})=\dim(\Hom(\oM^\gp_{y/z},k(y))),\ \ |u|=\trdeg(k(y)/k(\tilz))=\dim(\cD_{k(y)/k(\tilz)}).$$ Therefore, all homomorphisms $\phi_i$ have to be surjective and we obtain

\begin{lemma}\label{logsmoothlem}
Any logarithmically smooth morphism $f\:Y\to Z$ has abundance of derivations. In addition, for any choice of parameters $(t,q,v)$ at a geometric point $y$ of $Y$, there exists a unique dual family of logarithmic derivations $(\partial_t,\partial_q,\partial_v)$ and they form a basis of $\cD_{Y/Z,y}$.
\end{lemma}

\subsection{Relative logarithmic orbifolds}\label{logorbsec}
As in the absolute case of \cite{ATW-principalization}, the relative principalization algorithm involves non-representable analogues of blow ups. So, we are going to extend the theory of relative logarithmic manifolds to DM stacks.

\subsubsection{Logarithmically regular morphisms of stacks}
A morphism $f\:Y\to Z$ of logarithmic DM stacks is called {\em logarithmically regular} if it logarithmically regular \'etale-locally on $Y$ and $Z$. The latter means that there exist a compatible strict \'etale coverings of $Y$ and $Z$ by logarithmic schemes $Y_0$ and $Z_0$ such that the morphism $f_0\:Y_0\to Z_0$ is logarithmically regular. Moreover, this property is actually independent of the choice of a covering because logarithmic regularity is local with respect to strict \'etale morphisms.

A strict closed substack $Y'\into Y$ is called a {\em logarithmic $Z$-submanifold} if its preimage $Y'\times_YY_0$ is a logarithmic $Z_0$-submanifold.

\subsubsection{Sheaves of derivations}\label{sheafdersec}
Given a logarithmic $Z$-orbifold $Y$ find a strict \'etale covering $p\:Y_0\to Y$, $Z_0\to Z$, and set $Y_1=Y_0\times_YY_0$. Strict \'etale (and even logarithmically \'etale) morphisms are compatible with the sheaves of logarithmic derivations, hence $p_1^!\cD_{Y_0/Z_0}=\cD_{Y_1/Z_0}=p_2^!\cD_{Y_0/Z_0}$ and this isomorphism satisfies the cocycle condition. By \'etale descent we obtain a sheaf $\cD_{Y/Z}$ of logarithmic derivations of $Y/Z$. In the same way the sheaves of operators $\cD^{(\le i)}_{Y/Z}$ are defined. A submodule $\cF\subseteq\cD_{Y/Z}$ is {\em separating}, {\em logarithmically separating} or {\em abundant} if this property is satisfied by the submodule $\cF_0=p^!\cF$ of $\cD_{Y_0/Z}=p^!\cD_{Y/Z}$.

\subsubsection{Logarithmic orbifolds}
We say that $f$ is a {\em relative logarithmic orbifold} if it is logarithmically regular and $\cD_{Y/Z}$ is logarithmically separating. Sometimes, if $f$ is representable we will stress this by saying that $f$ is a {\em relative logarithmic manifold.} This is compatible with the terminology of Section \ref{logmansec}.

\subsection{Functoriality}\label{functorsec}
It will be important throughout the text to check that all our constructions on logarithmic $Z$-orbifolds $Y$ are functorial with respect to logarithmically smooth morphisms $Y'\to Y$ and arbitrary base changes $Z'\to Z$. So, let us study functoriality of notions introduced so far.

\subsubsection{Base change}
Base change functoriality is easy.

\begin{proposition}\label{basechangeprop}
Let $f\:Y\to Z$ and $g\:Z'\to Z$ be morphisms of logarithmic DM stacks with base changes $f'\:Y'\to Z'$ and $g'\:Y'\to Y$, and assume that all stacks are noetherian and $f$ is logarithmically regular. Then

(i) The morphism $f'$ is logarithmically regular. Furthermore, if $f$ is a relative logarithmic orbifold or has an abundance of derivations, then $f'$ satisfies the same property.

(ii) If a submodule $\cF$ of $\cD_{Y/Z}$ is separating, logarithmically separating or abundant, then $g'^*(\cF)$ satisfies the same property.

(iii) Let $y'\to Y'$ be a geometric point with $y=g'(y')$. If $t\subset\cO_y$ is a family of regular, logarithmic or constant parameters to $f$ at $y$, then its image in $\cO_{y'}$ is such a family for $f'$ at $y'$.

(iv) If $T\into Y$ is a logarithmic $Z$-submanifold, then $T'=T\times_YY'$ is a logarithmic $Z'$-submanifold of $Y'$.
\end{proposition}
\begin{proof}
Note that (iv) and the first claim of (i) are covered by Lemma \ref{logreglem}. Claim (iii) for regular and constant parameters follows from Lemma~\ref{fibchangelem}. In addition, it is easy to see that $\oM_{y/z}=\oM_{y'/z'}$ by a general property of monoid pushouts, and we obtain claim (iii) for monomial parameters. Finally, (ii) is a direct consequence of (iii), and (ii) implies the second claim of (i).
\end{proof}

\subsubsection{Parameters of a composition}\label{paramcompossec}
Next we will study compositions $h=f\circ g$ of logarithmically regular morphisms. The main point will be to describe local parameters of $h$, and loosely speaking, they can be obtained by combining parameters of $f$ and $g$ but with one important subtlety: some monomial parameters of the second morphism should be replaced by parameters of other types. This behavior is well-understood in case of logarithmically smooth morphisms, and our task is to describe it when $f$ is only logarithmically regular. We start with isolating the ``redundant'' monomial parameters.

Let $X\stackrel g\to Y\stackrel f\to Z$, points $x,y,z$ and neat charts be as in \S\ref{compossec}, and assume that $f$ and $g$ are logarithmically regular. In particular, $Q^\gp/P^\gp=\oM^\gp_{y/z}$ and $R^\gp/Q^\gp=\oM^\gp_{x/y}$, but the kernel $N$ of the surjection $R^\gp/P^\gp\onto\oM^\gp_{x/z}$ can be non-trivial. By a simple diagram chase (see \cite[Remark II.2.4.6]{Ogus-logbook}), there is an exact sequence
$$0\to N\to\oM^\gp_{y/z}\to\oM^\gp_{x/z}\to\oM^\gp_{x/y}\to 0.$$ Choose a family $r\subset R$ of monomial parameters of $g$ at $x$ and subsets $n\subset Q^\gp$, $q\subset Q$ such that the image of $(n,q)$ in $\oM_{y/z}^\gp\otimes\QQ$ is a basis and the image of $n$ is a basis of $N\otimes\QQ$. One can view $(n,q)$ as virtual monomial parameters of $f$ at $y$. In fact, it is easy to see that the preimage of $N$ in $Q^\gp$ has zero intersection with $Q$ and hence one cannot choose $n$ in $Q$.

Let $N'\subseteq Q^\gp\subseteq R^\gp$ be the lattice with basis $n$, and let $Z'=Z\times\bfD_{N'}$. Since $N'$ is mapped to 0 in $\oM^\gp_{x/z}$, it is contained in $P^\gp+R^\times$. In particular, shifting the elements of $n$ by elements of $P^\gp$ we can achieve that in addition $n\subset R^\times$, and hence a natural morphism $X\to X_P[R]\to Z'$ arises. We denote the image of $x$ by $z'\in Z'$.

\begin{proposition}\label{composlem}
Keep the above notation and let $(t,w)$ and $(s,v)$ be families of regular and constant parameters of $f$ and $g$ at $y$ and $x$, respectively. Then there exists a set $n'=(n'_1,n'_2)\subset\cO_{z'}$ such that $|n'|=|n|$ and the sets $(t,s,n'_1)$, $(w,v,n'_2)$, $(q,r)$ are the families of regular, constant and monomial parameters of $h$ at $x$.
\end{proposition}
\begin{proof}
We use the notation and  diagram from \S\ref{compossec} describing the composition of neat charts $\alpha:X \to Y_Q[R]$ and $\beta:Y \to Z_P[Q]$. By the usual theory of regular morphisms, parameters of $\gamma: X \to Z_P[R]$ at $x$ can be obtained by combining parameters of $\alpha$ and the base change $\beta_Q[R]$ of $\beta$, hence $(t,s)$ and $(w,v)$ are regular and constant parameters of $\gamma$. Since $\gamma$ is strict and regular, the parameters of $h$ are obtained by combining the parameters of $\gamma$ and $\lambda\:Z_P[R]\to Z$, hence it remains to describe parameters for $\lambda$. At this stage we already proved that constant and regular parameters of $f$ are pulled back to constant and regular parameters of $h$.

By construction, $q$ and $r$ are mapped to bases of the vector spaces $\oM^\gp_{x/y}\otimes\QQ$ and $(\oM^\gp_{y/z}/N)\otimes\QQ$, hence $(q,r)$ is mapped to a basis of their extension $\oM^\gp_{x/z}$, that is, $(q,r)$ is a family of monomial parameters of $\lambda$ at $x'=\gamma(x)$.

Finally, notice that $\lambda$ factors into the composition $Z_P[R]\to Z'\to Z$ of logarithmically regular morphisms. Since $Z'\to Z$ is a strict regular morphism of relative dimension $|n|=\operatorname{Rank} N'$, it possesses a set $n'=(n'_1,n'_2)$ of parameters at $z'$ consisting only of regular and constant parameters, and clearly $|n|=|n'|$.

In the first paragraph above we have proved that certain parameters pull back to parameters of the same type under composition. Applying this to $Z_P[R]\to Z'\to Z$, we obtain that the parameters $n'$ are pulled back to parameters of corresponding type of $\lambda$ at $x' \in Z_P[R]$, that we also denote by $n'$. Since the types are different, the union $(n',q,r)$ is a partial family of parameters of $\lambda$ at $x'$. To show that this is a full family  it remains to count dimensions: $$|(n',q,r)|=\rk_\QQ(N)+\rk_\QQ(\oM^\gp_{x/z})=\rk_\QQ(R^\gp/P^\gp),$$ hence $(n',q,r)$ is a full family of parameters at $x'$.
\end{proof}

\subsubsection{Compositions}
Now, we can establish the following properties of compositions.

\begin{proposition}\label{coverprop}
Assume that $g\:X\to Y$ and $f\:Y\to Z$ are logarithmically regular morphisms of DM logarithmic stacks and $h\:X\to Z$ is the composition. Then

(i) The morphism $h$ is logarithmically regular. Furthermore, if $g$ is of finite type and $f$ is a relative logarithmic orbifold or has abundance of derivations, then $h$ satisfies the same property as $f$ does.

(ii) If $g$ is of finite type and a submodule $\cF$ of $\cD_{Y/Z}$ is logarithmically separating or abundant, then the preimage of $g^*\cF$ under the homomorphism $\cD_{X/Z}\to g^!(\cD_{Y/Z})$ satisfies the same property.

(iii) Assume that $X,Y,Z$ are logarithmic schemes, $t_1\.t_n$ are regular parameters of $g$ at a point $x\in X$, and $s_1\.s_l$ are regular parameters of $f$ at $g(x)$. Then $t_1\.t_n,g^*(s_1)\.g^*(s_l)$ is a partial family of regular parameters of $h$ at $x$.

(iv) If $T\into Y$ is a logarithmic $Z$-submanifold, then $T\times_YX$ is a logarithmic $Z$-submanifold of $X$.
\end{proposition}
\begin{proof}
Lemma \ref{logreglem} implies that $h$ is logarithmically regular, and the remaining claims of (i) follow from (ii). Part (iii) is covered by Proposition~\ref{composlem}.

We prove  (iv). Lemma  \ref{logreglem} implies that $T\times_YX \to Z$ is logarithmically regular, being the composition of the logarithmically regular morphism $T\to Z$ and the base change $T\times_YX\to T$ of the logarithmically regular morphism $g$. By Lemma \ref{submanlem} it suffices to note that regular parameters defining   $T \subset Y$ at a point $y= g(x)$ provide regular parameters defining  $T\times_YX \subset X$ at $x$.

It remains to prove (ii). For concreteness, we assume that $\cF$ is abundant, but the other case can be obtained by taking $w_0=\emptyset$ below. By Lemma~\ref{firstseqlem} there is an exact sequence
$$0\to\cD_{X/Y}\to\cD_{X/Z}\stackrel\phi\to g^!(\cD_{Y/Z})\to 0.$$
We will work locally at a geometric point $x$ with images $y$, $z$. Choose parameters $(t,s,n'_1)$, $(w,v,n'_2)$, $(q,r)$ at $x$ as in Proposition~\ref{composlem}. Note that only $w$ can be infinite by our assumptions. By Lemma~\ref{standardderiv} it suffices to prove that any finite subfamily of parameters has a dual family of derivations. It thus suffices to work with a subfamily, where $w$ is replaced by a finite subset $w_0$ and the rest is unchanged. Since $\cD_{X/Y}$ is abundant, the family $(s,v,r)$ possesses a dual family $(\partial_s,\partial_v,\partial_r)$, whose image in $\cD_{X/Z}$ vanish on all parameters coming from $Y$. So, these derivations form a part of the dual family for $(s,v,r,t,w_0,q,n')$. Moreover, it suffices now to find a family dual to $(t,w_0,q,n')$, because by adding a linear combination of $\partial_s,\partial_v,\partial_r$ one can also achieve that its elements vanish on $(s,v,r)$.

Let $n\subset M^\gp_{y/z}$ be as in \S\ref{paramcompossec} where we noted that $(q,n)$ is only a family of virtual monomial parameters; we correspondingly define a virtually dual family of derivations $(\partial_t,\partial_w,\partial_q,\partial_n)$, where the virtual equation $\partial_a(u^a)=u^a$ for $a\in n$ means that the logarithmic derivative $\delta_a(u^a)=1$, where $\delta\:M_y\to\cO_Y$ is extended to $M^\gp_y$ by linearity. Using abundance of $\cF$ it is easy to see that it contains such a virtually dual family. Viewing its elements as derivations with values in $\cO_X$ we obtain a family $(g^*\partial_t,g^*\partial_w,g^*\partial_q,g^*\partial_n)$ in $g^*\cF$. It is not dual to the parameters $(t,w_0,q,n')$ only because $g^*\partial_n$ is dual to $n$ but not to $n'$. We claim that there exists a matrix $A\in M_{|n|\times|n|}(\cO_x)$ such that $Ag^*\partial_n$ is a family of $|n|$ logarithmic derivations $\cO_y\to\cO_x$ dual to $n'$. Once this is proved, the family $(g^*\partial_t,g^*\partial_w,g^*\partial_q,Ag^*\partial_n)$ is dual to $(t,w_0,q,n')$, and hence any its preimage under $\phi$ is such a family in $\cD_{X/Z,x}$, completing the proof.

It remains to construct $A$. Note that its coefficients solve a system of $|n|^2$ linear equations and we should only prove that a solution exists in $\cO_x$. First we recall that $n\subset M_x^\times$ hence for any $a\in n$ we have that $\partial'_a=u^{-a}\partial_a\in g^*\cD_{Y/Z}$, and the family $\partial'_n$ is dual to $n$ in the sense of usual derivations. (Here we have already used that the coefficients in $g^*\cF$ are in $\cO_x$ rather than $\cO_y$.) Their restrictions onto $Z'=Z\times \bfD_{N'}$ is the basis of the free module $\cD_{Z'/Z}$. Since $n'$ is a parameter system for the regular morphism $Z'\to Z$ at $z'$, it has a dual family of derivations $\partial_{n'}\in\cD_{Z'/Z,z'}$. Therefore already over $\cO_{z'}$ there exists a matrix $A$ such that $\partial_{n'}=A\partial_n$.
\end{proof}

\begin{corollary}\label{logfibreg}
Let $g\:X\to Y$ and $f\:Y\to Z$ be logarithmically regular morphisms, and let $S$ be a logarithmic fiber of $h=f\circ g$. Then $g$ takes $S$ to a logarithmic fiber $T$ of $f$ and for any point $x\in S$ with $y=g(x)$ the homomorphism $\phi_x\:\hatcO_{T,y}\to\hatcO_{S,x}$ is regular.
\end{corollary}
\begin{proof}
The claim can be checked \'etale locally, hence we can assume that $f$ and $g$ possess charts $Y\to Z_P[Q]$ and $X\to Y_Q[R]$ and a diagram as in \S\ref{compossec} arises. The logarithmic fibers of $h$ and $f$ are the connected components of the preimages in $X$ and $Y$ of $T_{R^\gp/P^\gp}$-orbits on $Z_P[R]$ and $T_{Q^\gp/P^\gp}$-orbits on $Z_P[Q]$, respectively. The map $X_P[R]\to X_P[Q]$ is equivariant with respect to the homomorphism $T_{R^\gp/P^\gp}\to T_{Q^\gp/P^\gp}$, hence it takes orbits of the first group to the orbits of the second group. It follows that $g$ takes the logarithmic fibers of $h$ to the logarithmic fibers of $f$.

By Corollary~\ref{logfibcor} the rings $\hatcO_{T,y}$ and $\hatcO_{S,x}$ are regular and hence of the form $k\llbracket t_1\.t_s\rrbracket$. By Proposition~\ref{coverprop}(iii) $\phi_x$ takes a family of parameters of $\hatcO_y$ to a partial family of parameters of $\hatcO_x$, hence it is regular.
\end{proof}

\subsubsection{Finite type stability}
Finally, we show that largeness of $\cD_{Y/Z}$ is preserved under morphisms of finite type:

\begin{theorem}\label{manifoldth}
Assume that $f\:Y\to Z$ is a relative logarithmic orbifold and $g\:Y'\to Y$ is a morphism of finite type such that $f'\:Y'\to Z$ is logarithmically regular. Then $f'$ is a relative logarithmic orbifold too. In particular, a logarithmic $Z$-submanifold of a logarithmic $Z$-orbifold is a logarithmic $Z$-orbifold too. In addition, if $f$ has abundance of derivations then so does $f'$.
\end{theorem}
\begin{proof}
The question is \'etale local, hence we can assume that $Y$ and $Y'$ are logarithmic $Z$-manifolds, and it suffices to work locally at a geometric point $y'\in Y'$ with $y=g(y')$. Notice that any morphism of finite type between logarithmic schemes locally can be split into a composition of a strict closed immersion followed by a logarithmically smooth morphism. Indeed, Theorem~\ref{neatth} reduces the claim to the case of strict morphisms, which follows since $Y'$ locally embeds in some $\AA^n_Y$. In particular, locally $Y'$ can be embedded as a logarithmic $Z$-submanifold into a logarithmic $Z$-manifold $W$ such that $W$ is logarithmically smooth over $Y$. Therefore, it suffices to consider two cases: $Y'\to Y$ is logarithmically smooth, $Y'\into Y$ is a logarithmic $Z$-submanifold. The first case is covered by Proposition~\ref{coverprop}(i).

In the second case, choose a family $(t,u,q)$ of parameters at $y$ such that $t=(t',t'')$ and $Y'$ is given by the vanishing of $(t')$ locally at $y$. In particular, $(t'',u,q)$ is a family of parameters at $y'$. If $f$ is abundant, let $u_0$ be a finite subset of $u$, and take $u_0=\emptyset$ otherwise. By Lemma~\ref{standardderiv}(ii) there exists a family of derivations $(\partial_{t'},\partial_{t''},\partial_u,\partial_q)$ dual to $(t,u_0,q)$. Since $(\partial_{t''},\partial_u,\partial_q)$ vanish on $(t')$, they restrict to derivations on $Y'$ which form a dual family to $(t'',u_0,q)$. So, $f'$ is abundant (resp. logarithmically separating) at $y'$ by Lemma~\ref{standardderiv}(ii).
\end{proof}

\begin{lemma}\label{logmanlem}
Assume that $f\:Y\to Z$ is a logarithmically regular morphism of logarithmically regular logarithmic schemes, $Y=\Spec(A)$, $Z=\Spec(R)$ are spectra of complete local rings, and $\dim(Z)\le 1$. Then $f$ is a relative logarithmic manifold.
\end{lemma}
\begin{proof}
Recall that $A=l\llbracket Q\rrbracket\llbracket t_1\.t_n\rrbracket$ with the logarithmic structure given by a sharp fs monoid $Q$. A similar description holds for $R$ and yields three possible cases: $R=k$, $R=k\llbracket x\rrbracket$, $R=k\llbracket P\rrbracket$, where the logarithmic structure is given by $P=\NN$ in the third case and is trivial otherwise. We will deal only with the third case, as the other two are similar and a bit simpler.

Extend the homomorphism $P\to Q$ to a surjective homomorphism $Q'=P\oplus\NN^r\onto Q$ and set $A'=l\llbracket Q'\rrbracket\llbracket t_1\.t_n\rrbracket$. Then $Y$ is finite over $Y'=\Spec(A')$, and by Theorem~\ref{manifoldth} it suffices to prove the claim for $Y'$ instead of $Y$. So, we assume that $Q=P\oplus\NN^r$ and denote its free generators $p,q_1\.q_r$.

Let us now prove that $\cD_{Y/Z}$ is logarithmically separating at a point $y\in Y$. Assume first that $y$ is a {\em coordinate point} in the sense that $m_y$ is generated by a subset $S$ of $\{p,q_1\.q_r,t_1\.t_n\}$. Then already the subsheaf of $\cD_{Y/Z}$ generated by $q_i\partial_{q_i}$ and $\partial_{t_j}$ with $q_i,t_j\in S$ is logarithmically separating at $y$ over $Z$. Assume now that $y$ is arbitrary. Renumbering $q$'s we can assume that $q_i\in m_y$ if and only if $1\le i\le r'$, and then $\oM_y=\oplus_{i=1}^{r'}q_i\NN$. Let $r'+n'$ be the codimension of $y$ in the fiber over $Z$, then there exists $f_1\.f_{n'}\in A$ such that $m_y$ is generated by $q_1\.q_{r'},f_1\.f_{n'}$ and $m_y\cap\{p\}$. It is easy to see that there exist $f_{n'+1}\.f_n$ such that $p,q_1\.q_r,f_1\.f_n$ generate an open ideal. Therefore, the induced morphism $Y\to Y'=\Spec(l\llbracket Q\rrbracket\llbracket f_1\.f_n\rrbracket)$ is finite, \'etale at $y$ and takes $y$ to a coordinate point $y'$. Since $Y'$ is a logarithmic $Z$-manifold at $y'$, we obtain that $Y$ is a logarithmic $Z$-manifold at $y$.
\end{proof}

\subsection{Relative logarithmic orders of ideals}

\subsubsection{Monomial saturation}\label{monomsatsec}
By the {\em monomial saturation} $\cM(\cI)$ of an ideal $\cI$ on a logarithmic DM stack $Y$ we mean the minimal monomial ideal containing $\cI$.

\subsubsection{Differential saturation}\label{difsatur}
Let $Y\to Z$ be a logarithmically regular morphism and $\cF\subseteq\cD_{Y/Z}$ an $\cO_Y$-submodule. We say that an ideal $\cI\subseteq\cO_Y$ is {\em $\cF$-saturated} if $\cF^{(\le 1)}(\cI)=\cI$. For example, any monomial ideal is $\cF$-saturated. The ideal, $\cF^{\infty}(\cI)$ will be called the {\em $\cF$-saturation} of $\cI$. Obviously, it is the minimal $\cF$-saturated ideal containing $\cI$, in particular, $\cF^{\infty}(\cI)\subseteq\cM(\cI)$. The ideal $\cD_{Y/Z}^{\infty}(\cI)$ will be called the $\cD$-saturation of $\cI$ over $Z$.

\subsubsection{Logarithmically clean ideals}\label{Sec:log-clean}
We say that $\cI$ is {\em logarithmically clean at} a point $y\in Y$ over $Z$ if $\cD_{Y/Z}^{\infty}(\cI)_y=\cM(\cI)_y$, and $\cI$ is {\em logarithmically clean} over $Z$ if $\cD_{Y/Z}^{\infty}(\cI)=\cM(\cI)$. We record the following obvious fact:

\begin{lemma}\label{logcleanlem}
Let $f\:Y\to Z$ be a logarithmically regular morphism, $\cI\subseteq\cO_Y$ an ideal, and $\cF\subseteq\cD_{Y/Z}$ an $\cO_Y$-submodule. If $\cF^{\infty}(\cI)$ is monomial, then $\cI$ is logarithmically clean over $Z$ and $\cF^{\infty}(\cI)=\cD_{Y/Z}^{\infty}(\cI)=\cM(\cI)$.
\end{lemma}

\begin{remark}
(i) In the absolute situation all ideals are logarithmically clean, see \cite[Theorem~3.4.2]{ATW-principalization}.

(ii) A typical example is illustrated by the following model case: $Z=\Spec(O)$ with a chart $P\to O$ and $Y=\Spec(A)$ with $A=O_P[Q][t_1\.t_n]$. If $I\subseteq A$ is generated by elements $h_i=\sum_{q\in Q,l\in\NN^n}b_{iql}u^qt^l$, then it is easy to see (and will be shown in Section~\ref{monomsec}) that $\cD_{A/O}^{\infty}(I)$ is the ideal generated by the elements $b_{iql}u^q$. In particular, if $b_{iql}$ are monomials, then $\cI$ is logarithmically clean. This can be achieved by a modification of the base, and even just by enlarging $P$. Surprisingly, extending this observation to more general logarithmically regular morphisms, and even logarithmically smooth ones, is substantially more difficult. This will be the central topic of Section~\ref{monomsec}.
\end{remark}

\subsubsection{The logarithmic order}
Let $Y\to Z$ be a logarithmically regular morphisms and $\cI$ an ideal on $Y$. Similarly to \cite[\S3.6.1]{ATW-principalization}, if $Y$ is a scheme then by the \emph{logarithmic order $\logord_{\cI/Z}(y)$ of $\cI$ relative to $Z$ at a point $y\in Y$} we mean the usual order of $\cI|_S$ at $y$, where $S=S_y$ is the logarithmic fiber of $y$. Since logarithmic order is compatible with strict \'etale morphisms $Y'\to Y$, this definition extends to the case when $Y$ is a logarithmic DM stack. In particular, we obtain a function $\logord_{\cI/Z}\:|Y|\to\NN\cup\{\infty\}$. By $\logord_{\cI/Z}(Y)=\max_{y\in Y}\logord_{\cI/Z}(y)$ we denote the {\em logarithmic order} of $\cI$ on $Y$.

\subsubsection{Clean ideals}\label{balancedsec}
Let $Y\to Z$ be a logarithmically regular morphism. An ideal $\cI$ is called {\em $Z$-clean at} a point $y\in|Y|$ if $\logord_{\cI/Z}(y)<\infty$, and $\cI$ is called {\em $Z$-clean} or simply {\em clean} if it is $Z$-clean at all points of $Y$.

\subsubsection{Relation to derivations}
As in the absolute case (\cite[Lemma 3.6.3]{ATW-principalization}), the relative logarithmic order can be computed using derivations.

\begin{lemma}\label{logordlem}
Let $Y\to Z$ be a logarithmically regular morphism, $\cF\subseteq\cD_{Y/Z}$ a separating $\cO_Y$-submodule, $\cI$ an ideal, and $y\in|Y|$ a point. Then $$\logord_{\cI/Z}(y)=\min\{a\in\NN\mid {\cF^{(\leq a)}(\cI)}_y=\cO_{Y,y}\},$$ where $\min(\emptyset)=\infty$ by convention. In particular, $\cI$ is clean at $y$ if and only if $y\notin V(\cF^{\infty}(\cI))$, and $\cI$ is clean if and only if $\cF^{\infty}(\cI)=1$.
\end{lemma}
\begin{proof}
Let $S$ be the logarithmic fiber of $y$, and set $\cI_S=\cI|{_S}$ and $\cF_S=\cF|{_S}$. Since $\cF^{(\leq a)}(\cI)|{_S}=(\cF_S)^{(\le a)}(\cI_S)$, we should check that $\ord_y(\cI_S)$ is the minimal $a$ such that $\cF_S^{(\leq a)}(\cI_S)_y=\cO_{S,y}$. It suffices to check the latter in the formal completion $\hatcO_{S,y}$. Fix a family of regular parameters $t_1\.t_n$. Then $\hatcO_{S,y}\toisom k(y)\llbracket t_1\.t_n\rrbracket$, and by Lemma~\ref{standardderiv}, $\cF_S$ contains derivations $\partial_i$ such that $\partial_i(t_j)=\delta_{ij}$ (though the action on $k(y)$ can be non-trivial).

Now, it suffices to prove that for any $h\in k(y)\llbracket t_1\.t_n\rrbracket$ of order $a>0$ there exists $i\in\{1\. n\}$ such that $\partial_i h$ is of order $a-1$. Without restriction of generality, $h=\sum_n h_nt_1^n $, where $h_n\in k(y)\llbracket t_2\.t_n\rrbracket$ and the inequality $a\le\ord(h_n)+n$ is an equality for some $n_0>0$. It is easy to see that $\ord(\partial_1 g)\ge\ord(g)$ for any $g\in k(y)\llbracket t_2\.t_n\rrbracket$, and using that $\partial_1(h_nt_1^n)=(n-1)h_nt_1^{n-1}+t_1^n\partial_1(h_n)$ one obtains that $\ord(\partial_1(h))=\max_{n>0}(\ord(h_n)+n-1)=a-1$.
\end{proof}

\begin{remark}
Lemma \ref{logordlem} implies that if $\cD_{Y/Z}$ is separating, then any clean ideal is logarithmically clean.
\end{remark}

\subsubsection{Balanced ideals}
An ideal $\cI\subseteq\cO_Y$ is called {\em balanced} if it is logarithmically clean and $\cM(\cI)$ is invertible.

\begin{lemma}\label{balncedlem}
An ideal $\cI$ is balanced if and only if it is of the form $\cI=\cN\cdot\cI^\cln$, where $\cN$ is invertible monomial and $\cI^\cln$ is clean.
\end{lemma}
\begin{proof}
Only the direct implication needs a proof. If $\cI$ is balanced, then $\cN=\cD_{Y/Z}^{\infty}(\cI)$ is invertible, hence an ideal $\cI^\cln=\cN^{-1}\cI$ is defined. Since $$\cN=\cD_{Y/Z}^{\infty}(\cN\cI^\cln)=\cN\cD_{Y/Z}^{\infty}(\cI^\cln),$$ we obtain that $\cD_{Y/Z}^{\infty}(\cI^\cln)=1$, and hence $\cI^\cln$ is clean by Lemma~\ref{logordlem}.
\end{proof}

The ideal $\cI^\cln=\cN^{-1}\cI$ will be called the {\em clean part} of $\cI$.

\subsubsection{Functoriality}
We conclude this section with studying functoriality of $\cD$-saturation and logarithmic order. We start with base change functoriality. Recall that we only consider noetherian base changes in the sense of \S\ref{convsec}.

\begin{lemma}\label{functorlogorder1}
Let $f\:Y\to Z$ be a logarithmically regular morphism of logarithmic DM stacks, $\cI\subseteq\cO_Y$ an ideal, $\cF\subseteq\cD_{Y/Z}$ an $\cO_Y$-submodule, $g\:Z'\to Z$ a morphism of logarithmic DM stacks with noetherian base changes $g'\:Y'\to Y$ and $f'\:Y'\to Z'$, and $\cI'=\cI\cO_{Y'}$. Then

(i) For a geometric point $y'$ of $Y'$ we have $\logord_{\cI/Z}(f'(y'))=\logord_{\cI'/Z'}(y')$. In particular, if $\cI$ is $Z$-clean, then $\cI'$ is $Z'$-clean.

(ii) If $\cF^{\infty}(\cI)$ is monomial, then $g'^{-1}(\cF^{\infty}(\cI))=g'^*\cF^{\infty}(\cI')$.

(iii) If $\cI$ is logarithmically clean over $Z$, then $\cI'$ is logarithmically clean over $Z'$.
\end{lemma}
\begin{proof}
Claim (i) follows from Lemma~\ref{fibchangelem}. Claim (ii) can be checked \'etale-locally using global sections $h\in\Gamma(\cI)$ and $\partial\in\Gamma(\cF)$, and then it becomes a tautology. Taking $\cF=\cD_{Y/Z}$ and applying Lemma~\ref{logcleanlem} we obtain (iii).
\end{proof}

Functoriality for logarithmically regular morphisms is checked similarly.

\begin{lemma}\label{functorlogorder2}
Let $g\:X\to Y$ and $f\:Y\to Z$ be logarithmically regular morphism of logarithmic DM stacks, $\cI\subseteq\cO_Y$ an ideal with $\cI'=\cI\cO_X$, and $\cF\subseteq\cD_{Y/Z}$ a submodule. Then,

(i) $\logord_{\cI/Z}=\logord_{\cI'/Z}\circ|g|$. In particular, if $\cI$ is $Z$-clean, then $\cI'$ is $Z$-clean too.

(ii) If $\cF^{\infty}(\cI)$ is monomial, then $g^{-1}(\cF^{\infty}(\cI))=g^*\cF^{\infty}(g^{-1}\cI)$.

(iii) If $\cI$ is logarithmically clean over $Z$, then $\cI'$ is logarithmically clean over $Z$ and $\cM(\cI')=g^{-1}(\cM(\cI))$.
\end{lemma}
\begin{proof}
Order of ideals on regular schemes can be computed formally locally, hence (i) follows from Corollary~\ref{logfibreg}. Claims (ii) and (iii) are checked as in Lemma~\ref{functorlogorder1}.
\end{proof}

\section{Monomialization of $\cD$-saturated ideals}\label{monomsec}
\addtocontents{toc}
{\noindent We introduce monomial ideals, $B$-monomial ideals and $\cD$-saturated ideals, and study their relations and properties. We introduce modifications of logarithmic schemes. Finally we show that $\cD$-saturated ideals can be monomialized  by base change.}

This section is devoted to proving a monomialization theorem, Theorem \ref{monomialization}. This is the only ingredient in our algorithms resulting in a base change $B'\to B$. Unfortunately, the argument is non-functorial and involves choices. It consists of two parts: (1) induction on codimension based on results of \S\ref{formsec}, and (2) localization arguments in Lemmas~\ref{monomdescentlem}, \ref{henslem} and \ref{secondmonom} based on results of \S\ref{modifsec}, cofinality arguments and RZ spaces. Part (2) is non-functorial. This seems to be unavoidable in general. In a separate manuscript, we will provide more functorial arguments in case $f$ is proper.

\subsection{Monomial ideals}
We say that an ideal $\cI$ on a logarithmic stack $Y$ is {\em monomial} at a geometric point $y\to Y$ if \'etale locally it is generated by monomials $u^q\in M_y$. Monomial ideals play an important role in the relative resolution algorithm. To begin we establish some their properties.

\subsubsection{The class $\bB$}\label{Bsec}
Let $\bB$ be the class of noetherian qe logarithmically regular logarithmic DM stacks of characteristic zero. In the sequel, we will usually stick to the following notation: $f\:Y\to Z$ denotes a logarithmically regular morphism, $f\:Y\to B$ is used if in addition the target is in $\bB$, and $f\:X\to B$ is used if in addition $f$ is a relative logarithmic orbifold.

\subsubsection{Schematic closure}
In our case, monomiality is preserved by closures.

\begin{lemma}\label{closlem}
Assume that $Y\in\bB$, $Y_0\subseteq Y$ is open, $Z_0\into Y_0$ is a monomial substack and $Z\into Y$ is the schematic closure of $Z_0$. Then $Z$ is also monomial.
\end{lemma}
\begin{proof}
The claim can be checked \'etale-locally, hence we can assume that $Y$ is a scheme. Monomiality of a subscheme is an open condition, hence by noetherian induction it suffices to consider the case when $Y_0=Y\setminus\{y\}$ for a closed point $y$. The latter case is local at $y$, hence we can assume that $Y=\Spec(\cO_y)$. Flatness of the completion implies that it is compatible with schematic closures and an ideal $\cI\subseteq\cO_y$ is monomial if and only if $\cI\hatcO_y$ is monomial. This reduces the claim to the case when $Y=\Spec(A)$ for a complete local ring $A$. Since $Y\in\bB$, we have that $A=B\llbracket P\rrbracket$, where $B$ is a complete normal local ring.

Consider the flat morphism $Y\to T=\bfA_P$ and let $T_0$ be the image of $Y_0$. If $A$ is not a field, then $T_0=T$, and $T_0=T\setminus\{0\}$ otherwise. By assumption, locally $Z_0$ is induced from monomial subschemes of $T_0$. Since the set of monomial schemes is discrete and the fibers of the map $Y_0\to T_0$ are easily seen to be connected, these subschemes glue to a single monomial subscheme $V_0\into T_0$ whose preimage in $Y_0$ is $Z_0$. It follows that $Z$ is the preimage of the schematic closure $V\into T$ of $V_0$. Using notation from \ref{Sec:charts}, since $V_0$ is $\bfD_{P^\gp}$-equivariant, $V$ is $\bfD_{P^\gp}$-equivariant too, which means that it is monomial. Thus, its preimage $Z$ is monomial too.
\end{proof}

\subsubsection{Kummer descent}
On nice enough logarithmic schemes monomiality is local in the Kummer topology.

\begin{lemma}\label{monomcoverlem}
Assume that $f\:X\to B$ is a Kummer logarithmically \'etale cover, $B\in\bB$ and $\cI\subseteq\cO_B$ is an ideal. Then $\cI\cO_X$ is monomial if and only if $\cI$ is monomial.
\end{lemma}
\begin{proof}
Monomiality of ideals is an \'etale-local property, and it can be checked at formal completions. For a fine enough \'etale cover of $B$ the completed local rings are of the form $l\llbracket t_1\. t_n\rrbracket\llbracket P\rrbracket$. Therefore the assertion follows from Lemma~\ref{monomlem2} below.
\end{proof}

\begin{lemma}\label{monomlem2}
Assume that $P\into Q$ is a Kummer embedding of sharp fs monoids, $O$ is a ring, $A=O\llbracket P\rrbracket$ and $C=O\llbracket Q\rrbracket$. Then an ideal $I\subset A$ is monomial if and only if $J=IC$ is monomial.
\end{lemma}
\begin{proof}
Only the inverse implication needs a proof, so assume that $J$ is monomial. If $f=\sum_{p\in P} f_pu^p$ is an element of $I$, then each $u^p$ lies in $J$, say, $$u^p\ \ =\ \ \sum c_i\cdot g_i$$ with $g_i\in I$ and $c_i\in C$. Note that $C$ is $Q^\gp/P^\gp$-graded and the component of weight zero is $A$. Taking the weight-zero component of the above equality we obtain that $$u^p\ \ =\ \ \sum (c_i)_0\cdot g_i$$ with $(c_i)_0\in A$, and hence $u^p\in I$. So $I$ is monomial, as claimed.
\end{proof}

\subsubsection{Integral closure and saturation}
We refer to \cite[\S4.3.1]{ATW-principalization} for the definition of the integral closure $I^\nor$ of an ideal and to \cite[Corollary~5.3.6]{AT1} for the following result.

\begin{lemma}\label{norsatlem}
If $Y\in\bB$ and $\cI$ is a monomial ideal on $Y$, then $\cI^\nor=\cI^\sat$.
\end{lemma}

\subsubsection{Special logarithmic schemes}\label{specialsec}
Let $Y$ be a logarithmic scheme and let $i\:U=Y_\triv\into Y$ denote the maximal open subscheme on which the logarithmic structure is trivial. We say that $Y$ is {\em special} if the underlying scheme is normal and $M_Y=i_*\cO^\times_{U\et}\cap\cO_{Y\et}$. The latter condition simply means that $M_Y\into\cO_Y$ and any $f\in \cO_Y$ dividing a monomial is monomial.

\begin{lemma}\label{stablem}
(i) Any logarithmic stack in $\bB$ is special.

(ii) If $Y\to B$ is logarithmically regular and $B\in\bB$, then $Y\in\bB$.
\end{lemma}
\begin{proof}
The first claim is proved in \cite[Theorem 11.6]{Kato-toric}. The second one follows from Lemma \ref{logreglem}.
\end{proof}

\begin{remark}\label{Brem}
(i) In fact, our choice of $\bB$ is only dictated by Lemmas~\ref{closlem}, \ref{monomcoverlem}, \ref{norsatlem} and \ref{stablem}. Any wider class for which these hold would work as well.

(ii) At the very least, the class $\bB$ can be enlarged to include logarithmic schemes $B$ such that \'etale locally on $B$ formal completions are of the form $\hatcO_{Y,y}=A\llbracket P\rrbracket$, where $A$ is a normal complete local ring, $P$ is a sharp fs monoid and the logarithmic structure is induced by $P$. We do not pursue such generality because of a current lack of potential applications.
\end{remark}

\subsection{Gradings}\label{intsec}
Throughout \S\ref{intsec}, $\phi\:P\into Q$ is an injective homomorphism of sharp fs monoids. In addition, we assume that $\phi$ is {\em exact}, that is, $P=Q\cap P^\gp$.

\subsubsection{The monoid $\tilQ$}\label{the-monoid-tilQ}
In Section \ref{monomsec}, we will not use the sharp factorizations and the notation $\tilP$ from Section~\ref{basicsec}, but we set $\tilQ^\gp=Q^\gp/P^\gp$ instead and denote the image of $Q$ in it by $\tilQ$. Note that $(\tilQ)^\gp=\tilQ^\gp$, so there is no ambiguity in the notation. For any $\tilq\in\tilQ$ let $Q_\tilq$ denote the preimage of $\tilq$ in $Q$. It is a $P$-set such that $P^\gp$ acts transitively on $Q_\tilq+P^\gp$.

The following result is certainly not new. In particular, it can be deduced from the results of \cite{Bruns-Gubeladze}. Since we could not find a reference, we provide a sketch of the proof.

\begin{lemma}\label{nonintlem}
Keep the above notation and let $\tilq\in\tilQ$. Then

(i) The $P$-set $Q_\tilq$ is finitely generated: there exist elements $q_1\.q_n\in Q_\tilq$ such that $Q_\tilq=\cup_i(q_i+P)$.

(ii) There exist $p\in P$ and $q\in Q_\tilq$ such that $p+Q_\tilq\subseteq q+P$.
\end{lemma}
\begin{proof}
Note that (ii) follows from (i) and the fact that translations of a strictly convex cone cover the whole space. To prove (i) we consider the real cone $P_\RR$ spanned by $P$ in $P^\gp_\RR:=P^\gp\otimes\RR$. Note that $P_\RR=Q_\RR\cap P^\gp_\RR$ by the exactness of $P\into Q$. Fix a lift $q\in Q$ of $\tilq$. Then $Q_\tilq$ is the intersection of the lattice $q+P^\gp$ with the section $X=Q_\RR\cap(q+P^\gp_\RR)$ of $Q_\RR$. Exactness implies that $X$ is a polyhedron with recession cone $P_\RR=Q_\RR\cap P^\gp_\RR$, hence by \cite[Theorem~1.2.7 and Proposition~1.2.8]{Bruns-Gubeladze} $X=P_\RR+X_0$ for a polytope $X_0$. Let $\{v_1\.v_m\}$ be a set of generators of $P$, consider the bounded set $X_1=X_0+Y$ with $Y :=\sum_{i=1}^m[0,1)v_i$, and consider the finite set $X_1\cap(q+P^\gp)=\{q_{1}\.q_n\}$. Clearly, the set $q_1\.q_n$ generates $Q_\tilq$.
\end{proof}

\begin{figure}[htb]
\input{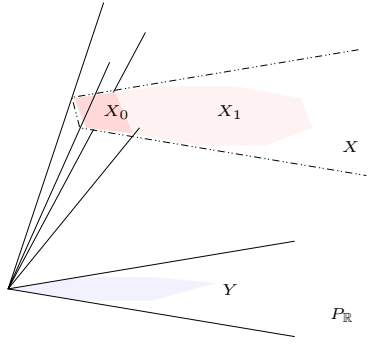}
\caption{The sets $X_0, Y$ and  $X_1$ (with $\tilQ$ being the vertical axis).  The set $X_1$  contains $P$-representatives of $Q_\tilq=X\cap Q$.}
\label{Fig:not-rigid}
\end{figure}

\begin{remark}\label{inthom}
It is well known (e.g. see \cite[Proposition~I.4.6.9]{Ogus-logbook}) that $\phi$ is an integral homomorphism if and only if each $Q_\tilq$ is of the form $q+P$. In this case $q$ is unique, hence one obtains a set-theoretic section $\tilQ\into Q$ and a bijection $\tilQ\times P=Q$. The above lemma is a technical tool, which will allow us to work with morphisms which are exact but not necessarily integral in a similar fashion.
\end{remark}


\subsubsection{The grading of $O_P[Q]$}
If $O$ is a ring and $u\:P\to (O,\cdot)$ is a homomorphism, then $A=O_P[Q]$ acquires a natural $\tilQ^\gp$-grading $A=\oplus_{\tilq\in \tilQ}A_\tilq$ trivial on $O$. One can define it directly or simply take the grading corresponding to the action of $T_\phi$ on $X_P[Q]$, where $X=\Spec(A)$. By Lemma~\ref{nonintlem}(i), each $A_\tilq$ is an $O$-module generated by finitely many elements $u^q$ with $q\in Q_\tilq$.

\subsubsection{The formal grading of $O_P\llbracket Q\rrbracket$}
If $O$ is local with maximal ideal $m$, then $$n=m\oplus\left(\oplus_{0\neq\tilq\in \tilQ}A_\tilq\right)$$ is a maximal ideal of $A$, and the following explicit description of the formal completion $\hatO_P\llbracket Q\rrbracket$ immediately follows from \cite[Prposition~4.5.6]{ATLuna}. Notice that a formal grading is a product of $\hatO$-modules rather than a sum, and it makes sense only when all non-trivially graded components lie in the maximal ideal.

\begin{lemma}\label{formalgradedcor}
If $O$ is local and $A=O_P[Q]$, then $\hatA=\hatO_P\llbracket Q\rrbracket$ acquires a formal grading $\hatA=\prod_{\tilq\in \tilQ} \hatA_\tilq$, where each $\hatA_\tilq$ is the $m$-adic completion of $A_\tilq$. In particular, $\hatA_\tilq$ is an $\hatO$-module generated by finitely many elements $u^q$ with $q\in Q_\tilq$.
\end{lemma}

\subsubsection{Homogeneous ideals}\label{gridealsec}
An ideal $\hatcJ$ in $\hatA$ is (formally) {\em homogeneous} if it is of the form $\prod_{\tilq\in \tilQ}\hatcJ_\tilq$. This happens if and only if for any $f\in\hatcJ$ with homogeneous decomposition $f=\sum_{\tilq\in \tilQ}f_\tilq$, all homogeneous components $f_\tilq$ belong to $\hatcJ$.

\subsection{$\cD$-saturated and $B$-monomial ideals}\label{stablesec}
Until the end of Section \ref{monomsec} we are  given  $f\:X\to B$ such that
\begin{itemize}
\item[A1.] $f\:X\to B$ is a logarithmic orbifold, and
\item[A2.] $B\in\bB$.
\end{itemize}

\subsubsection{$\cD$-saturatedness}\label{Sec:saturated}
We say that an ideal $\cJ\subseteq\cO_X$ is {\em $\cD$-saturated over $B$ at a geometric point $x\to X$} if $\cD^{(\le 1)}_{X/B}(\cJ_x)=\cJ_x$, and $\cJ$ is {\em $\cD$-saturated over $B$} if it is $\cD$-saturated at all geometric points of $X$. When $B$ is clear from the context we will simply say that $\cJ$ is {\em $\cD$-saturated}. For example, if $X\to B$ is logarithmically \'etale, then any ideal is $\cD$-saturated. We will use the term \emph{$\cF$-saturated} if the property holds with respect to a submodule $\cF \subseteq \cD^{(\le 1)}_{X/B}$. The following immediate result will be our main source of producing $\cD$-saturated ideals:

\begin{lemma}\label{inftylem}
For any ideal $\cJ\subseteq\cO_X$, the ideal $\cD^\infty_{X/B}(\cJ)$ is the minimal $\cD$-saturated ideal that contains $\cJ$.
\end{lemma}


\subsubsection{Why the monomialization theorem is needed}
In some sense, $\cD$-saturated ideals are ``impenetrable'' for our principalization algorithm and its main tool, the sheaf $\cD$. So, the principalization algorithm can only treat a $\cD$-saturated ideal $\cJ$  very primitively: it would blow up the monomial saturation $\cM(\cJ)$. This works fine when $\cJ$ is monomial, but usually fails when it is not. Our main result about $\cD$-saturated ideals will be that they can be made monomial by an appropriate logarithmic blowing up of $B$. In particular, if $B$ is of dimension at most one, then $\cD$-saturated ideals are automatically monomial, once $B$ is provided with an appropriate logarithmic structure.

\subsubsection{Ideals defined over the base}
Given a ring homomorphism $R\to A$ we say that an ideal $J\subseteq A$ is {\em defined over $R$} or {\em $R$-defined} if it is generated by elements of $R$. Geometrically this means that the subscheme $V(J)\into\Spec(A)$ is the pullback of a subscheme of $\Spec(R)$. We say that $\cJ\subseteq\cO_X$ is {\em $B$-defined at} a geometric point $x\to X$ if $\cJ_x$ is defined over $\cO_{f(x)}$.

\begin{remark}
Simple examples with \'etale homomorphisms $R\to A$ show that even if $J$ is $R$-defined at every point of $\Spec(A)$, it does not have to be defined over $R$ globally. The local notion will be more important for us.
\end{remark}

\subsubsection{$B$-monomial ideals}\label{def:B-monomial}
We will also need a logarithmic analogue of being $B$-defined. For simplicity, we introduce this notion only when $f$ is sharp (\S \ref{Sec:sharp-chart}) at all geometric points of $X$ and $B$ is a logarithmic schemes whose logarithmic structure is Zariski. Then, an ideal $\cJ\subseteq\cO_X$ is {\em $B$-monomial at $x$} if $\cJ_x$ is generated by a homogeneous ideal of the $\tilQ^\gp$-graded ring $(\cO_b)_P[Q]$, where $b=f(x)$, $P=\oM_b$ and $Q=\oM_x$. We say that $\cJ$ is {\em $B$-monomial} if it is $B$-monomial at all geometric points. (This is a local property, in contrast to being $B$-defined.)

\subsubsection{The strategy}
Relative logarithmic derivations take any homogeneous ideal of the $\tilQ^\gp$-graded ring $(\cO_b)_P[Q]$ to itself. In particular, if $\cJ$ is $B$-monomial, then it is $\cD$-saturated. The converse is not true in general, {even if $P=Q=1$. For example, if $\cO_x=\hatcO_b$ then there are no non-trivial derivations (though there are enough derivations), hence any ideal in $\cO_x$ is $\cD$-saturated, but one can easily construct examples of ideals in $\cO_x$ not defined over $\cO_b$. However, the converse does hold true} in some specific cases, including the formal case (Proposition~\ref{formalstabprop}). If $\cJ$ is $B$-monomial, then it can be easily monomialized (Lemma~\ref{secondmonom}). Loosely speaking, our approach to proving the monomialization theorem is to control the difference between a $\cD$-saturated ideal $\cI$ and its maximal $B$-monomial subideal. This will be done by studying a formal case and applying an appropriate localization procedure on $B$.

\subsection{Descent}\label{formsec}

\subsubsection{Notation and assumptions}\label{assumsec}
Let $f\:X\to B$ satisfy Assumptions A1--A2 of \S\ref{stablesec}. Until the end of \S\ref{formsec}, we study a $\cD$-saturated ideal $\cJ\subseteq\cO_x$ and work locally on $B$ and \'etale-locally on $X$. So fix a point $b\in B$ and a geometric point $x\to X$ over it, and let $P=\oM_b$, $Q=\oM_x$. In addition, until \S\ref{complsec} we make the following assumptions:
\begin{itemize}
\item[A3.] $f$ has abundance of derivations,
\item[A4.] $k(b)$ is algebraically closed,
\item[A5.] the logarithmic structure at $b$ is Zariski,
\item[A6.] $f$ is exact, and
\item[A7.] $Q^\gp/P^\gp$ is torsion free.
\end{itemize}
Recall that the exactness at $x$ means that the homomorphism $P\to Q$ is exact (\S\ref{sharpfactorsec}, \S\ref{intsec}).

\subsubsection{Formal descent - trivial logarithmic structure}
We start with studying the situation on formal completions. When the logarithmic structure is trivial we have:

\begin{proposition}\label{redprop}
Assume that $R$ is a complete noetherian local ring with residue field $k=R/m$ of characteristic zero and $l/k$ is a field extension such that $k$ is algebraically closed in $l$. Let $k\into R$ be a field of definition and $A=R\llbracket t_1\. t_n\rrbracket\wtimes_{k}l$. Let $\cF\subseteq\cD_{A/R}$ be a submodule of the module of $R$-derivations of $A$ which contains $\partial_{t_1}\.\partial_{t_n}$ and a dense submodule $\cF_0\subseteq\cD_{l/k}$. Then an ideal $J\subseteq A$ is $\cF$-stable if and only if it is defined over $R$.
\end{proposition}
\begin{proof}
Only the direct implication needs a proof, so assume that $J$ {is $\cF$-stable}. The general claim follows by applying successively two cases: (a) $A=R\llbracket t\rrbracket$, (b) $n=0$.

In case (a), choose $h=\sum_i a_it^i\in J$. Since ideals of $A$ are closed, {the series} $h$ lies in the ideal generated by the monomials $a_nt^n$; it suffices to prove that $a_nt^n\in J$ for {every} $n$, because then $a_n=\partial_t^n(a_nt^n)/{n!}\in J$. Choose $N>n$ and set $h_N=\sum_{i=0}^N a_it^i$. Since $a_it^i$ is an eigenvector of $\partial=t\partial_t$ of eigenvalue $i$ and the characteristic is zero, there exists an operator $\delta=\sum_{j=0}^Nl_j\partial^j$ such that $\delta(h_N)=a_nt^n$. (Use Vandermonde for $0\. N$.) Therefore $J$ contains the element $\delta(f)=a_nt^n+\delta(h-h_N)\in a_nt^n+t^{N+1}A$, and using that $J$ is closed and $N$ is arbitrary, we obtain that $a_nt^n\in J$.

In case (b), fix a basis $S$ of $l$ over $k$. Then any element $h\in A$ possesses a unique presentation $h=\sum_{e\in S} a_e e$, where $a_e\in R$; by definition of the completed tensor product $\wtimes$, for any $n\in\NN$, there exists a finite subset $T_n\subseteq S$ such that if $e\in S\setminus T_n$, then $a_e\in m^n$. We will prove the claim by first reducing it to a very particular case.

Step 1. {\it It suffices to establish the case when $J=\cF^\infty(h)$.} Indeed, choose $g\in J$. Then $J$ contains the $\cF$-stable ideal $I=\cF^\infty(g)$ generated by $g$, and it suffices to prove that $g\in(I\cap R)A$.

Step 2.  {\it It suffices to consider the case where $h=\sum_{e\in T}a_ee$ for a finite subset $T\subset S$.} Indeed, if this case is proved, then in the general case setting $h_n=\sum_{e\in T_n}a_ee$, we obtain that $a_e\in\cF^\infty(h_n)$ for any $e\in T_n$. Thus, for any $n$ and $e\in T_n$ there exists an operator $\delta_{e,n}$ with $a_e=\delta_{e,n}(h_n)$. Since $m_A=mA$ is preserved by any derivation in $\cD_{A/R}$, it follows that any $m_A^n$ is also preserved by $\delta$, and hence the sequence $\delta_{e,n}(h)$ tends to $a_e$. Since $J$ is closed, we obtain that $a_e\in J$ and, moreover, $J$ is generated by these elements.

Step 3. {\it It suffices to consider the case when $\trdeg(l/k)<\infty$ and $\cF=\cD_{A/R}$.} By steps 1--2 we can assume that $J=\cF^\infty(h)$, $h=\sum_{e\in T}a_ee$ and $|T|<\infty$. Let $l'=k(T)$ and $A'=R\wtimes_kl'$. If $t_1\.t_n$ form a transcendence basis of $l'/k$, then there exists a dual family of derivations $\partial_i\in\cF_0$, and their restrictions form an $l$-basis of $\cD_{l'/k}$ and hence also an $A$-basis of $\cD_{A'/R}=\cD_{l'/k}\otimes_{l'}A'$. Therefore, $J\cap A'$ contains the ideal $\cD^\infty_{A'/R}(h)$ and it suffices to show that the latter is defined over $R$.

Step 4. {\it It suffices to consider the case when $\trdeg(l/k)=1$.} Find a tower $k=l_0\subset l_1\subset\ldots l_n=l$ such that $l_i$ is algebraically closed in $l_{i+1}$ and $\trdeg(l_{i+1}/l_i)=1$. It suffices to establish descent from $R\wtimes_kl_{i+1}$ to $R\wtimes_kl_i$, hence the claim reduces to the case of extensions of transcendence degree one..

Step 5. {\it If the conditions of steps 1--4 are satisfied, then the claim holds true.} We can assume that $h=\sum_{i=1}^n{a_ie_i}$ with $a_i\in R$ and $e_1\.e_n\in l$ linearly independent over $k$. It suffices to prove that there exists an operator $\delta\in\cD^\infty_{l/k}$ such that $\delta(e_1)=1 $ and $\delta(e_i)=0$ for $i>1$, because then $a_1=\delta(h)\in J$, and similarly for other $a_i$.

Let $\{x\}$ be a transcendence basis of $l/k$, then $\{\partial_x\}$ is a basis of $\cD_{l/k}$. Since $k$ is algebraically closed in $l$, the kernel of $\partial_x$ is $k$. Therefore the operator $\partial_x\cdot e_n^{-1}$ {annihilates} $e_n$  and takes $e_1\.e_{n-1}$ to a family $$e_{i}^{(1)}:= \partial_x(e_{i}/e_n)$$ linearly independent over $k$. Iterating this as $${\delta'}:= \ldots\partial_x\cdot(e_{n-1}^{(1)})^{-1}\cdot\partial_x\cdot e_n^{-1}$$ we obtain an operator vanishing on $e_2\.e_n$, but $\delta'(e_1)\neq 0$. It remains to take $\delta = (\delta'(e_1))^{-1} \cdot\delta'$.
\end{proof}

\subsubsection{Logarithmic descent to $\hatO_b$}
Let $A=\hatcO_x$, $O=\hatO_b$, $k=k(b)$, $l=k(x)$, and $R=O_P\llbracket Q\rrbracket$. Note that $A=R\llbracket t_1\. t_n\rrbracket\wtimes_{k}l$ by Lemma~\ref{logregchart}. We consider the submodule $\cF=\cD_{X/B,x}\otimes_{\cO_x}A$ of $\cD_{A/O}$. Recall that $\cD_{A/O}$ was described in Lemma~\ref{formderlem}, and $\cF$ contains $\partial_1\.\partial_n$, $\Hom(\tilQ^\gp,A)$ and a dense submodule of $\cD_{l/k}$ by Lemma~\ref{standardderiv}.

\begin{proposition}\label{formalstabprop}
Keep the above notation. Then for an ideal $J\subseteq A$ the following conditions are equivalent:

(i) $J$ is $\cF$-saturated over $O$ (\S\ref{Sec:saturated}),

(ii) $J$ is $O$-monomial (\S\ref{def:B-monomial}),

(iii) $J$ is generated by finite sums of the form $a_\tilq=\sum_{q\in Q_\tilq}a_qu^q$, with $\tilq\in \tilQ$ and $a_q\in O$.
\end{proposition}
\begin{proof}
Clearly, (ii) and (iii) are equivalent. It follows from the description of $\cD_{A/O}$ that any element $a_\tilq$ is an eigenvector of $\cD_{A/O}$, hence (iii) implies (i). Conversely, assume that $J$ is $\cD$-saturated and let us prove (ii).
By Proposition~\ref{redprop}, $J$ is defined over $R$, reducing our problem to the case $A=R$.
Recall that $R$ is graded, and an element is homogeneous if and only if it is of the form $\sum_{q\in Q_\tilq}a_qu^q$. So we need to prove that $J$ is homogeneous, and since $J$ is closed, it suffices to show that if $h\in J$ and $h=\sum_{\tilq\in \tilQ}h_\tilq$ is its homogeneous formal decomposition (\S\ref{gridealsec}), then each $h_\tilq$ is in $J$.

The $R$-algebra $\cD^\infty_{R/O}$ is generated by logarithmic derivations $\partial_\phi$ with $\phi\in\Hom(\tilQ^\gp,R)$. By the assumptions of \S\ref{assumsec}, $\tilQ^\gp$ has no torsion, and hence these derivations separate elements of $\tilQ$. Using Vandermonde as in the proof of Proposition \ref{redprop}, it follows that for any finite subset $S\subset \tilQ^\gp$ and an element $\tilq\in S$, there exists a differential operator $\partial_S\in\cD^\infty_{R/O}$ such that $\partial_S(\tilq)=\tilq$ and $\partial_S(\tilq')=0$ for any other $\tilq'\in S$. Let $h_S=\sum_{\tilq\in S}h_\tilq$. For any $N>0$, taking $S$ large enough we achieve that $h-h_S\in m^N$ and then $\partial_S(h)=h_\tilq+\partial_S(h-h_S)\in h_\tilq+m^{N-1}$. So, $h_\tilq\in J$, as claimed.
\end{proof}

\subsubsection{Descent to $\cO_b$}
So far we have used $\cO_b$-derivations to show that $\cF$-saturated ideals on $\hatcO_x$ are defined over $\hatcO_b$. We would like to know when they are defined over $\cO_b$ itself, but derivations are not useful anymore, because they vanish on $\hatcO_b$ by continuity. The only tool which we have at our disposal is that open ideals are automatically defined over $\cO_b$. Surprisingly, this obvious fact can be used to make non-trivial conclusions. Here is the first step:

\begin{proposition}\label{algebrth}
Let $f\:X\to B$, $b\in B$, $x\to X$, and $\cJ\subseteq\cO_x$ be as in \S\ref{assumsec}. Assume that there exists an ideal $\cI\subset\cO_b$ such that $\cJ$  is monomial over the complement of $V(\cI)$ and $\hatcJ$ is $\hatcO_{b,\cI}$-monomial, where $\hatcJ$ and $\hatcO_{b,\cI}$ are the completions of $\cJ$ and $\cO_b$ along $m_x\cJ$ and $\cI$, respectively. Then $\cJ$ is $\cO_b$-monomial.
\end{proposition}
We proved earlier that $\hatcJ$ is $\hatcO_b$-monomial, where $\hatcO_b$ is the $m_b$-adic completion. The proposition assumes that $\hatcJ$ is monomial with respect to the slightly smaller ring $\hatcO_{b,\cI}$.
\begin{proof}
We can replace $X$ by $\Spec(\cO_x)$. Let $T=V(\cJ)$ be the closed subscheme defined by $\cJ$. Set $B_0=B\setminus V(\cI)$, $X_0=X\times_BB_0$ and $T_0=T\times_XX_0$, and let $T'$ be the schematic closure of $T_0$ in $X$. Since $T_0$ is monomial by our assumption and $X$ is logarithmically regular by Lemma~\ref{logreglem}, $T'$ is also monomial by Lemma~\ref{closlem}. Let $\cJ'\subseteq\cO_x$ be the monomial ideal defining $T'$ at $x$, say $\cJ'=u^{J'}\cO_x$ for an ideal $J'\subseteq Q$. By the construction, $\cJ\subseteq\cJ'$ and $\cJ'/\cJ$ is annihilated by a power $\cI^l$ of $\cI$. In particular, we have inclusions of the $m_x$-adic completions $\cI^l\hatcJ'\subseteq\hatcJ\subseteq\hatcJ'$ in $\hatcO_x$.

Set $C=\cO_b$ and $\hatC=\hatcO_{b,\cI}$ for brevity. By assumption, $\hatcJ$ is generated by homogeneous elements $a_i\in\hatC_P[Q]$. Since each $a_i$ lies in the monomial ideal $\cJ'\hatC_P[Q]$, one can present it as $\sum x_ru^{q_r}$ with $q_r\in J'$. Expanding $x_r$ in $\hatC_P[Q]$ we obtain a presentation $a_i=\sum_j a_{ij}u^{q_{ij}}$, where $a_{ij}\in\hatC$ and $q_{ij}\in J'$. Moreover, the degree of the homogeneous element $a_i$ is an element $\tilq_j\in\tilQ$, and all $a_{ij}$ vanish except those with $q_{ij}\in Q_{\tilq_i}$, {using the notation of Section \ref{the-monoid-tilQ}}. Thus we can assume that $q_{ij}\in Q_{\tilq_i}$.

Consider approximations $b_{ij}\in C$ such that $a_{ij}-b_{ij}\in\cI^{l+1}\hatC$, and set $b_i=\sum_j b_{ij}u^{q_{ij}}$. Then $a_i-b_i\in\cI^{l+1}u^J\hatC\subset\cI\hatcJ\subseteq m_x\hatcJ$ in $\hatcO_x$, and hence $b_i$ also generate $\hatcJ$ by Nakayama's lemma. By flatness of the completion, $b_i$ are generators of $\cJ$ in $\cO_x$, and it remains to notice that $b_i$ are homogeneous elements of $C_P[Q]$.
\end{proof}

\subsubsection{The summary}
We summarize \S\ref{formsec} in the following result.

\begin{proposition}\label{algebrcor}
Let $f\:X\to B$ and $b\in B$ satisfy Assumptions A1--A6 as in \S\ref{assumsec}, $\cJ\subseteq\cO_X$ a $\cD$-saturated ideal, $B'\to B$ a modification, $X'=X\times_BB'$ and $\cJ'=\cJ\cO_{X'}$. Assume that $b$ is closed, $B'\in\bB$ and the restriction of $\cJ'$ to $X'\times_BB_0$ is monomial, where $B_0=B\setminus\{b\}$. Then $\cJ'$ is $B'$-monomial at any geometric point $x'\to X'$.
\end{proposition}
\begin{proof}
This is only non-trivial when $x'$ is in the preimage of $b$. So, assume this is the case, and let $b'\in B'$ and $x\to X$ denote the images of $x'$. It suffices to check that $\cJ'$ and $X'\to B'$ satisfy the assumptions of Proposition~\ref{algebrth} at $x'$ for the ideal $\cI=m_b\cO_{b'}$. First, the complement of $V(\cI\cO_{X'})$ coincides with $X'\times_BB_0$, hence the restriction of  $\cJ'$ to $V(\cI\cO_{X'})$ is monomial. Second, $\hatcJ$ is $\hatcO_b$-monomial by Proposition~\ref{formalstabprop}, and since the homomorphism $\hatcO_b\to\hatcO_{b'}$ factors through the completion $\hatcO_{b',\cI}$ along $\cI$, we obtain that $\hatcJ'$ is $\hatcO_{b',\cI}$-monomial.
\end{proof}

\subsubsection{A complement in dimension $\leq 1$}\label{complsec}
The above results can be essentially strengthened when the base is at most one-dimensional, in particular, Assumptions A3--A7 are not needed. This case is used in the proof of Theorem \ref{monomialization}.

\begin{proposition}\label{logorbprop}
Assume that $f\:X\to B$ is a logarithmic orbifold and $B=\Spec(O)$ is local, logarithmically regular, of dimension at most one, and with the logarithmic structure $O\setminus\{0\}$. Then any $\cD$-saturated ideal $\cJ\subseteq\cO_X$ is monomial.
\end{proposition}
\begin{proof}
Note that $O$ is either a field or a DVR and fix a chart $u\:P\to O$, where either $P=0$ or $P=\NN$ and $u=u^1$ is a uniformizer of $O$. The claim is \'etale local on $X$, hence we can assume that $X=\Spec(A)$ is local and the logarithmic structure is Zariski. Replacing $B$ by the generic point if necessary, we can assume that $f$ is surjective, and then $P\to O$ extends to a sharp chart $Q\to A$.

Since $\cJ$ is determined by its completion $\hatcJ\subseteq\hatA$, it suffices to prove that the latter is monomial. Consider the residue fields $k=O/m_O$ and $l=A/m_A$, fix a field of coefficients $k\into\hatO$ and set $R=\hatO\wtimes_kl$. Then $\hatA=R_P\llbracket Q\rrbracket\llbracket t_1\.t_n\rrbracket$ by Lemma~\ref{logregchart}, and by Propositions~\ref{redprop} and \ref{formalstabprop}, the ideal $\hatcJ$ is $R$-monomial. It suffices to prove that $u^p\hatcJ$ is monomial for an appropriate $p\in P$, hence by Lemma~\ref{nonintlem} we can assume that $\hatcJ$ is generated by non-zero elements of the form $ru^q$ with $r\in R$. It remains to note that $R$ is a field or a DVR and $m_R=m_OR$, hence $r$ is a monomial.
\end{proof}

\subsection{Modifications of logarithmic schemes}\label{modifsec}

\subsubsection{Modifications and blow ups}
By a {\em modification} of logarithmic schemes we mean any morphism $h\:Y'\to Y$ which is proper and restricts to an isomorphism of dense open subschemes. We say that $h$ is a {\em blow up along} an ideal $\cI\subseteq\cO_Y$ if the underlying morphism is the {blowing up along} $\cI$, the ideal $\cI\cO_{Y'}$ is monomial and $h$ is an isomorphism over $Y\setminus V(\cI)$.

\begin{remark}
The blow up of a logarithmic scheme is a bit of an artificial notion. In particular, it is not uniquely determined by $\cI$, as the following example shows. Nevertheless, it will be convenient to use this notion in the paper.
\end{remark}

\begin{example}
(i) Let $Y=\Spec(k[x])$ with the trivial logarithmic structure and let $\cI=(x)$. Then the blow up along $\cI$ is the same underlying scheme provided with the logarithmic structure generated by $x^\NN$. In this specific case, $Y'$ is also logarithmically smooth over $k$, while the morphism $Y'\to Y$ is not logarithmically smooth, though it is an isomorphism on the level of schemes.

(ii) With our definition, $Y'\to Y$ is also a blow up along $\cI_1=(x^2)$, though the natural (or universal) blow up $Y'_1$ along $\cI_1$ is the scheme $Y$ provided with the logarithmic structure generated by $x^{2\NN}$. It is not logarithmically smooth over $k$.
\end{example}

\subsubsection{Supports}
We will also want to control supports of blow ups. Assume given a morphism $g\:Y\to S$ and a closed subset $T\subseteq S$. A modification $h\:Y'\to Y$ is called a {\em $T$-modification} if it induces an isomorphism over $g^{-1}(S\setminus T)$. In particular, this is the case when $h$ is a {\em $T$-supported} blow up, that is, a blow up along $\cI$ such that $V(\cI)\subseteq g^{-1}(T)$.

\subsubsection{The class $\bB^\st$}\label{obBsec}
The class $\bB$ of \S\ref{Bsec} is not a priori stable under modifications, and we stabilize it as follows: the class $\bB^\st$ consists of logarithmic stacks $Y\in\bB$ such that for any closed $T\subseteq Y$ and a $T$-modification $Y'\to Y$ there exists a $T$-supported blow up $Y''\to Y$ which factors through $Y'\to Y$ and satisfies $Y''\in\bB$. In fact, it follows from Theorem~\ref{absdesingth} that $\bB=\bB^\st$, but until we prove it we have to distinguish the two classes.

\subsubsection{Basic properties of blow ups}
It is easy to see that the well-known properties of usual blow ups imply their logarithmic analogues: $T$-supported blow ups of logarithmic stacks are preserved by compositions, are compatible with flat morphisms, can be extended from open substacks, and are cofinal in the family of all $T$-modifications. The latter property is a version of Chow's lemma, which was extended to DM stacks by Rydh, see \cite[Corollary~5.1]{Rydh}.

\subsubsection{Kummer extension of cofinality}
We will also need the following subtle property. The main argument of its proof is based on relative RZ spaces and requires that we recall some material not related to anything else in the paper. It is given in appendix \ref{RZappend}, and here we only deduce a corollary about logarithmic schemes.

\begin{proposition}\label{cofinalprop}
Assume that $h\:Y\to Z$ is a Kummer logarithmically \'etale covering with $Z\in\bB^\st$ and $T\subseteq Z$ is a closed subset. Then for any $T$-modification $\tilY\to Y$ there exists a $T$-supported blow up $Z'\to Z$ such that $Z'\in\bB$ and $Y\times_ZZ'\to Y$ factors through $\tilY$. 
\end{proposition}
\begin{proof}
We can replace $Y$ and $\tilY$ by an \'etale cover $X$ of $Y$ and the $T$-modification $\tilX=X\times_Y\tilY$. In this way we can achieve that $h$ factors into the composition of a Kummer cover of logarithmic \emph{schemes} $h_1\:Y\to Y_0$ and a strict \'etale cover \emph{of stacks} $h_0\:Y_0\to Z$. So, it suffices to prove the claim for $h_0$ and $h_1$, and we restrict to two cases: (1) $Y$ and $Z$ are schemes, (2) $h$ is \'etale.

Since the proofs only differ in one ingredient, we prove case (1) and then indicate the change for case (2). We can assume that $T\subsetneq Z$, as otherwise we can simply replace it by any proper closed subset $T'\subsetneq Z$ such that $\tilY\to Y$ is a $T'$-modification. Note that $U=Z\setminus T$ and its preimage $V=Y\setminus h^{-1}(T)$ lie in $\bB$ and hence are normal. By Theorem~\ref{cofinalth} applied to $Y\to Z$ and $V\to U$, there exists a $T$-supported blow up of schemes $g\:Z'=Bl_\cI(Z)\to Z$ and a finite $T$-modification $Y'$ of the closure of $V$ in $Z'\times_ZY$ such that the $T$-modification $Y'\to Y$ factors through $\tilY$. If $\cI'$ is the ideal of $T$, then replacing $Z'$ by $Bl_{\cI\cI'}(Z)$ we obtain a finer blow up whose center is a closed subscheme with reduction $T$. Choosing an appropriate logarithmic structure on $Z'$ so that $\cI\cI'\cO_Z$ is monomial we obtain an enrichment of $g$ to a blow up of logarithmic schemes, which will also be denoted $g\:Z'\to Z$. Finally, since $Z\in\bB^\st$, we can replace $g$ by a finer blow up again so that $Z'\in\bB$.

We claim that $Z'\to Z$ is as required. The morphism $h'\:Y'\to Z'$ is Kummer logarithmically \'etale, so by Lemma \ref{stablem} the logarithmic scheme $Y'=Z'\times_ZY$ lies in $\bB$ and hence is special (\S\ref{specialsec}). In particular, its underlying scheme is normal and hence coincides with the closure of $V$ in the scheme-theoretic fiber product. Therefore, the morphism of schemes $Y'\to Y$ factors through a morphism, of underlying schemes, $\phi\:Y'\to\tilY$, and it remains to show that the latter extends to a morphism of logarithmic schemes. Since $Y'$ is special, it suffices to show that $Y'_\triv\subseteq\phi^{-1}(\tilY_\triv)$. By our construction, $Y'_\triv=h'^{-1}(Z'_\triv)$, which is the preimage of $Z_\triv\setminus T$. On the other hand, $Y_\triv=h^{-1}(Z_\triv)$ and since $\tilY\to Z$ is a $T$-modification, $\tilY_\triv$ contains the preimage of $Z_\triv\setminus T$. The claim follows.

Case (2) is proved similarly, replacing the use of  Theorem \ref{cofinalth} by the flattening theorem, extended by Rydh to stacks in \cite[Theorem~D]{Rydh}.
\end{proof}

\subsection{Monomialization of $\cD$-saturated ideals}\label{monomthsec}
Throughout \S\ref{monomthsec}, $f\:X\to B$ satisfies Assumptions A1--A2 of \S\ref{stablesec}, $\cJ\subseteq\cO_X$ is an ideal and $T\subseteq B$ is the closure of $f(V(\cJ))$.

\subsubsection{Ideals monomializable over a base}

We say that a blow up $g\:B'\to B$ {\em monomializes} (resp. {\em almost monomializes}) $\cJ$ if  the pullback $\cJ'=\cJ\cO_{X'}$ of $\cJ$ to the saturated base change $X'=X\times_BB'$ is monomial and $g$ is $T$-supported (resp. the center of $g$ is monomial outside of $T$). If such a blow up exists we say that $\cJ$ is {\em (almost) monomializable}.

\subsubsection{Descent}
It seems probable that any $\cD$-saturated ideal $\cJ$ is monomializable, and we will prove this in the end of \S\ref{monomsec} for an integral $f$. However, in the general case we will only prove that $\cJ$ is almost monomializable. Our proof will run by showing that the class of almost monomializable ideals is large enough. We start with descent.

\begin{lemma}\label{monomdescentlem}
The classes of monomializable and almost monomializable ideals are Kummer local both on the source and on the base.
\end{lemma}
\begin{proof}
The two cases are similar, so we will only deal with monomializable ideals. Say $X_1 \to X$ is a Kummer cover and $B'\to B$ a blowing up such that $\cJ\cO_{X_1'}$ is monomial as in the following cartesian diagram:
$$\xymatrix{X_1' \ar[r]\ar[d] & X_1 \ar[d] \\  X' \ar[r]\ar[d] & X \ar[d] \\ B' \ar[r] & B.}$$
By Lemma \ref{monomcoverlem} applied to ${X'_1} \to X'$ the ideal $\cJ'=\cJ\cO_{X'}$ is monomial, proving that the property is local on the source.

The second claim asserts that if $B_1\to B$ is a Kummer logarithmically \'etale cover and the pullback $\cJ_1$ of $\cJ$ to $X_1=X\times_BB_1$ is monomializable over $B_1$, then $\cJ$ is monomializable over $B$. Let $g_1\:B''_1\to B_1$ be a modification that monomializes $\cJ_1$. By Proposition~\ref{cofinalprop} there exists a $T$-supported modification $g\:B'\to B$ such that $B'_1=B'\times_BB_1\to B_1$ factors through $g_1$, 
and hence monomializes $\cJ_1$:
$$ \xymatrix{B_1'' \ar[dr]_{g_1} & B_1'\ar@{.>}[l]\ar[d]\ar[r] & B' \ar[d]^{\exists g} \\ & B_1 \ar[r] &B.}$$
We claim that $g$ monomializes $\cJ$.
Set $X'=X\times_BB'$, $\cJ'=\cJ\cO_{X'}$, $X'_1=X\times_BB'_1$ and $\cJ'_1=\cJ\cO_{X'_1}$:
$$\xymatrix{X_1' \ar[r]\ar[d] & X' \ar[d] \\  B_1' \ar[r]\ar[d] & B' \ar[d]^g \\ B_1 \ar[r] & B.}$$
As we noted, $\cJ'_1$ is monomial, and since $X'_1=X'\times_BB_1$ is Kummer logarithmically \'etale over $X'$, we obtain that $\cJ'$ is monomial by Lemma~\ref{monomcoverlem}. Thus, $g$ monomializes $\cJ$.
\end{proof}

\begin{corollary}\label{firstmonom}
If $f\:X\to B$ is Kummer logarithmically \'etale, then any ideal $\cJ\subseteq\cO_X$ is monomializable.
\end{corollary}
\begin{proof}
Kummer locally on the base $f$ is an isomorphism. By Lemma~\ref{monomdescentlem} we can assume that $X=B$ and then $\cJ$ is monomializable by a blow up along it.
\end{proof}

\subsubsection{Localization}
We will need a finer result about localization on the base.

\begin{lemma}\label{henslem}
Let $f\:X\to B$ be a logarithmic manifold, and assume that $X$ and $B\in\bB$ are logarithmic schemes. Then an ideal $\cJ\subseteq\cO_X$ is monomializable over $B$ if and only if for any point $b\in B$ with the strict henselization $B_\ob=\Spec(\cO_\ob^{sh})$ and pullback $X_\ob=X\times_BB_\ob$, the ideal $\cJ_\ob=\cJ\cO_{X_\ob}$ is monomializable over $B_\ob$.
\end{lemma}
\begin{proof}
Only the inverse implication needs a proof, so assume that each $\cJ_\ob$ is monomializable. First we claim that for any $b\in B$ with $B_b=\Spec(\cO_b)$ and $X_b=X\times_BB_b$, each $\cJ_b=\cJ\cO_{X_b}$ is monomializable. By definition, $\cO_b^{sh}$ is the filtered union of \'etale $\cO_b$-subalgebras $A_i$. Choosing $i$ large enough we can assume that the center $I\subseteq \cO_b^{sh}$ of a monomializing blow up $B'_\ob\to B_\ob$ is of the form $I_i\cO_b^\sh$ for an ideal $I_i\subset A_i$. Since $B_\ob\to Y_i=\Spec(A_i)$ is flat, $B'_\ob\to B_\ob$ is the pullback of the blow up $Y'_i\to Y_i$ along $I_i$. Furthermore, for any $j\ge i$ one can take $I_j=I_i\cO_{Y_j}$ and then $Y'_j=Y_j\times_{Y_i}Y'_i$. Since $\cJ\cO_{B'_\ob}$ is monomial, the ideal $\cJ\cO_{Y'_j}$ is monomial for a large enough $j$. This means that $\cJ\cO_{Y_j}$ is monomializable by the blow up $Y'_j\to Y_j$, and hence $\cJ_b$ is monomializable by Lemma~\ref{monomdescentlem}.

For a point $b\in B$ consider the blow up of $B_b$ that monomializes $\cJ_b$ and extend it to a blow up of $B$ simply by taking the Zariski closure of the center. 
Since being monomial is an open condition, the new blow up monomializes $\cJ$ over a neighborhood of $b$. So, there exists an open cover $X=\cup_{i=1}^n X_i$ such that $\cJ|_{X_i}$ is monomializable by a blow up $g_i\:B_i\to B$. Any blow up which dominates all $g_i$ and whose center lies in the union of the centers of $g_i$ monomializes $\cJ$.
\end{proof}

\subsubsection{$B$-monomial ideals}
Finally, let us obtain some criteria for monomializability of $B$-monomial ideals (see \S\ref{def:B-monomial}).

\begin{lemma}\label{secondmonom}
Let $f\:X\to B$ be a logarithmic manifold with $X$ and $B\in\bB$ schemes and $\cJ\subseteq\cO_X$ an ideal. Assume that $f$ is integral (\cite[Definition~III.2.5.1]{Ogus-logbook}), the logarithmic structure on $B$ is Zariski and $\cJ$ is $B$-monomial. Then $\cJ$ is monomializable. 
\end{lemma}
\begin{proof}
By Lemma \ref{monomdescentlem}, we can work \'etale locally on $X$ and $B$, hence we can assume that $X=\Spec(A)$, $B=\Spec(O)$, $f$ is dominant and possesses a chart $X\to B_P[Q]$, where $P\into Q$ is an integral embedding of sharp monoids. By definition of $B$-monomiality, localizing further we can assume that $\cJ\subseteq A$ is generated by elements $a_i=\sum_j a_{ij}u^{q_{ij}}$, where $q_{ij}\in Q_{\tilq_i}$ and $a_{ij}\in O$. By Remark~\ref{inthom} $Q_{\tilq_i}=q_i+P$, hence $q_{ij}=q_i+p_{ij}$ and $\cJ$ is generated by the elements $c_iu^{q_i}$, where $c_i=\sum_j a_{ij}u^{p_{ij}}\in O$.

Let $\cN$ be the maximal monomial subideal of $\cJ$, and choose a decomposition $\cJ=\cJ'+\cN$ with $\cJ'=(c_1u^{q_1}\.c_nu^{q_n})$ and minimal possible $n$. Then each $c_i$ is non-monomial, and by induction on $n$ it suffices to find a $T$-supported blow up $B'\to B$ such that the pullback of $\cJ$ possesses an analogous decomposition with a smaller $n$. A naive solution now would be to blow up the principal ideals $(c_i)$ on $B$, that is, to just enlarge the logarithmic structure of $B$ so that each $c_i$ becomes a monomial. However, $V(c_i)$ does not have to be contained in $T$, and we have to be more careful.

Consider the ideals $\cI'=(c_1\.c_n)\subseteq O$ and $\cN_0=\cN\cap O$. We claim that $\cN_0$ is monomial, say $\cN_0=(u^{p_1}\.u^{p_m})$ with $p_i\in P$. In fact, the claim is that the schematic image of $V(\cN)$ is a monomial subscheme of $B$. As in the proof of Lemma~\ref{closlem}, it suffices to check this for the case $X=\bfA_Q$ and $B=\bfA_P$, but then the schematic image is $\bfD_{P^\gp}$-equivariant and hence monomial.

Next we claim that the ideal $$\cI=\cI'+\cN_0=(c_1\.c_n,u^{p_1}\.u^{p_m})$$ is $T$-supported. Since $\cJ\subseteq \cI'A+\cN$, we have that $f^{-1}(V(\cI'))\cap V(\cN)\subseteq V(\cJ)$ and $$V(\cI')\cap f(V(\cN))=f(f^{-1}(V(\cI'))\cap V(\cN))\subseteq T.$$ Hence $T$ contains the closure of the left hand side, which is $V(\cI')\cap V(\cN_0)=V(\cI)$.

Let $g\:B'\to B$ be the $T$-supported blow up along $\cI$ and let $X'=X\times_BB'$ and $\cJ'=\cJ\cO_{X'}$. It suffices to prove, with the induction scheme above, that $\cJ'$ is monomializable over $B'$, and since the property is local it suffices to restrict to a single chart of $g$. There are two cases.

\begin{enumerate} \item If $B'_i$ is the chart corresponding to $c_i$, set $X'_i=B'_i\times_BX$ and $\cJ'_i=\cJ\cO_{X'_i}$. Since $c_i$ becomes monomial on $B_i$, the set $c_1 u^{q_1}\.c_lu^{q_l},u^{q_{n+1}}\.u^{q_l}$ of generators of $\cJ'_i$ contains at most $n-1$ non-monomial elements, and hence $\cJ'_i$ is monomializable by induction.
\item
If $B'_j$ is the chart corresponding to $u^{p_j}$ and $X'_j$, $\cJ'_j$ are defined as above, then $u^{p_j}$ divides each $c_i$ on $B'_j$, and since $u^{p_j}\in\cN\subseteq\cJ$ the ideal $\cJ'_j=(u^{p_j},u^{q_{n+1}}\.u^{q_l})$ is already monomial.
\end{enumerate}
\end{proof}

\begin{remark}\label{exactrem}
With an additional work Lemma~\ref{secondmonom} can be extended to the case of an exact $f$ as follows. By Lemma~\ref{nonintlem} the above argument can be run for an appropriate ideal $u^p\cJ$. In such a case, one obtains that $\cJ$ is almost monomializable and, moreover, the blow up center is principal and monomial on $B\setminus T$. In particular, $B'\to B$ is a $T$-modification, and by Proposition~\ref{cofinalprop} enlarging $B'\to B$ one can make it a $T$-supported blow up.
\end{remark}

\subsubsection{Integralization and saturization}\label{saturatingsec}
We will need the following results known to experts, and proven in different context in the literature: 
\begin{proposition} Suppose given $f:X \to B$ as above. 
\begin{itemize} 
\item There exists a logarithmic blowing up $B' \to B$ such that the  fs base change $f'\:X'\to B'$ is integral.
\item There exists a further Kummer \'etale covering Kummer \'etale covering $B''\to B'$ such that the fs base change $f''\:X''\to B''$ is saturated. 
\end{itemize}
\end{proposition}
\begin{proof} For part (1) we follow the argument of \cite[Proposition 4.4 and  Remark 4.6]{AK}. That paper is set in the language of toroidal embeddings over an algebraically closed field, but the argument is general. An almost identical result, with similar argument, in the logarithmic setting is  given as the main result of \cite{FKato}\footnote{We are informed by the author that that paper is expected to appear in revised form, with results stated in their full power, in the near future.}.

Choose a subdivision of cone complexes $\Delta_B' \to \Delta_B$, with induced subdivision $\Delta_X' \to \Delta_X$ and morphism $\Delta_f' : \Delta_X' \to \Delta_B'$ so that $\Delta_B'$ is regular and  $f'$ maps cones to cones (rather than subcones). Let $f':X' \to B'$ be the resulting  fs pullback. By \cite[Lemma 4.1]{AK} the morphism $f'$ is equidimensional. Since  $B'$ is regular and $X'$ Cohen-Macaulay, the morphism $f'$ is flat, hence integral.

For part (2), we follow the argument of \cite[Proposition 5.1]{AK} but skipping the use of Kawamata's trick\footnote{See \cite{Molcho} for a global result in the language of stacks.}. Once again that paper is set in the language of toroidal embeddings over an algebraically closed field, but the argument is general. A logarithmic argument in a different context is given in \cite[Appendix~ A]{Illusie-Kato-Nakayama}, see the proof of \cite[Theorem~3.7]{Illusie-Temkin-X} for an explanation why this applies to schemes.\footnote{We note that Tsuji's long-lost article \cite{Tsuji} is key to the logarithmic arguments.}

For each cone $\sigma\in \Delta_X'$ with image $\tau \in \Delta_B'$ corresponding to atomic open $B'_\tau\subset B'$ consider the corresponding lattice homomorphism $\phi:N_\sigma \to N_\tau.$ There is a positive  integer $m_{\sigma\tau}$ such that $\phi^{-1} (m_{\sigma\tau}N_\tau) \to N_\tau$ is surjective. Take $m_\tau = \lcm\{m_{\sigma\tau'} | \sigma \text{ maps to a face } \tau' \text{ of } \tau\}$ and choose a Kummer-\'etale morphism $B''_\tau \to B'_\tau$ corresponding to the sublattice $m_\tau N_\tau.$ Let $X''_\tau$ be the fs pullback. By \cite[Lemma 5.2]{AK} the flat logarithmically smooth morphism $X''_\tau \to B''_\tau$ has reduced fibers, hence it is saturated. Taking $B'' = \sqcup B''_\tau$ and $ X'' = \sqcup X''_\tau$ gives the result.

\end{proof}

%
%
%
%

\subsubsection{The monomialization theorem}
Finally, we can prove the main result of \S\ref{monomsec}.

\begin{theorem}\label{monomialization}
Let $f\:X\to B$ be a relative logarithmic orbifold with $B\in\bB^\st$, let $\cJ\in\cO_B$ be a $\cD$-saturated ideal with $T=\overline{f(V(\cJ))}$, and assume that either $\dim(B)\le 1$ or $f$ has abundance of derivations. Then

(i) There exists a blow up $g\:B'\to B$ with saturated base change $X'=X\times_BB'$ such that $\cJ\cO_{X'}$ is monomial.

(ii) If $B$ is a scheme, then $\cJ$ is almost monomializable: one can also achieve that the center of $g$ is monomial outside of $T$.

(iii) If $f$ is integral, then $\cJ$ is monomializable: one can also achieve that $g$ is $T$-supported.
\end{theorem}
For the sake of completeness we note that part (iii) can be extended to the case of exact morphisms using Remark~\ref{exactrem}.
\begin{proof}
{\sc Reduction to schemes.} Applying Proposition \ref{cofinalprop} to an \'etale cover of $B$ by a scheme reduces claim (i) to the case when $B$ is a scheme. 
In the sequel we assume that $B$ is a scheme, and it suffices to prove (ii) and (iii).

{\sc Reduction to saturated morphisms.} By Lemma~\ref{monomdescentlem} in the process of proof we can replace $f$ by the saturated base change $f'\:X'\to B'$ with respect to a Kummer \'etale covering $B'\to B$. In addition, while proving (ii) we can also do saturated base changes with respect to log blow ups $B'\to B$, hence by \S\ref{saturatingsec} we can assume that $f$ is saturated, in particular, integral. Thus, our task reduces to proving (iii) for saturated morphisms.

{\sc Localization.} Working \'etale locally we can assume that $X$ and $B$ are logarithmic schemes and $f$ possesses a global sharp chart $X\to B_P[Q]$. Since $f$ is saturated, $Q^\gp/P^\gp$ is torsion free.

{\sc Induction setup and reduction to strictly henselian schemes.} We proceed with few more reductions. By induction on $d=\dim(B)$ we can assume that the claim is proved for smaller dimensions. By Lemma~\ref{henslem} it suffices to prove the theorem when $B=\Spec(\cO_b^{sh})$ is local and strictly henselian. In particular, $k(b)$ is algebraically closed. We can also assume that $d=\dim(\cO_b)$.

{\sc The base case $d\le 1$.} If $d=1$ and the logarithmic structure is trivial, increase the logarithmic structure to $\cO_b\setminus\{0\}$ (that is, blow up the closed point). Then $\cJ$ becomes monomial by Proposition~\ref{logorbprop}.

{\sc The induction step.} Assume now that $d>1$ and $f$ has abundance of derivations. Set $B_0=B\setminus\{b\}$ and $X_0=X\times_BB_0$. Since $\dim(B_0)<d$,  $\cJ_0=\cJ|_{X_0}$ is monomializable by a $T$-supported blow up $B'_0\to B_0$ by the induction assumption. Consider an arbitrary extension of $B'_0\to B_0$ to a blow up $B'\to B$. By assumption of the theorem, replacing $B'$ by its $T$-supported blow up we can also assume that $B'$ is logarithmically regular. By the construction, if $f'\:X'\to B'$ denotes the base change of $f$, then the restriction of $\cJ'=\cJ\cO_{X'}$ onto $X'_0=X'\times_BB_0$ is monomial. Moreover, $\cJ'$ is $B'$-monomial at any point of $X'$ by Proposition~\ref{algebrcor}, hence it is monomializable by Lemma~\ref{secondmonom}.
\end{proof}

\section{Kummer blow ups and transforms}
\addtocontents{toc}
{\noindent Submonomial Kummer blow ups. Transforms of ideals and differential operators.}

\subsection{Submonomial Kummer ideals}\label{kumidsec}

\subsubsection{Submonomial ideals}\label{submonomialsec}
Let $f\:Y\to Z$ be a logarithmically regular morphism of logarithmic DM stacks. Then by a {\em $Z$-submonomial} ideal on $Y$ we mean any ideal $\cI\subseteq\cO_Y$ locally generated by a monomial ideal and a $Z$-submanifold ideal.
Equivalently, one has that locally $V(\cI)$ is a monomial substack of a $Z$-submanifold of $Y$. Our principalization algorithm runs by blowing up submonomial centers and more general centers involving fractional powers of monomials. To introduce them we use the Kummer topology.

\subsubsection{Kummer topology and ideals}\label{ketsec}
First, we recall generalities that apply to an arbitrary logarithmic scheme $Y$ and a morphism $Y\to Z$. Kummer \'etale morphisms and covers form a Kummer \'etale topology $Y\ket$, that we simply call the Kummer topology of $Y$, e.g. see \cite[\S5.3]{ATW-destackification} or \cite{ATW-principalization}. The presheaf $\cO_{Y\ket}$ is a sheaf and its finitely generated ideals will be called {\em Kummer ideals} on $Y$. Similarly to the argument in \S\ref{sheafdersec}, the presheaf of derivations $\cD_{Y\ket/Z}$, which assigns to a Kummer \'etale morphism $Y'\to Y$ the $\cO_{Y'}$-module $\cD_{Y/Z}(Y')$, is also a sheaf.

Often, we will view an ordinary ideal $\cI\subseteq\cO_Y$ as the Kummer ideal $\cI\ket=\cI\cO_{Y\ket}$ it generates. In other words, $\cI\ket$ is the pullback of $\cI$ with respect to $\pi\:Y_{\ket}\to Y$. This pullback operation is functorial and preserves arithmetic operations and derivations: $(\cI\cJ)\ket=\cI\ket \cJ\ket$, $(\cI+\cJ)\ket= \cI\ket+\cJ\ket$, $(\cD_{Y/Z}\cI)\ket=\cD_{Y\ket/Z}(\cI\ket)$, etc. In the sequel we will safely write $\cI$ instead of $\cI\ket$ and $\cD_{Y/Z}$ instead of $\cD_{Y\ket/Z}$. With these conventions, any Kummer ideal $\cI$ is an ideal Kummer-locally: there exists a Kummer covering $Y'\to Y$ such that $\cI|_{Y'}$ is an ordinary ideal. By the vanishing locus $V(\cJ)$ we denote the minimal closed set such that  $\cJ$ is the unit ideal on its complement.

Finally, we note that all the definitions above are local with respect to strict \'etale (and even Kummer \'etale) morphisms and hence extend to logarithmic DM stacks.

\subsubsection{Submonomial Kummer ideals}
Now, assume that $Y\to Z$ is logarithmically regular. A Kummer ideal $\cJ$ is called {\em $Z$-submonomial} (resp. {\em monomial}) if Kummer locally it is a $Z$-submonomial (resp. monomial) ideal, as in \ref{submonomialsec}.

\subsubsection{Integral closure}
The integral closure $\cJ^\nor$ of a Kummer ideal $\cJ$ is defined via Kummer \'etale sheafification of the usual integral closure, see \cite[\S4.3.1]{ATW-principalization}. This operation will only be used in the following special case: for any submonomial ideal $\cJ$ we set $\cJ^{(a)}=(\cJ^a)^\nor$ for convenience of notation.

\begin{remark}\label{norrem}
If may happen that $\cJ$ is integrally closed (even monomial and saturated) but $\cJ^a$ is not, and it will be convenient to use the ideals $\cJ^{(a)}$ in the definition of admissibility because for any ideal $\cI$ on $X$, if $x^n\subseteq\cJ^{(an)}$ then $x\subseteq\cJ^{(a)}$. This property is an immediate consequence of integral closedness.
\end{remark}

Part (2) of the following lemma is the only other result about $\cJ^{(a)}$ we will need. However, to prove this we have to describe this ideal precisely.

\begin{lemma}\label{Jnorlem}
Assume that $Y\to B$ is a logarithmically regular morphism of logarithmic DM stacks and $B\in\bB$, $\cJ\subseteq\cO_Y$ is a $B$-submonomial ideal, and $a\ge 0$.

(i) Fix a presentation $\cJ=\cI+\cN$, where $\cN$ is a monomial Kummer ideal and $\cI$ is a $B$-submanifold ideal. Then $$\cJ^{(a)}=\sum_{j=0}^a(\cN^j)^\sat\cdot\cI^{a-j}.$$

(ii) For any $i$ with $0\le i\le a$ the inclusion $\cD_{Y/B}^{(\le i)}\left(\cJ^{(a)}\right)\subseteq\cJ^{(a-i)}$ holds.
\end{lemma}
\begin{proof}
The argument from \cite[Lemma 4.3.2]{ATW-principalization} applies here too: (ii) follows from (i) because $\cD_{Y/B}^{(\le i)}$ preserves $(\cN^j)^\sat$ and takes $\cI^b$ to $\cI^{b-i}$, and (1) is proved by reducing to the monomial case as follows. First, one increases the logarithmic structure via Lemma~\ref{incrloglem} so that $\cJ$ becomes monomial. Then a direct computation shows that the sum in (i) describes the saturation $(\cJ^a)^\sat$, which by  Lemma~\ref{norsatlem} coincides with $\cJ^{(a)}$.
\end{proof}

\subsection{Submonomial Kummer blow ups}\label{submonomblowsec}
Blow ups of submonomial Kummer centers can be defined as in the absolute case, see \cite[Section~5]{ATW-destackification}. Rydh observed that one can give a more conceptual construction using a stack theoretic  Proj construction --- as also used in \cite[Section~3]{ATW-weighted}. We will follow the latter approach with one important difference: instead of the $h$-topology and corresponding ideals we will use the Kummer topology and ideals. This is possible because although our blow ups are weighted, all regular parameters have weight one. Our construction will use non-saturated logarithmic schemes as an intermediary, so this is the only section where we work with arbitrary integral logarithmic schemes.

\subsubsection{$\cProj$ construction}
As in \cite[\S3.1]{ATW-weighted}, given a scheme $Y$ and a graded $\cO_Y$-algebra $\cA=\oplus_{e=0}^\infty\cA_e$
by $\cProj_Y(\cA)$ we denote the stack-theoretic enhancement of the usual Proj construction, which is defined similarly to the latter but with the quotient by $\GG_m$ taken stack-theoretically: $$\cProj_Y(\cA)=[(\Spec_Y(\cA)\setminus V_\cA)/\GG_m],$$ where $V_\cA=V_\cA(\cA_{\ge 1})$ is the vanishing locus of the ideal $\cA_{\ge 1}=\oplus_{e>0}\cA_e$.

Note also that $\Proj$ extends to graded $\cO_{Y\et}$-algebras by flat descent. In this way one extends the operation to DM stacks (and even further to Artin stacks). We claim that $\cProj$ extends too, and the only non-trivial part is to divide the stack $Z=\Spec_Y(\cA)\setminus V$ by $\GG_m$. For this we consider a smooth groupoid in schemes $p_{1,2}\:Y_1\rightrightarrows Y_0$ with quotient $Y$ and note that $Z_i=\Spec_{Y_i}(\cA|_{Y_i})\setminus V_i$ with maps $q_{1,2}\:Z_1\rightrightarrows Z_0$ induced by $p_{1,2}$ form a groupoid with quotient $Z$. The natural actions of $\GG_m$ on $Z_i$ are compatible with $q_{1,2}$ hence they induce a strict action of $\GG_m$ on $Z$. By \cite[Theorem~4.1]{Romagny} there exists a stack-theoretic quotient $[Z/\GG_m]$, which is an Artin stack.

\begin{lemma}\label{projfun}
If $f\:Y'\to Y$ is a flat morphism, $\cA$ is a graded $\cO_Z$-algebra and $\cA'=f^*(\cA)$, then $\cProj_{Y'}(\cA')=\cProj_Y(\cA)\times_YY'$.
\end{lemma}
\begin{proof}
Clearly, $\Spec_Y(\cA')\setminus V'=(\Spec_Y(\cA)\setminus V)\times_YY'$. It remains to use that by \cite[Theorem~4.1]{Romagny}, taking the quotient by the action of $\GG_m$ is compatible with the base change with respect to $f$.
\end{proof}

\subsubsection{$\cProj$ and inertia}
The $\cProj$-construction may increase inertia of a stack, but one has a certain control.

\begin{lemma}\label{inertialem}
Let $Y$ be a DM stack, $\cA$  a graded $\cO_Y$-algebra, and $X=\cProj_Y(\cA)$.

(i) The relative inertia $I_{X/Y}$ is a finite subgroup of $(\GG_m)_X$. In particular, the relative stabilizers $I_{X/Y,\ox}$ are $\mu_n$.

(ii) Assume that $Y\to S$ is a morphism of stacks such that $I_{Y/S}$ is diagonalizable (resp. finite). Then $I_{X/S}$ is also diagonalizable (resp. finite).
\end{lemma}
\begin{proof}
(i) Recall that $X$ is obtained as the $\GG_m$-quotient of the stack $\Spec_Y(\cA)\setminus V$, which is representable over $Y$. Hence $I_{X/Y}$ is a closed subgroup of $(\GG_m)_X$. Moreover, $V$ is precisely the locus of $\GG_m$-invariant points, hence $I_{X/Y}$ is a finite subgroup.

(ii) Similarly, $I_{X/S} \subset (I_{Y/S}  \times_Y X)  \times \GG_m$ is closed, hence diagonalizable, and finite if  $I_{Y/S}$ is finite.
\end{proof}

\subsubsection{Kummer blow ups}
Let $Y$ be a logarithmic DM stack and $\pi\:Y\ket\to Y\et$ denote the natural map. For any Kummer ideal $\cI$ we consider the associated Rees algebra $\cR_\cI=\oplus_{e=0}^\infty\cI^e$, push it forward to $Y$ and apply $\cProj$. The resulting stack $\cProj_Y(\pi_*(\cR_\cI))$ will be called the Kummer blow up of $Y$ along $\cI$. If $\cI$ is a usual ideal on $Y$, then $\pi_*(\cI\cO_{Y\ket})=\cI$ and we recover the usual definition of Rees algebra of $\cI$ and blow up along $\cI$. This construction is \'etale local on $Y$ because all its ingredients, including $\pi_*$ and quotient by $\GG_m$, are \'etale local: if $f\:Z\to Y$ is \'etale, then $\pi_*(\cR_{f^{-1}\cI})=f^*(\pi_*(\cR_\cI))$, and by Lemma~\ref{projfun} $$\cProj_Z(\pi_*(\cR_{f^{-1}\cI)})=\cProj_Y(\pi_*(\cR_\cI))\times_YZ.$$

\subsubsection{The submonomial case}
The only case of Kummer blow ups used in this paper is when the center $\cJ$ is $Z$-submonomial for a logarithmically regular morphism $f\:Y\to Z$. Set $\cR'=\cR'_\cJ=\pi_*(\cR_\cJ)$ and $Y'=\cProj_Y(\cR')$ with the natural morphism $g'\:Y'\to Y$. In this case, we will also promote $Y'$ to a logarithmic stack by pulling back the logarithmic structure of $Y$ and enlarging it by the exceptional divisor. The saturation of the resulting logarithmic scheme $Y'$ will be called the {\em Kummer blow up} of $Y$ along $\cJ$ and denoted $Bl_\cJ(Y)$. This plan is worked out in detail in \S\S\ref{logstrsec}--\ref{kummersec}.

\subsubsection{Local charts}
We start with studying the situation \'etale locally, so assume that $Y=\Spec(A)$ and $$\cJ=(t_1\.t_l,m^{1/d}_1\.m^{1/d}_r),$$ where $V(t)$ is a $Z$-submanifold of codimension $l$ and $m_j=u^{q_j}$ are monomials.

For convenience of notation we embed $\cR'$ into $\cA=\oplus_{e\in\ZZ}A_e=A[z^{\pm 1}]$, where $A_e=A$ and $z$ corresponds to $1\in A_1$. By the {\em Rees subalgebra associated} with this choice of generators and the number $d$ we mean the graded algebra $$\tilcR=\oplus_{e\ge 0}\tilcR_e\subseteq\cR'$$ generated by the sets of homogeneous elements $\tilt=zt\subset A_1$ and $z^im\subset A_i$ with $1\le i\le d$. Also, we denote $\tilm=z^dm$ and $g\:\tilY=\cProj_Y(\tilcR)\to Y$.

\begin{lemma}\label{chartlem}
The morphism $h\:\Spec(\cR')\to\Spec(\tilcR)$ is integral, hence $V_{\cR'}=h^{-1}(V_\tilcR)$ and $h$ induces an integral morphism $Y'\to\tilY$.
\end{lemma}
\begin{proof}
Find a finite Kummer cover $\Spec(B)\to Y$ such that $m^{1/d}$ is a set of usual monomials in $B$ and hence $\cI=\cJ|_B$ is an ordinary ideal. Then $\cR_\cI$ is the usual Rees algebra generated over $B$ by $zt$ and $zm^{1/d}$ and hence it is finite over $\tilcR$. It remains to note that $\tilcR\subseteq\cR'\subseteq\cR_\cI$.
\end{proof}

For any $\tils\in\{\tilt,\tilm\}$ consider the open substacks $\tilY_\tils=[\Spec(\tilcR_s)/\GG_m]$ and $Y'_\tils=[\Spec(\cR'_s)/\GG_m]$ of $\tilY$ and $Y'$, respectively.

\begin{corollary}\label{chartcor}
The closed sets $V_\tilcR$ and $V_{\cR'}$ are given by the vanishing of $(\tilt,\tilm)$. In particular, $\tilY=\cup_\tils\tilY_\tils$ and $Y'=\cup_\tils Y'_\tils$, and these coverings are compatible with the morphism $Y'\to\tilY$.
\end{corollary}
\begin{proof}
The claim for $Y_\tilcR$ is obvious, and the claim for $Y_{\cR'}$ follows by Lemma~\ref{chartlem}.
\end{proof}

\subsubsection{Computation of charts}
Let $n\in\{1,d\}$ be the degree of $\tils$. In particular, $s=z^{-n}\tils\in\{t,m\}$. To divide by $\GG_m$ one can first divide by $\mu_n$ and then by $\GG_m/\mu_n$. The invertible homogeneous element $\tils\in\tilcR_s$ of degree $n$ trivializes the latter operation in the following sense: $V(\tils-1)$ is a $\mu_n$-invariant slice of $\Spec(\tilcR_\tils)$ and $[V(\tils-1)/\mu_n]=\tilY_\tils$. The same is true for $Y'_s$, so $$Y'_\tils=[\Spec(\cR'/(\tils-1))/\mu_n]\ \ \ {\rm and}\ \ \ \tilY_\tils=[\Spec(\tilcR/(\tils-1))/\mu_n].$$

We will now describe $\tilY_\tils$ explicitly. An analogous computation of $Y'_s$ is messier, since it should take into account divisibility properties of the monomials $m_j$. So, we prefer not to do it. In particular, the integral morphism $Y'\to\tilY$ is finite, but we will not prove this.

\begin{lemma}\label{chart2lem}
(i) If $Y$ is integral, then $Y'$ and $\tilY$ are integral. In particular, $Y'$ is a partial normalization of $\tilY$.

(ii) Both $g$ and $g'$ are isomorphisms over $Y\setminus V(t,m)$. In particular, $g$ is a $V(t,m)$-supported modification.

(iii) For any $\tils\in\{\tilt,\tilm\}$ $$\tilY_\tils=\left[\Spec\left(A\left[s^{1/n},s^{-1/n}t,s^{-d/n}m\right]\right)/\mu_n\right].$$

(iv) The Kummer ideal $\cJ$ pulls back to invertible ordinary ideals $\cJ'$ on $Y'$ and $\tilcJ$ on $\tilY$. On the $s$-chart both are principal generated by the element $(s^{1/n})$.
\end{lemma}
\begin{proof}
Claim (i) is obvious. To prove (ii) note that for any $s\in\{t,m\}$ one has that $zs\in\tilcR$ and hence $\tilcR_s=\cR'_s=A_s[z]$. It follows that everything trivializes over $Y\setminus V(s)$.

By definition, $z^{-1}\tils\in\tilcR$. Hence $z^{-1}\in\tilcR_\tils$, and $s=z^{-d}\tils=z^{-d}$ in the $\ZZ/n\ZZ$-graded ring $\tilcR/(s-1)$. It follows that $\tilcR/(\tils-1)$ is generated by $z^{-1}=s^{1/n}$ and the sets $\tilt=s^{-1/n}t$, $\tilm=s^{-d/n}m$ yielding the formula for the chart in (iii). Finally, this formula implies that $\cJ\cO_{(\tilY_s)\ket}$ is generated by the element $s^{1/n}$, and similarly for $Y'_s$.
\end{proof}

Note that claim (iii) above  is a generalization of the usual blow ups formulas to a weighted case, and it reduces to description of a usual blow up when $d=1$ and hence $\cJ$ is an ideal.

\subsubsection{Divisorial enlargements of logarithmic structures}\label{logstrsec}
If $Y$ is an integral logarithmic scheme and $\cI$ is an invertible ideal, by $M_Y(\cI)$ we mean the canonical enlargement of $M_Y$ obtained by locally adding the generator of $\cI$ and all possible quotients of monomials by it. More concretely, locally at a geometric point $y$ we have that $\cI=(m)$ and we enlarge $M_y$ by adding $s$ with $u^s=m$ and elements $q-ks$ with $u^{q-ks}=u^q/m^k$ whenever $k\in\NN$ and $q\in M_y$ satisfy $u^q\in(m^k)$.

\begin{remark}
(i) Since we work in the category of integral logarithmic structures, some relations are automatically imposed. For example, if $q\in M_y$ satisfies $u^q=m^k$, then automatically $ks=q$. In particular, if $q'\in M_y$ is another element with $u^{q'}=m^k$, then $q=q'$ in $M_Y(\cI)$. Thus, the homomorphism $M_Y\to M_Y(\cI)$ does not have to be injective when $u\:M_Y\to\cO_Y$ is not injective.

(ii) In general, $M_Y(\cI)$ does not have to be fine even when $M_Y$ is. It is somewhat analogous to the fact that in pathological situations normalization might not be finite.
\end{remark}

\subsubsection{The unsaturated model}
We provide $\tilY$ and $Y'$ with the logarithmic structures $M_\tilY=g^*M_Y(\tilcJ)$ and $M_{Y'}=g'^*M_Y(\cJ')$, and denote the resulting fine logarithmic schemes over $Y$ by $g\:\tilY\to Y$ and $g'\:Y'\to Y$. We call $\tilY$ the {\em unsaturated model} associated with the choice $(t,m^{1/d})$ of generators and $d$. First, let us describe a model case.

\begin{example}\label{unsatexam}
Assume that $P\to Q$ is an injective homomorphism of toric monoids, $Z^\circ=\bfA_P$ and $Y^\circ=\Spec(A^\circ)$, with $A^\circ=\QQ[Q][t_1\.t_l]$ and the logarithmic structure given by $Q$. Let $\cJ^\circ=(t,m^{1/d})$ be a $Z^\circ$-submonomial Kummer ideal on $Y^\circ$ and let $\tilY^\circ$ be the unsaturated model of Kummer blow up along $\cJ^\circ$. It then follows from Lemma~\ref{chart2lem} that
\begin{itemize}
\item[(1)] If $s=t_i$, then $\tilY^\circ_\tils=\Spec(\QQ[Q_i][\frac{t_1}{t_i}\.\frac{t_n}{t_i}])$, where $\frac{t_i}{t_i}$ is skipped, $Q_i$ is the submonoid of $Q^\gp\oplus\ZZ q$ generated by $Q$ and $q,q_1-dq\.q_r-dq$ and the logarithmic structure is given by $Q_i$ with $u^q=s$.
\item[(2)] If $s=m_j^{1/d}$, then $\tilY^\circ_\tils=\Spec(\QQ[Q_j][\frac{t_1}{m_j}\.\frac{t_n}{m_j}])$, where the logarithmic structure is given by $Q_j$, which is the submonoid of $\frac{1}{d}Q^\gp$ generated by $Q$ and the elements $q=\frac{1}{d}q_j,q_1-q_j\.q_r-q_j$.
\end{itemize}
In fact, $\tilY^\circ$ is an unsaturated version of Kummer blow up of the smooth logarithmic variety $Y^\circ$ as described in \cite[\S5]{ATW-destackification}.
\end{example}

All properties of submonomial blow ups will be deduced from this example by use of the following

\begin{lemma}\label{modelkummerlem}
Let $f\:Y\to Z$ be a relative logarithmic orbifold and $\cJ=(t,m^{1/d})$ a submonomial ideal, and suppose that $f$ possesses a chart $h\:Y\to Z_P[Q]$ such that $h$ is a regular morphism (for example, the chart is neat) and $Z\to\bfA_P$ is flat (for example, $Z\in\bB$). Let $Y^\circ\to Z^\circ$, $\cJ^\circ$ and $\tilY^\circ$ be as in Example~\ref{unsatexam}, with $\cJ = \cJ^\circ\cO_Y$. Then $\tilY=\tilY^\circ\times_{Y^\circ}Y$ both as DM stacks and as fine logarithmic DM stacks.
\end{lemma}
\begin{proof}
The morphism $(h,t)\:Y\to Z_P[Q]\times\bfA^l$ is regular by Lemma~\ref{regularlem}(i). The morphism $Z_P[Q]\times\bfA^l\to Y^\circ$ is the base change of $Z\to\bfA_P$, hence it is flat and the composition $Y\to Y^\circ$ is flat. Let $\tilcR$ and $\tilcR^\circ$ be the model algebras of $\cJ$ and $\cJ^\circ$ associated with the generators $\{t,m^{1/d}\}$ and $d$. By the definition, we have equality of ideals $\tilcR^\circ_iA=\tilcR_i$ and by the flatness of $A^\circ\to A$ we also have that $\tilcR^\circ\otimes_{A^\circ}A=\tilcR^\circ_iA$. So we obtain an isomorphism of schemes $\tilY=\tilY^\circ\times_{Y^\circ}Y$.

Let $g\:\tilY\to Y$ and $g^\circ\:\tilY^\circ\to Y^\circ$ denote the natural morphisms and let $\tilcJ$ and $\tilcJ^\circ$ be the pullbacks of $\cJ$ and $\cJ^\circ$. Since the morphism $Y\to Y^\circ$ is strict, we should show that $M_{\tilY}=g^*M_Y(\tilcJ)$ is the pullback of $M_{\tilY^\circ}=(g^\circ)^*M_{Y^\circ}(\tilcJ^\circ)$. Clearly, $g^*M_Y$ is the pullback of $(g^\circ)^*M_{Y^\circ}$, and it remains to notice that $\tils^k|u^q$ on $\tilY^\circ_\tils$ if and only if the same divisibility holds on $\tilY$ because the morphism $\tilY\to\tilY^\circ$ is flat.
\end{proof}

\begin{corollary}\label{modelkummerreg}
With assumptions as in Lemma~\ref{modelkummerlem}, the morphism $\tilY\to Z$ is logarithmically regular. If $\cJ$ is monomial (that is, $l=0$), then the morphism $\tilY\to Y$ is logarithmically smooth.
\end{corollary}
\begin{proof}
By Lemma \ref{modelkummerlem} there exists the following commutative diagram with three cartesian squares
$$
\xymatrix{
\tilY \ar[r]^(.3)\tilh\ar[d] & \tilY^\circ\times_{Z^\circ}Z \ar[r]\ar[d]& \tilY^\circ \ar[d]\\
Y\ar[r]^(.3)h\ar[rd] & Z_P[Q] \ar[r]\ar[d] & Y^\circ\ar[d]\\
 & Z\ar[r]& Z^\circ
}
$$
Since $h$ is logarithmically regular (even strict regular), the same is true for $\tilh$.

By Example \ref{unsatexam}, $\tilY^\circ\to Z^\circ$ is logarithmically smooth, so the base change $\phi\:\tilY^\circ\times_{Z^\circ}Z\to Z$ is logarithmically smooth, and the composition $\phi\circ\tilh\:\tilY\to Z$ is logarithmically regular. In the monomial case, $\tilY^\circ\to Y^\circ$ is logarithmically smooth, hence its base change $\tilY\to Y$ is logarithmically smooth.
\end{proof}

\begin{corollary}\label{modelkummercor}
With assumptions as in Lemma~\ref{modelkummerlem}, $\tilY^\sat=(\tilY^\circ\times_{Y^\circ}Y)^\sat$ is logarithmically regular over $Z$. If, in addition $Z\in\bB$ (see \S\ref{Bsec}), then the morphism $Y'\to\tilY$ induces an isomorphism of saturations $Y'^\sat=\tilY^\sat$. In particular, $\tilY^\sat$ does not depend on the choice of generators of $\cJ$.
\end{corollary}
\begin{proof}
The first claim follows from Lemma~\ref{modelkummerlem} immediately. If $Z\in\bB$, then $\tilY^\sat$ is special by Lemma~\ref{stablem}. Recall that $Y'\to\tilY$ is integral and birational by Lemma~\ref{chart2lem}, and by the construction the logarithmic structure $M_{Y'}\into\cO_{Y'}$ is injective and is trivial over the triviality locus of $M_\tilY$. Since $Y'$ is special, this necessarily implies that the morphism $Y'^\sat\to\tilY^\sat$  is an isomorphism.
\end{proof}

\subsubsection{Submonomial blow ups}\label{kummersec}
We have established independence of generators in the local case and can now deal with the case when $\cJ$ is an arbitrary submonomial Kummer ideal. In general, our definition of Kummer blow ups is based on Corollary~\ref{modelkummercor}, so we have to assume that the target is in $\bB$ and switch to the notation $f\:Y\to B$.

The pullback $\cJ'=\cJ\cO_{Y'\ket}$ with respect to $g'\:Y'=\cProj_Y(\pi_*\cR_\cJ)\to Y$ is an invertible ideal. Indeed, the claim can be checked \'etale locally, and hence follows from Lemma~\ref{chart2lem}. So, as in the local case we provide $Y'$ with the enlarged pullback logarithmic structure $M_{Y'}=g'^*M_Y(\cJ')$, and define the Kummer blow up along $\cJ$ to be both the logarithmic DM stack $Bl_\cJ(Y)=Y'^\sat$ and the morphism $g\:Bl_\cJ(Y)\to Y$. The pullback of $\cJ$ will be also denoted $\cJ'=\cI_E$, where its vanishing locus $E$ is the {\em exceptional divisor} of the blow up.

\begin{theorem}\label{kummerblowth}
Let $Y\to B$ be a logarithmically regular morphism of DM logarithmic stacks with $B\in\bB$, and let $X=Bl_\cJ(Y)\to Y$ be a $B$-submonomial Kummer blow up with center $\cJ$. Then,

(i) $\sigma\:Bl_\cJ(Y)\to Y$ is a $V(\cJ)$-supported modification with \'etale diagonalizable relative inertia.

(ii) $Bl_\cJ(Y)$ is logarithmically regular over $B$.

(iii) $\sigma^{-1}\left(\cJ^{(a)}\right)=\cI^a_E$ for any $a\ge 0$.

(iv) For any morphism $B'\to B$, the pullback of $\cJ$ to $Y'=Y\times_BB'$ is a $B$-submonomial Kummer ideal $\cJ'$ and $Bl_{\cJ'}(Y')=Y\times_BB'$.

(v) For any logarithmically regular morphism $T\to Y$, the pullback of $\cJ$ to $T$ is a $B$-submonomial Kummer ideal $\cI$ and $Bl_{\cI}(T)=Bl_\cJ(Y)\times_YT$ (in the fs category).

(vi) $I_{X/Y}$ is a finite subgroup of $(\GGm)_X$. In addition, if $Y\to S$ is a morphism of stacks with a finite diagonalizable inertia $I_{Y/S}$ acting trivially on monoids $\oM_\oy$, then $I_{X/S}$ is also finite diagonalizable and acts trivially on monoids $\oM_\ox$.
\end{theorem}
\begin{proof}
The claim is \'etale local on $B$, $Y$, $B'$ and $T$. Therefore, (i), (ii) and (iv) follow from Lemma~\ref{chart2lem} and Corollaries~\ref{modelkummerreg} and \ref{modelkummercor}. To prove (iii) we note that $\cI^a_E=\sigma^{-1}(\cJ^a)\subseteq\sigma^{-1}\left(\cJ^{(a)}\right)\subseteq(\cI^a_E)^\nor$. Since $\cI^a_E$ is an invertible monomial ideal, it is saturated, and hence $(\cI^a_E)^\nor=\cI^a_E$ by Lemma~\ref{norsatlem}.

Let us prove (v). By Theorem~\ref{neatth} we can assume that there exist neat charts $Y\to B_P[Q]$ and $T\to Y_Q[R]$. The composed chart does not have to be neat, but the morphism $T\to B_P[R]$ is regular by the diagram in \S\ref{compossec}. Therefore, by Lemma~\ref{modelkummerlem} it suffices to prove the claim in the case when $B=\bfA_P$, $Y=\Spec(\QQ[Q][t])$ and $T=\Spec(\QQ[R][t])$ with $t=(t_1\.t_l)$ and $\cJ=(t,m^{1/d})$ for monomials $m=(m_1\.m_r)$. In this case, the claim follows by comparing the explicit charts in Example~\ref{unsatexam} and observing that the new monoids $R_i$ and $Q_i$ are generated over the old monoids $R$ and $Q$ by the same generators and relations, and hence $R_i=Q_i\oplus_QR$.

Since $X$ is obtained as $\cProj_Y$, everything in (vi) except triviality of the action on $\oM_\ox$ follows from Lemma~\ref{inertialem}. The triviality can be easily reduced to the case of charts and then checked explicitly.
\end{proof}

\subsubsection{Sequences of submonomial Kummer blow ups}
Since logarithmic regularity is preserved by submonomial Kummer blow ups, we can define sequences of such blow ups. Typically such a sequence
$$
\xymatrix{Y'=: Y_n \ar[r]^{\sigma_n} &Y_{n-1} \ar[r]^{\sigma_{n-1}} & \dots \ar[r]^{\sigma_2} &Y_{1} \ar[r]^(.4){\sigma_{1}} & Y_0:=  Y.}
$$
will be denoted $\sigma_k\:Y_k\to Y_{k-1}$, $1\le k\le n$, or just $\sigma\:Y'\dashrightarrow Y$.

\subsubsection{Strict transforms}
As in the classical case, by the {\em strict transform} $H'$ of a closed substack $H\into Y$ under a $B$-submonomial Kummer blow up $g\:Y'=Bl_\cJ(Y)\to Y$ we mean the schematic closure of $H\setminus V(\cJ)$ in $Y'$.

\begin{lemma}\label{strictlem}
Keep the above notation and assume that $H$ is a logarithmic $B$-submanifold such that $\cI=\cJ\cO_{H\ket}$ is a $B$-submonomial ideal on $H$. Then $H'$ underlies a $B$-submanifold of $Y'$ and $H'\to H$ is the $B$-submonomial Kummer blow up along $\cI$.
\end{lemma}
\begin{proof}
The claim is \'etale local, hence we can assume that there exists regular parameters $(t_1\.t_n)$ and Kummer monomials $m=(m_1\.m_r)$ such that $\cJ=(t_1\.t_l,m)$ and $H=V(t_{l+1}\.t_k)$. Then Corollary~\ref{modelkummercor} reduces the claim to the model case analogous to the one described in Example~\ref{unsatexam} but with additional regular parameters: $$Y^\circ=\Spec(\QQ[Q][t_1\.t_k])\ \ {\rm  and}\ \ H^\circ=V(t_{l+1}\.t_k)=\Spec(\QQ[Q][t_1\.t_l]).$$ In this case, the claim immediately follows from the description of charts in the same example.
\end{proof}

\subsubsection{Pushforwards from submanifolds}\label{pushsec}
For any logarithmic $B$-submanifold $i\:H\into Y$ and a Kummer ideal $\cJ\subseteq\cO_{H\ket}$ we denote its preimage in $\cO_{Y\ket}$ by $i_*(\cJ)$.

\begin{lemma}\label{pushlem}
Assume that $Y_0\to B$ is a logarithmically regular morphism of logarithmic DM stacks with $B\in\bB$, $H_0\into Y_0$ is a logarithmic $B$-submanifold, and $g_k\:H_k\to H_{k-1}$, $1\le k\le n$ a sequence of $B$-submonomial Kummer blow ups with centers $\cJ_k$. Then there exists a unique sequence of $B$-submonomial Kummer blow ups $h_k\:Y_k\to Y_{k-1}$, $1\le k\le n$ with centers $\cI_k$ and logarithmic $B$-submanifolds $i_k\:H_k\into Y_k$ such that $\cI_k=(i_k)_*(\cJ_k)$ for any $k\in\{0\.n-1\}$.
\end{lemma}
\begin{proof}
These conditions define $h_0$, and since $\cI_0|_{H_0}=\cJ_0$, Lemma~\ref{strictlem} yields that $H_1=Y_1\times_{Y_0}H_0$. The rest follows by induction on $k$.
\end{proof}

In the situation described by Lemma~\ref{pushlem} we call the blow up sequence $h$ the {\em pushforward} of $g$ and use the notation $h=i_*(g)$.

\subsubsection{Admissibility}\label{admisssec}
We say that a $B$-submonomial Kummer ideal $\cJ$ is {\em $a$-admissible with respect to} an ideal $\cI\subseteq\cO_Y$ if $\cI\subseteq\cJ^{(a)}$. In this case, the submonomial Kummer blow up $\sigma\:Y'\to Y$ along $\cJ$ is also called {\em $a$-admissible with respect to $\cI$}.

\begin{remark}
In the classical Hironaka's algorithm $\cJ^a=\cJ^{(a)}$ automatically, but in our case it is more convenient to use $\cJ^{(a)}$ in this definition.

\end{remark}

\subsection{Pullbacks}

\subsubsection{Serre's twist}
If $\sigma\:Y'\to Y$ is a blow up along an ideal $\cJ$, then $\cI_E=g^{-1}(\cJ)$ is the Serre's twisting sheaf $\cO_{Y'}(1)$ for $Y'=\Proj_Y(\oplus_e \cJ^e)$. Informally, we will view the invertible sheaf $\cI_E$ in this way also in the case of submonomial Kummer blow ups. Transforms of various objects under blow ups often involve multiplication by an appropriate power $\cI^d_E$ with $d\in\ZZ$, which can be viewed as Serre's twist.

\subsubsection{Ideals}\label{twistedpullback}
Assume that $\sigma$ is $a$-admissible with respect to an ideal $\cI$. Then $\sigma^{-1}\cI\subseteq\cI_E^a$ by Theorem~\ref{kummerblowth}(iii), and hence the twisted pullback $\sigma^{-1}(\cI,a):=\cI_E^{-a}\sigma^{-1}(\cI)$ is defined as an ideal on $Y'$. The notation follows the one in \cite{ATW-principalization}; note that the twist is by $-a$.

\subsubsection{Derivations}
As in the absolute case (see \cite[Lemma~4.2.1]{ATW-principalization}), differential operators can be pulled back with an opposite twist.

\begin{lemma}\label{transformderivlem}
Assume that $Y\to B$ is a logarithmically regular morphism of DM logarithmic stacks with $B\in\bB$ and $\sigma\:Y'\to Y$ is a submonomial Kummer blow up. Then for any $i\in\NN$ there is a natural embedding $\cI_E^i\cD^{(\le i)}_{Y/B}(Y')\into\cD^{(\le i)}_{Y'/B}$ given by a unique extension of differential operators from $\cO_Y$ to $\cO_{Y'}$.
\end{lemma}
\begin{proof}
Since the characteristic is zero, we can assume that $i=1$. In addition, the claim is Kummer-local on $Y$, hence it reduces to the particular case when $Y=\Spec(A)$ and $B=\Spec(C)$ are affine logarithmic schemes and $\sigma$ is a submonomial blow up along an ideal $J$, see the proof of \cite[Lemma~4.2.1]{ATW-principalization} for details. One can now proceed as in the cited proof, but a direct computation as follows seems to be the shortest argument.

Let $t=t_1\.t_n$ be generators of $J$. It suffices to study the situation on a chart $Y'_t=\Spec(A[\frac Jt])$ of $\sigma$. We should show that for a logarithmic $C$-derivation $\partial\:A\to A':=A[\frac Jt]$, the logarithmic derivation $t\partial$ uniquely extends to $A'$. It suffices to deal with the $A$-generators $\frac{t_i}t$ of $A'$. The formula $t\partial(\frac{t_i}t)=\partial(t_i)-\frac{t_i}t\partial(t)$ gives rise to a unique extension of $t\partial$ to a $C$-derivation $\partial'\:A'\to A'$. It is a logarithmic derivation since $\frac{\partial'(t)}{t}=\partial(t)\in A$, and for each $t_i$, which is a monomial, we have that $\frac{\partial(t_i)}{t_i}\in A$ and hence $\partial'(\frac{t_i}t)/\frac{t_i}t\in A'$.
\end{proof}

By Lemma~\ref{firstseqlem} we obtain the following corollary for derivations:

\begin{corollary}\label{transformderivcor}
In the situation of Lemma \ref{transformderivlem}, $\cD_{Y'/Y}=0$ and $$\cI_E\cD_{Y/B}(Y')\subseteq\cD_{Y'/B}\subseteq\cD_{Y/B}(Y').$$
\end{corollary}

\subsubsection{Derived ideals}
For any $\cO_Y$-submodule $\cF\subseteq\cD_{Y/B}$ let $\sigma^*\cF\subseteq\cD_{Y/B}(Y')$ be as in \S\ref{difpullbacksec}. Then $\sigma^c(\cF):=\cI_E\sigma^*(\cF)$ is an $\cO_{Y'}$-submodule of $\cD_{Y'/B}$ by Lemma~\ref{transformderivlem}. More generally, for a sequence of submonomial Kummer blow ups $\sigma\:Y'\dashrightarrow Y$ we define the {\em controlled transform} $\sigma^c(\cF)\subseteq\cD_{Y'/B}$ by induction on the length, and will usually use abbreviations $\cF_\sigma=\sigma^*(\cF)$ and $\cF_\sigma^{(\le i)}=(\cF_\sigma)^{(\le i)}$. The following result describes compatibility of pullbacks and derived ideals. It extends and makes \cite[Lemma 4.3.11]{ATW-principalization} sharper, but the argument is the same.

\begin{lemma}\label{Flem}
Let $Y\to B$ be a logarithmically regular morphism of DM stacks with $B\in\bB$, $\sigma\:Y'\to Y$ a submonomial Kummer blow up, and $\cF\subseteq\cD_{Y/B}$ an $\cO_Y$-submodule. If $0\le i<a$ are integers and $\sigma$ is $a$-admissible with respect to an ideal $\cI\subseteq\cO_Y$, then
$$\cF_\sigma^{(\le i)}(\sigma^{-1}(\cI,a))=\sum_{j=0}^i \sigma^{-1}\left(\cF^{(\le j)}(\cI),a-j\right).$$
\end{lemma}
\begin{proof}
Set $\cI'=\sigma^{-1}(\cI,a)$.
The case $i=1$ reads:
$$\cF_\sigma(\cI')+\cI'=\cF_\sigma^{(\le 1)}(\cI')=\sigma^{-1}\left(\cF^{(\le 1)}(\cI),a-1\right)+\cI'=\sigma^{-1}(\cF(\cI),a-1)+\cI',$$
where we use that $\sigma^{-1}(\cI,a-1)\subset\cI'$. \'Etale-locally $\cF_\sigma(\cI')$ and $\sigma^{-1}(\cF(\cI),a-1)$ are generated by global sections $y\partial(y^{-a}h)$ and $y^{1-a}\partial(h)$, respectively, where $h\in\Gamma(\cI)$ and $\partial\in\Gamma(\cF)$, and $y\in\Gamma(\cI_E)$. Note that $$y^{1-a}\partial(h)-y\partial(y^{-a}h)=a\partial(y)\cdot y^{-a}h\in\Gamma(\cI'),$$
giving the case $i=1$. Inductively,
\begin{align*} \cF_\sigma^{(\le i+1)}(\sigma^{-1}(\cI,a)) &= \cF_\sigma^{(\le 1)}\left(\cF_\sigma^{(\le i)}(\sigma^{-1}(\cI,a))\right) \\
&= \cF_\sigma^{(\le 1)}\left( \sum_{j=0}^i \sigma^{-1}\left(\cF^{(\le j)}(\cI),a-j\right)   \right)\\
&= \sum_{j=0}^{i+1} \sigma^{-1}\left(\cF^{(\le j)}(\cI),a-j\right),
\end{align*} as needed.
\end{proof}



\section{Marked ideals and admissibility}
\addtocontents{toc}
{\noindent Marked ideals and their derivations, sums and products. Equivalence of marked ideals. Coefficient ideals.}

\subsection{Basic facts}\label{basicfactssec}

\subsubsection{The definition}\label{def-resolved-marked}
A {\em marked ideal} $\ucI=(\cI,a)$ for a logarithmically regular morphism $f\:Y\to B$ consists of an ideal $\cI\subseteq\cO_Z$ and a positive integer $a$. The {\em support} of $\ucI$ is the set $\supp(\ucI)$ of points $y\in|Y|$ such that $\logord_{\cI/B,y}(y)\ge a$, and $\ucI$ is {\em resolved} if $\supp(\ucI)=\emptyset$. At some places it is convenient to consider {\em generalized marked ideals}, where $a$ can equal 0.

\begin{remark}
Probably the best way of thinking about marked ideals is as a sort of weighted ideals. In particular, this logic is consistent with operations we define on marked ideals below.
\end{remark}

\subsubsection{Normalized invariants}\label{invarsec}
Invariants of ideals possess the following normalized analogues. Given a marked ideal $\ucI=(\cI,a)$ we define its {\em normalized logarithmic order} as follows: $\mu_y(\ucI)=0$ if $\logord_{\cI/B}(y)<a$, and $\mu_y(\ucI)=\logord_{\cI/B}(y)/a$ otherwise. Also, we define the {\em normalized monomial saturation} to be the monomial Kummer ideal $\cW_{Y/B}(\ucI)=(\cM(\cI)^\sat)^{1/a}$. Sometimes we will write $\cW(\ucI)$ for shortness.

\begin{remark}
The notation $\mu_y$ follows its analogue in \cite[Section 6]{Bierstone-Milman-funct}. Note that $\mu_y(\ucI)\in\QQ_{\ge 1}\cup\{0,\infty\}$ and $\mu_y(\ucI)=0$ if and only if $y\notin\supp(\ucI)$.
\end{remark}

\subsubsection{Cleaning blow up}
A logarithmically clean ideal can be made clean by a single blow up along $\cW(\ucI)$.

\begin{theorem}\label{cleaningth}
Assume that $X\to B$ is a relative logarithmic orbifold with $B\in\bB$ and $\ucI=(\cI,a)$ is a marked ideal on $X$ such that $\cI\subseteq\cO_X$ is logarithmically clean over $B$ (\S\ref{Sec:log-clean}). Then the monomial Kummer blow up $\sigma\:X'\to X$ along $\cW(\ucI)$ is $\ucI$-admissible and the ideal $\sigma^{-1}(\cI,a)$ is clean.
\end{theorem}
\begin{proof}
The blow up is admissible because $\cI\subseteq\cM(\cI)\subseteq\cW(\ucI)^a$. Furthermore, $$\cD_{X'/B'}^{\infty}(\sigma^{-1}\cI)=\sigma^{-1}(\cD_{X/B}^{\infty}(\cI))=\sigma^{-1}(\cM(\cI)),$$ where the first equality holds by Lemma~\ref{functorlogorder2}(iii) since $\sigma$ is logarithmically \'etale, and the second one holds because $\cI$ is logarithmically clean. In addition, $\sigma^{-1}(\cM(\cI))$ is a monomial ideal whose saturation is the invertible monomial ideal $\cI_E^a$, and hence $\sigma^{-1}(\cM(\cI))=\cI_E^a$. Therefore $\cD_{X'/B'}^{\infty}(\sigma^{-1}(\cI,a))=\cI_E^a\cI_E^{-a}=\cO_{X'}$, and we obtain that $\sigma^{-1}(\cI,a)$ is clean by Lemma~\ref{logordlem}.
\end{proof}

\subsubsection{Arithmetic operations} \label{Homogenized-addition}
Given $n$ marked ideals $\ucI_i=(\cI_i,a_i)$ one usually sets
$$\ucI_1+\ldots+\ucI_n=\left(\cI_1^{a/a_1}+\dots+\cI_n^{a/a_n},a\right),\ \ \ \ucI_1\cdot\ldots\cdot\ucI_n=\left(\cI_1\cdot\ldots\cdot\cI_n,a_1+\ldots+a_n\right),$$
where $a=a_1\cdot\ldots\cdot a_n$. However, we prefer to replace the definition of sums by the following homogenized variant
$$\ucI_1+\ldots+\ucI_n:=\left(\sum_{l\in\NN^n|\ l_1a_1+\ldots+l_na_n\ge a}\cI_1^{l_1}\cdot\ldots\cdot\cI_n^{l_n},a\right)$$
In addition, we write $\ucI_1\subseteq\ucI_2$ if $a_1=a_2$ and $\cI_1\subseteq\cI_2$.

\begin{remark}\label{sumrem}
(i) Similarly to the situation with the usual definition, the addition is commutative but not associative. However, it is associative up to an equivalence relation introduced in \S\ref{equivsec} below and this is enough for applications, see also \cite[\S5.2]{ATW-principalization}.

(ii) In fact, our definition and the usual one produce equivalent marked ideals. The usual definition is lighter, so it may be preferable for computations. Homogenized sums are critical to have Theorem~\ref{coefprop} below.
\end{remark}

\subsubsection{Derivations}
The action of differential operators on $\ucI$ is also defined in a weighted way: if $\ucI=(\cI,a)$ and $\cF$ is an $\cO_Y$-submodule of $\cD_{Y/B}$, then we provide $\cF$ with weight $1$ and set $\cF^{(\le i)}(\ucI)=(\cF^{(\le i)}(\cI),a-i)$ for any $i$ such that $0\le i\le a$. Note that for $i=a$ we obtain a generalized marked ideal.

\subsubsection{Admissible sequences and transforms}
Let $\ucI$ be a generalized marked ideal. A $B$-submonomial Kummer blow up $\sigma\:Y'\to Y$ is called {\em $\ucI$-admissible} if it is $a$-admissible with respect to $\cI$, and in this case we define the {\em controlled transform} $\sigma^c\ucI=(\sigma^{-1}(\cI,a),a)$, see \S\ref{twistedpullback}. These definitions extend to the case when $\sigma$ is a sequence of $B$-submonomial Kummer blow ups of length $n$. Namely, $\sigma$ is {\em $\ucI$-admissible} if each $\sigma_i\:Y_i\to Y_{i-1}$ for $0\le i\le n-1$ is $\ucI_{i-1}$-admissible, where $\ucI_i=\sigma_i^c\ucI_{i-1}$, and the {\em controlled transform} under $\sigma$ is $\sigma^c\ucI=\ucI_n$. We will refer to such $\sigma$ as an {\em $\ucI$-admissible sequence}. Note that any sequence of $B$-submonomial Kummer blow ups is $(\cI,0)$-admissible and the controlled transform is the usual pullback in this case.

\subsubsection{Order reduction}
By an {\em order reduction} of a marked ideal $\ucI$ we mean an admissible sequence $\sigma\:X'\dashrightarrow X$ such that $\sigma^c(\ucI)$ is resolved in the sense of Section \ref{def-resolved-marked}.

\subsubsection{Transforms and operations}
Compatibility relations between transforms and operations are the same as in the classical case, see \cite[\S3]{Wlodarczyk}, \cite[\S3]{Bierstone-Milman-funct}, and the absolute logarithmic case, see \cite[\S5]{ATW-principalization}. Arguments are also the same, so we just sketch them.

\begin{lemma}\label{operlem}
Assume that $Y\to B$ is a logarithmically regular morphism of logarithmic DM stacks with $B\in\bB$, $\sigma\:Y'\dashrightarrow Y$ is a sequence of $B$-submonomial Kummer blow ups, and $\ucI=\ucI_1,\ucI_2\dots\ucI_n$ are $n$ marked ideals on $Y$. Setting $\ucP=\ucI_1\cdot\ldots\cdot\ucI_n$ and $\ucS=\ucI_1+\ldots+\ucI_n$ we have:

(i) If $\sigma$ is $\ucI_i$-admissible for $1\le i\le n$, then $\sigma$ is $\ucP$-admissible and $\sigma^c(\ucP)=\sigma^c(\ucI_1)\cdot\ldots\cdot\sigma^c(\ucI_n)$.

(ii) Let $k$ be a positive integer. Then $\sigma$ is $\ucI$-admissible if and only if it is $\ucI^k$-admissible, and in this case $(\sigma^c\ucI)^k=\sigma^c(\ucI^k)$.

(iii) $\sigma$ is $\ucI_i$-admissible for each $i\in\{1\. n\}$ if and only if it is $\ucS$-admissible, and in this case $\sigma^c(\ucS)=\sigma^c(\ucI_1)+\ldots+\sigma^c(\ucI_n)$.

(iv) If $\sigma$ is $\ucI$-admissible, $\cF\subseteq\cD_{Y/B}$ is an $\cO_Y$-submodule, and $0\le i\le a$, where $\ucI=(\cI,a)$, then $\sigma$ is $\cF^{(\le i)}(\ucI)$-admissible and $\sigma^c\left(\cF^{(\le i)}(\ucI)\right)\subseteq \cF_\sigma^{(\le i)}(\sigma^c\ucI)$.
\end{lemma}
\begin{proof}
In all claims, {induction} on the length of $\sigma$ reduces the proof to the case when $\sigma$ is a single Kummer blow up along $\cJ$. The argument is the same for each $n\ge 2$, and we consider the case of $n=2$ to simplify the notation. We start with the admissibility claims. Admissibility is obvious in (i), and follows from Remark~\ref{norrem} in (ii). In (iii), if $\cI_i\subseteq\cJ^{(a_i)}$ for $i=1,2$ then $\cI_1^m\cI_2^l\subseteq\cJ^{(a_1a_2)}$ whenever $ma_1+la_2\ge a_1a_2$ and hence $\cJ$ is $(\ucI_1+\ucI_2)$-admissible. Conversely, if $\cJ$ is $(\ucI_1+\ucI_2)$-admissible, then it is admissible for $\ucI_1^{a_2}$ and $\ucI_2^{a_1}$, and it remains to use (ii). Finally, in (iv) we apply $\cF^{(\le i)}$ to the inclusion $\cI\subseteq\cJ^{(a)}$ and use that $\cF^{(\le i)}\left(\cJ^{(a)}\right)\subseteq\cJ^{(a-i)}$ by Lemma~\ref{Jnorlem}(ii).

Concerning the relations between the transforms, the equalities in (i), (ii) and (iii) are obtained by unwinding the definitions, and the inclusion in (iv) is obtained from the equality in Lemma~\ref{Flem} by taking the summand with $j=i$.
\end{proof}

\subsubsection{Maximal order}
{As usual,} a marked ideal $\ucI=(\cI,a)$ is said to be of {\em maximal order} if $\logord_{\cI/B}(y)\le a$ (or $\mu_y(\ucI)\le 1$) for any $y\in Y$.

\begin{lemma}\label{maxorderlem}
Let $Y\to B$ be a logarithmically regular morphism of logarithmic DM stacks with $B\in\bB$, $\ucI\, \ucI_1\.\ucI_n$ marked ideals, and $k$ a positive integer.

(i) $\ucI$ is of maximal order if and only if $\ucI^k$ is of maximal order.

(ii) If $\ucI_1$ is of maximal order, then $\ucI_1+\ldots+\ucI_n$ is of maximal order.
\end{lemma}
\begin{proof}
Arithmetic operations commute with restriction onto the logarithmic fibers, hence this reduces to basic properties of the usual order of ideals on regular schemes.
\end{proof}

Probably the following result holds for arbitrary logarithmically regular morphisms, but proving this would involve a direct computation with completed logarithmic fibers. We chose to establish a slightly less general case, where derivations provide a nice short argument.

\begin{lemma}\label{maxordertransform}
Let $X\to B$ be a relative logarithmic orbifold and $\ucI=(\cI,a)$ a marked ideal. Then

(i) $\ucI$ is of maximal order if and only if $\cD^{(\le a)}_{X/B}(\cI)=\cO_X$.

(ii) If $\ucI$ is of maximal order and $\sigma\:X'\dashrightarrow X$ is an $\ucI$-admissible sequence, then $\ucI'=\sigma^c(\ucI)$ is of maximal order too.
\end{lemma}
\begin{proof}
Claim (i) follows from Lemma~\ref{logordlem}. Using (i) and Lemma~\ref{operlem}(v) we obtain in (ii) that $$(\cO_{X'},0)=\sigma^c\left(\cD^{(\le a)}_{X/B}(\ucI)\right)\subseteq\cD^{(\le a)}_{X'/B'}(\ucI').$$  So $\cD^{(\le a)}_{X'/B'}(\cI')=\cO_{X'}$, and hence $\ucI'$ is of maximal order by (i).
\end{proof}

\subsubsection{Balanced marked ideals}
We say that a marked ideal $\ucI$ is balanced if $\cI$ is, see \S\ref{balancedsec}. By the {\em clean part} of $\ucI$ we mean the marked ideal $\ucI^\cln=(\cI^\cln,b)$, where $b=\max(a,\logord_{\cI/B}(X))$. Note that $\ucI^\cln$ is of maximal order, and if $\ucI$ is clean, then $\ucI^\cln$ is obtained from $\ucI$ by the maximal increase of the weight that keeps it of maximal order.

\begin{corollary}\label{balancedcor}
Let $X\to B$ be a relative logarithmic orbifold with $B\in\bB$ and $\ucI$ a balanced marked ideal. Then any $\ucI^\cln$-admissible sequence $\sigma\:X'\dashrightarrow X$ is also $\ucI$-admissible and $\sigma^c(\ucI)$ is balanced with equality of clean parts  $(\sigma^c(\ucI))^\cln = \sigma^c(\ucI^\cln)$.
\end{corollary}
\begin{proof}
Let $\ucI=(\cI,a)$ and $\ucI^\cln=(\cI^\cln,b)$, in particular, $\cI=\cM\cdot\cI^\cln$, where $\cM=\cM(\cI)$ is an invertible monomial ideal. By induction on the length it suffices to consider a single blow up and then the admissibility claim is obvious.

By Lemma~\ref{maxordertransform}, $\sigma^c(\ucI^\cln)=(\cI',b)$ is of maximal order, in particular, $\cI'$ is clean. It remains to note that $\sigma^c(\ucI)=(\sigma^{-1}(\cM)\cdot\cI_E^{b-a}\cdot\cI',a)$.
\end{proof}

\subsection{Equivalence and domination of marked ideals}\label{equivsec}

\subsubsection{The definition}
Let $Y\to B$ be a logarithmically regular morphism and let $\ucI_1$ and $\ucI_2$ be marked ideals on $Y$. We say that $\ucI_1$ is {\em dominated} by $\ucI_2$ if any $\ucI_2$-admissible sequence of $B$-submonomial Kummer blow ups is also $\ucI_1$-admissible. If $\ucI_1$ and $\ucI_2$ dominate each other, then we say that the marked ideals are {\em equivalent}. Furthermore, we say that $\ucI_1$ and $\ucI_2$ are {\em functorially equivalent} (resp. {\em functorially dominated}) if they stay equivalent (resp. dominated) after any noetherian base change $B'\to B$ and after pullback to any $Y'$ for a logarithmically regular morphism $Y'\to Y$. We will use notation $\ucI_1\preccurlyeq \ucI_2$ and $\ucI_1\approx\ucI_2$ to denote functorial domination and functorial equivalence.

\begin{remark}\label{equivrem}
(i) Our definition of equivalence extends that of \cite[\S5.1]{ATW-principalization}. Similarly to the approach of Bierstone-Milman \cite{Bierstone-Milman-funct}, our proof of independence of maximal contact is based on equivalence, rather than \'etale isomorphisms used in \cite{Wlodarczyk,ATW-principalization}.

(ii) Non-functorial equivalence is not informative in general since non-resolved marked ideals may admit no non-trivial admissible blow ups, see Example~\ref{stuckexample}.

(iii) In fact, we will only need base changes $B'\to B$ with $B'=\Spec(O)$, where $O$ is either $k(B)$ or a DVR with $\Frac(O)=k(B)$, and pullbacks with respect to morphisms $Y'\to Y$ which are either strict \'etale or localizations. We prefer not to restrict to these cases in the definition for aesthetical reasons. In the absolute case with $B$ the spectrum of a field, it even suffices to use the usual equivalence.

(iv) Naively, one might only want to test (or define) equivalence using only blow up sequences. Already in the classical algorithm this is not enough, so Hironaka and Bierstone-Milman had to introduce additional test morphisms in the definition of equivalence, see \cite[Definition~1.3]{infinitely_near} and \cite[\S1.2]{Bierstone-Milman-funct}. They might look a bit artificial, but logarithmic geometry clarifies the situation: these test morphisms $(Y',E')\to(Y,E)$ are logarithmically smooth, so including them just makes the equivalence partially functorial in the logarithmic sense.

In our principalization algorithm, base changes and logarithmically regular morphisms $Y'\to Y$ are such additional test morphisms. Their role is absolutely clear -- we want the algorithm to be functorial with respect to these classes of morphisms.
\end{remark}

\subsubsection{Main examples of equivalence}
The main cases of functorial equivalence of marked ideals follow from Lemma~\ref{operlem}.

\begin{lemma}\label{equivlem}
Assume that $f\:Y\to B$ is a logarithmically regular morphism of logarithmic DM stacks with $B\in\bB$, $\sigma\:Y'\dashrightarrow Y$ is a sequence of $B$-submonomial Kummer blow ups, and $\ucI=\ucI_1\.\ucI_n$ are marked ideals on $Y$.

(o) If $\ucI_1\subseteq\ucI_2$, then $\ucI_1\preccurlyeq\ucI_2$. In particular, if $\ucI_1\subseteq\ucI_2\subseteq\ucI_3$ and $\ucI_1\approx\ucI_3$, then $\ucI_1\approx\ucI_2\approx\ucI_3$.

(i) If  $({\cI}_i,a_i)\preccurlyeq ({\cI}_1,a_1)$ for $2\leq i\leq n$, then $(\cI_2,a_2)\cdot\ldots\cdot({\cI}_m,a_m)\preccurlyeq({\cI}_1,a_1)$.

(ii) $\ucI\approx\ucI^k$ for any positive natural $k$.

(iii) If $({\cI}_i,a_i)\preccurlyeq ({\cI}_1,a_1)$ for $2\leq i\leq n$, then $(\cI_1,a_1)+\ldots+({\cI}_m,a_m)\approx ({\cI}_1,a_1)$.

(iv) $\ucI\approx\ucI+\cF^{(\le 1)}\ucI+\ldots+\cF^{(\le a-1)}\ucI$ for any $\cO_Y$-submodule $\cF\subseteq\cD_{Y/B}$.
\end{lemma}
\begin{proof}
Claim (o) is trivial. Claims (i)--(iv) follow from their counterparts in Lemma \ref{operlem} and the previous claims. For example, (iv) follows from (iii) and the fact that $\cF^{(\le i)}\ucI\preccurlyeq\ucI$ by Lemma~\ref{operlem}(iv).
\end{proof}

\subsection{Coefficients ideals}

\subsubsection{The definition}
As in \cite[\S6.1]{ATW-principalization}, given a marked ideal $\ucI=(\cI,a)$ and an $\cO_Y$-submodule $\cF\subseteq\cD_{Y/B}$ we define the  {\em $\cF$-coefficient marked ideal} of $\ucI$ by the formula $\ucC_\cF(\ucI):=\sum_{i=0}^{a-1}\cF^{(\le i)}\ucI$. {Note that we are using the homogenized addition of Section \ref{Homogenized-addition}}. In particular, for $\cF=\cD_{Y/B}$ one obtains the {natural \emph{relative} logarithmic version $\ucC_{Y/B}(\ucI)$ of the coefficient ideal defined in  \cite[\S6.1]{ATW-principalization}}, again homogenized according to Section \ref{Homogenized-addition}.

\begin{remark}
Usually, one defines the coefficient ideal using non-homogenized sums. Our definition is an analogue of Kollar's tuning ideal $W_{a!}(\cI)$, see \cite{Kollar}.
\end{remark}

\subsubsection{Equivalence}
Lemmas~\ref{equivlem}(v) and \ref{maxorderlem} {yield} the following result.

\begin{lemma}\label{coefcor}
Assume that $Y\to B$ is a logarithmically regular morphism of logarithmic DM stacks with $B\in\bB$, $\ucI$ is a marked ideal on $Y$ and $\cF\subseteq\cD_{Y/B}$ is an $\cO_Y$-submodule. Then $\ucI\approx\ucC_\cF(\ucI)$, and $\ucC_\cF(\ucI)$ is of maximal order whenever $\ucI$ is of maximal order.
\end{lemma}

\subsubsection{Transforms}
The following result is an analogue of \cite[Proposition 6.1.3]{ATW-principalization}, which {makes the statement sharper}. Somewhat surprisingly, compatibility of our version of the coefficient ideal with transforms is as good as possible.

\begin{theorem}\label{coefprop}
Assume that $Y\to B$ is a logarithmically regular morphism of logarithmic DM stacks with $B\in\bB$, $\ucI$ is a marked ideal of maximal order, $\cF\subseteq\cD_{Y/B}$ is an $\cO_Y$-submodule, and $\sigma\:Y'\dashrightarrow Y$ is an $\ucI$-admissible sequence of $B$-submonomial Kummer blow ups. Then
$$\ucC_{\sigma^c(\cF)}(\sigma^c(\ucI))=\sigma^c(\ucC_\cF(\ucI)).$$
\end{theorem}
\begin{proof}
The inclusion $\sigma^c(\ucC_\cF(\ucI)) \subseteq \ucC_{\sigma^c(\cF)}(\sigma^c(\ucI))$ follows from Lemma~\ref{operlem}, so let us prove that $\ucC_{\sigma^c(\cF)}(\sigma^c(\ucI))\subseteq\sigma^c(\ucC_\cF(\ucI)).$ Induction on the length reduces the claim to the case of a single blow up. We should prove that if $n_0\.n_{a-1}\in\NN$ satisfy $\sum_{i=0}^{a-1}n_i(a-i)\ge a!$, then $$\cI_n:=\prod_{i=0}^{a-1}\left(\cF_\sigma^{(\le i)}(\sigma^{-1}(\cI,a))\right)^{n_i}\subseteq\sigma^{-1}(\cC_\cF(\ucI),a!).$$
By Lemma~\ref{Flem} we have that
$$\cI_n=\prod_{i=0}^{a-1}\left(\sum_{j=0}^i\sigma^{-1}\left(\cF^{(\le j)}(\cI),a-j\right) \right)^{n_i}=\sum_{(n_{ij})}\prod_{0\le j\le i\le a-1}\left(\sigma^{-1}\left(\cF^{(\le j)}(\cI),a-j\right)\right)^{n_{ij}},$$
where the right hand sum is over the sets of partitions $n_i=n_{i0}+\ldots+n_{ii}$ for $0\le i\le a-1$. For each $j\in\{0\.a-1\}$ set $l_j=\sum_{i=0}^{a-1}n_{ij}$. Then $$\sum_{j=0}^{a-1}l_j(a-j)=\sum_{0\le j\le i\le a}n_{ij}(a-j)\ge\sum_{0\le j\le i\le a}n_{ij}(a-i)=\sum_{i=0}^{a-1}n_i(a-i)\ge a!$$ and hence
\begin{align*}
\prod_{0\le j\le i\le a-1}\left(\sigma^{-1}\left(\cF^{(\le j)}(\cI),a-j\right)\right)^{n_{ij}}=\prod_{j=0}^{a-1}\left(\sigma^{-1}\left(\cF^{(\le j)}(\cI),a-j\right)\right)^{l_{j}}=\\
\sigma^{-1}\left(\prod_{j=0}^{a-1}\left(\cF^{(\le j)}(\cI)\right)^{l_{j}},\sum_{j=0}^{a-1}l_j(a-j)\right)\subseteq
\sigma^{-1}\left(\prod_{j=0}^{a-1}\left(\cF^{(\le j)}(\cI)\right)^{l_{j}},a!\right)\subseteq\sigma^{-1}(\cC_\cF(\ucI),a!).
\end{align*}
\end{proof}

\section{The theory of maximal contact}\label{contactsec}
\addtocontents{toc}
{\noindent Maximal contact and its existence. Restriction onto maximal contact. Equivalence theorem.}

\subsection{Existence}

\subsubsection{Hypersurfaces of maximal contact}
Let $\ucI=(\cI,a)$ be a marked ideal of maximal order on a logarithmically regular morphism $Y\to B$. A \emph{hypersurface of maximal contact} or simply a {\em maximal contact} to $\ucI$ is a closed suborbifold $H\into Y$ of pure codimension 1 such that its ideal $\cI_H$ is contained in $\cT(\ucI):=  \cD_{Y/B}^{(\leq a-1)}(\cI)$.

\subsubsection{Local existence}
In the classical situation, maximal contact exists Zariski locally. For DM stacks one should use the \'etale topology instead.

\begin{theorem}\label{contactexists}
Assume that $X\to B$ is a relative logarithmic orbifold with $B\in\bB$ and $\ucI$ is a marked ideal on $X$ of maximal order. Then there exists an \'etale covering $g\:X'\to X$ with a hypersurface of maximal contact $H'\into X'$ to $\ucI'=g^{-1}\ucI$.
\end{theorem}
\begin{proof}
Since $\cD_{X/B}^{(\le a)}(\cI)=1$ by Lemma \ref{maxordertransform}(i), one can find an \'etale covering $X'\to X$ and global sections $h\in\Gamma(X',\cT(\ucI))$ and $\partial\in\Gamma(\cD_{X'/B})$ such that $\partial(h)$ is a unit. By Lemma~\ref{logordlem} $\logord(h)\le 1$ on $X'$, hence $H'=V(h)$ is a logarithmic $B$-suborbifold as required.
\end{proof}

\subsection{Equivalence of marked ideals on suborbifolds}
In order to show that restriction to maximal contact preserves equivalence of certain marked ideals we should first extend the notion of equivalence that was defined in \S\ref{equivsec}. Our goal is to compare marked ideals $\cI_i$, $i=1,2$ on logarithmic $B$-orbifolds $Y_i$ via embeddings of $Y_i$ as $B$-suborbifolds into a logarithmic $B$-orbifold $Y$. We will push forward sequences of Kummer blow ups of $Y_i$ and compare the obtained sequences of Kummer blow ups of $Y$.

\subsubsection{$H$-admissibility}
Given a logarithmic relative morphism of logarithmic DM stacks $Y\to B$ with a $B$-suborbifold $i\:H\into Y$ let $\cI_H\subseteq\cO_Y$ be the ideal defining $H$ and $\ucI_H=(\cI_H,1)$. A submonomial Kummer blow up sequence $\sigma\:Y'\dashrightarrow Y$ is called {\em $H$-admissible} if it is $\ucI_H$-admissible.

\begin{lemma}\label{Hadmisslem}
Keep the above notation. Then a sequence $\sigma\:Y'\dashrightarrow Y$ is {\em $H$-admissible} if and only if it is a pushforward $i_*(\tau)$ of a submonomial Kummer blow up sequence $\sigma\:H'\dashrightarrow H$. Furthermore, in this case $\ucI_{H'}=\sigma^c(\ucI_H)$.
\end{lemma}
\begin{proof}
Induction on the length of $\sigma$ reduces the claim to the case when it is a single Kummer blow up. Let $\cJ$ be its center. Since $\sigma$ is $H$-admissible if and only if $\cI\subseteq\cJ$, the first claim reduces to unravelling the definitions, see \S\ref{pushsec}.

To check that $\ucI_{H'}=\sigma^c(\ucI_H)$ we can work \'etale-locally on $Y$, and then one can assume that there exist regular parameters $t_1\.t_n$ such that $H=V(t_1\.t_l)$ and $\cJ=(t_1\.t_m,u_1\.u_s)$ for $l\le m\le n$ and monomials $u_1\.u_s$. In this case, for any $y\in\{t_1\.t_m,u_1\.u_s\}$ the restrictions of both $\cI_{H'}$ and $\sigma^{-1}(\cI_H)(1)$ to the $y$-chart equal $(y^{-1}t_1\.y^{-1}t_l)$. (In particular, both are trivial when $y\in\{t_1\.t_l\}$.) So, the corresponding marked ideals are equal.
\end{proof}

\subsubsection{$(H,\ucI)$-admissibility}
If $H\into Y$ is a suborbifold and $\ucI$ is a marked ideal on $H$, then a sequence $Y'\dashrightarrow Y$ is called {\em $(H,\ucI)$-admissible} if it is $H$-admissible and the induced sequence $H'\dashrightarrow H$ is $\ucI$-admissible. Thus, pushforward induces a bijection between $\ucI$-admissible sequences $H'\dashrightarrow H$ and $(H,\ucI)$-admissible sequences $Y'\dashrightarrow Y$.

\subsubsection{Equivalence on suborbifolds}\label{equivsubsec}
As in \S\ref{equivsec}, a pair $(H_1,\ucI_1)$ {\em dominates} $(H_2,\ucI_2)$ if $(H_1,\ucI_1)$-admissibility implies $(H_2,\ucI_2)$-admissibility, and equivalence is defined as mutual domination. Equivalence or domination is {\em functorial} if it is preserved under noetherian base changes $B'\to B$ and pullbacks with respect to logarithmically regular morphisms $Y'\to Y$. The functorial equivalence class of the pair $(H,\ucI)$ will be denoted $[H,\ucI]$, and we say that a sequence $Y'\dashrightarrow Y$ is {\em $[H,\ucI]$-admissible} if it is $(H,\ucI)$-admissible.

\subsubsection{Pushforwards}\label{markedpushsec}
If $i\:H\into Y$ is a suborbifold and $\ucI=(\cI,1)$ is a marked ideal on $H$ of weight 1, then we set $i_*\ucI=(i_*\cI,1)$. It is easy to see (and can be deduced from Lemma~\ref{restrictlem}(ii) below) that $(H,\ucI)\approx(Y,i_*\ucI)$.

\subsection{Restriction to maximal contact}

\subsubsection{Basic compatibilities}
Operations on marked ideals and restriction onto a suborbifold are related as follows.

\begin{lemma}\label{restrictlem}
Assume that $Y\to B$ is a logarithmically regular morphism of logarithmic DM stacks with $B\in\bB$ and $i\:H\into Y$ a $B$-suborbifold.

(i) If $\cJ$ is an $H$-admissible submonomial Kummer ideal on $Y$ and $a\ge 1$, then $\cJ^{(a)}|_H\subseteq(\cJ|_H)^{(a)}$.

(ii) If $\ucI$ is a marked ideal on $Y$ and $\tau\:H'\dashrightarrow H$ is an $\ucI|_H$-admissible sequence of submonomial Kummer blow ups, then the pushforward $\sigma=i_*(\tau)\:Y'\dashrightarrow Y$ is $\ucI$-admissible and $\sigma^c(\ucI)|_{H'}=\tau^c(\ucI|_H)$.
\end{lemma}
\begin{proof}
These follows from the definitions, and induction in the case of (ii).
\end{proof}

\subsubsection{$H$-contracting modules of derivations}
Let $H$ be a $B$-suborbifold of pure codimension 1 in $Y$. As in \cite[\S6.2.3]{ATW-principalization}, a module $\cF\subseteq\cD_{Y/B}$ is called {\em $H$-contracting} if $\cF(\cI_H)=\cO_Y$. This property is preserved under submonomial Kummer blow ups:

\begin{lemma}\label{Hcontract}
If $\cF$ is $H$-contracting, $\sigma\:Y'\to Y$ an $H$-admissible sequence of submonomial Kummer blow ups, and $H'$ is the strict transform of $H$, then $\cF'=\sigma^c(\cF)$ is $H'$-contracting.
\end{lemma}
\begin{proof}
The proof is the same as in \cite[Lemma 6.2.4]{ATW-principalization}. By \'etale localization and induction on the length this reduces to the case of a single blow up along $\cJ=(t_1\.t_n,u_1\.u_s)$ such that $\cI_H=(t_1)$ and there exists $\partial\in\Gamma(\cF)$ with $\partial(t_1)$ a unit. On a $y$-chart one has that $\cI_{H'}=(t'_1)$ for $t'_1=y^{-1}t_1$ and $\partial':=y\partial\in\Gamma(\cF')$. Then it remains to note that $\partial'(t'_1)=\partial(t_1)-t'_1\partial(y)$ is a unit along $H'$ because $\partial(t_1)$ is a unit on $H$ and hence also on $H'$.
\end{proof}

\subsubsection{Lift of admissibility}
Despite a relatively simple proof, the following proposition is the only result where a simple induction on the order of differential operators does not work, and one has to use Taylor series.

\begin{proposition}\label{liftlem}
Let $f\:Y\to B$ be a logarithmically regular morphism of logarithmic DM stacks with $B\in\bB$ and let $\ucI$ be a marked ideal on $Y$ of maximal order. Assume that $H$ is a maximal contact to $\ucI$ and $\cF\subseteq\cD_{Y/B}$ is an $H$-contracting submodule. Let $\cJ$ be a submonomial Kummer ideal. Then $\cJ$ is $\ucI$-admissible if and only if it is $(H,\ucC_{\cF}(\ucI)|_H)$-admissible.
\end{proposition}
\begin{proof}
Set $\ucC=\ucC_{\cF}(\ucI)$ and $\cC_i:=\cF^{(\le i)}(\cI)^{a!/(a-i)}$, where $0\le i\le a-1$.
By Remark~\ref{sumrem}(ii) we have $\ucC\approx(\cC,a!)$,
where $\cC=\sum_{i=0}^{a-1}\cC_i$. Recall that $\ucI\approx\ucC$ by Lemma \ref{coefcor}, hence $\cJ$ is $\ucI$-admissible if and only if $\cC\subseteq\cJ^{(a!)}$. If $\cJ$ is $\ucI$-admissible, then $$\cC|_H\subseteq\cJ^{(a!)}|_H\subseteq(\cJ|_H)^{(a!)}$$ by Lemma~\ref{restrictlem}(i), so $\cJ|_H$ is $\ucC|_H$-admissible. In addition, $\cI_H^{a!}\subseteq\cC_{a-1}\subseteq\cJ^{(a!)}$, hence $\cI_H\subseteq\cJ^\nor$ and $\cJ$ is $H$-admissible.

Conversely, we should prove that if $\cC_i|_H\subseteq(\cJ|_H)^{(a!)}$ for $0\le i\le a-1$ and $\cI_H\subseteq\cJ^\nor$, then $\cI\subseteq\cJ^{(a)}$. This can be checked \'etale-locally, so we can assume that $Y$ and $B$ are schemes with Zariski logarithmic structure, $H=V(h)$, and $\partial\in\Gamma(\cD_{Y/B})$ is such that $u=\partial(h)$ is a unit. Moreover, by flatness of completions we can work on formal completions at a point $y\in Y$, where, crucially, we have an isomorphism $$\hatcO_{Y,y}=\hatcO_{H,y}\llbracket h\rrbracket.$$ Such an isomorphism exists by Lemma~\ref{logregchart} applied to the sharp factorization $Y\to\tilB$ of $f$ (see \S\ref{sharpfactorsec}) and a family of regular parameters $t_1=h,t_2\.t_n$. For any $\phi\in\cI_y$ we obtain a Taylor series presentation $\phi=\sum_{i=0}^\infty c_ih^i$ with $\frac{1}{i!u^i}\partial^i(\phi)|_H=c_i$. Then $c_i\in(\hatcJ_{H,y}^{a-i})^\nor\subset(\hatcJ_y^{a-i})^\nor$ and $h^a\in\cJ^{(a)}$, and hence $\phi\in(\hatcJ_y^a)^\nor$.
\end{proof}

\subsubsection{Equivalence of the restriction}
Now, we are ready to prove the main theorem of the theory of maximal contact.

\begin{theorem}\label{Htheorem}
Assume that $Y\to B$ is a logarithmically regular morphism of logarithmic DM stacks with $B\in\bB$, $\ucI$ is a marked ideal of maximal order on $Y$, $i\:H\into Y$ is a hypersurface of maximal contact to $\ucI$, and $\cF\subseteq\cD_{Y/B}$ is an $H$-contractible $\cO_Y$-submodule. Then $(Y,\ucI)\approx(H,\ucC_{\cF}(\ucI)|_H)$.
\end{theorem}
\begin{proof}
Recall that $\ucI\approx\ucC:=\ucC_\cF(\ucI)$ by Lemma \ref{coefcor}, hence $(Y,\ucI)\preccurlyeq(H,\ucC_H)$, where we set $\ucC_H=\ucC|_H$. The main part is to prove the opposite domination. By induction on the length of the sequences it suffices to prove the following assertion:

Assume that $\tau\:H'\dashrightarrow H$ is $\ucC_H$-admissible and such that its pushforward $\sigma=i_*(\tau)\:Y'\to Y$ is $\ucI$-admissible, and assume that $\cJ$ is an $H'$-admissible submonomial Kummer ideal on $Y'$ such that $\cJ|_{H'}$ is $\tau^c(\ucC_H)$-admissible. Then $\cJ$ is $\ucI'$-admissible, where $\ucI'=\sigma^c(\ucI)$.

To prove the assertion, recall that by Theorem \ref{coefprop} $\ucC_{\cF'}(\ucI')=\sigma^c(\ucC)$, where $\cF'=\sigma^c(\cF)$. Restricting this onto $H'$ and applying Lemma \ref{restrictlem}(ii) we obtain that $$\ucC_{\cF'}(\ucI')|_{H'}=\sigma^c(\ucC)|_{H'}=\tau^c(\ucC_H),$$ and hence $\cJ|_{H'}$ is $\ucC_{\cF'}(\ucI')|_{H'}$-admissible. Since $\cF'$ is $H'$-contractible by by Lemma~\ref{Hcontract}, $\cJ$ is $\ucI'$-admissible by Proposition~\ref{liftlem}.
\end{proof}

\subsection{Functoriality}
We conclude Section \ref{contactsec} with a brief explanation of the fact that all constructions with marked ideals are compatible with logarithmically regular morphisms and noetherian base changes. If $\ucI=(\cI,a)$ is a marked ideal on $Y$ and $g\:Y'\to Y$ is a morphism, then the {\em pullback} marked ideal is $g^{-1}(\ucI)=(g^{-1}(\cI),a)$.

\begin{lemma}\label{functormarked1}
Let $f\:Y\to B$ be a logarithmically regular morphism of logarithmic DM stacks with $B\in\bB$, $\ucI$ a marked ideal on $Y$, $\cF\subseteq\cD_{Y/B}$ an $\cO_Y$-submodule, $g\:B'\to B$ a morphism of logarithmic DM stacks with $B'\in\bB$ and noetherian\footnote{Noetherian in the sense of  \S\ref{convsec}.} fs base changes $g'\:Y'\to Y$ and $f'\:Y'\to B'$, and $\ucI'=g'^{-1}\ucI$, $\cF'=g^*(\cF)$. Then

(i) Assume that $H$ is a maximal contact to $\ucI$. Then $H'=H\times_YY'$ is a maximal contact to $\ucI'$. In addition, if $\cF$ is $H$-contracting, then $\cF'$ is $H'$-contracting.

(ii) $g'^{-1}(\ucC_\cF(\ucI))=\ucC_{\cF'}(\ucI')$.

(iii) $g'^{-1}(\cW(\ucI))=\cW(\ucI')$ and $\mu_\ucI\circ g'=\mu_{\ucI'}$.

(iv) If a submonomial Kummer blow up sequence $Y_n\dashrightarrow Y$ is admissible for $\ucI$ (resp. is an order reduction of $\ucI$), then the same is true for the pullback sequence $Y'_n\dashrightarrow Y$ and $\ucI'$.
\end{lemma}
\begin{proof}
Claim (iii) follows from Lemma~\ref{functorlogorder1}, and claims (i) and (ii) are proved by unravelling the definitions. For example, if $\cF$ is $H$-contracting, then $\cF(\cI_H)=\cO_Y$ and applying $g^{-1}$ yields $\cF'(\cI_H')=\cO_{Y'}$, that is, $\cF'$ is $H'$-contracting.
\end{proof}

In the same way, but using  Lemma~\ref{functorlogorder2} as an input one proves

\begin{lemma}\label{functormarked2}
Assume that $g\:Y'\to Y$ and $f\:Y\to B$ are logarithmically regular morphism of logarithmic DM stacks with $B\in\bB$, $\cI\subseteq\cO_Y$ an ideal with $\cI'=g^{-1}\cI$, and $\cF\subseteq\cD_{Y/B}$ a submodule with $\cF'=g^*(\cF)$. Then,

(i) Assume that $H$ is a maximal contact to $\ucI$. Then $H'=H\times_YY'$ is a maximal contact to $\ucI'$. In addition, if $\cF$ is $H$-contracting, then $\cF'$ is $H'$-contracting.

(ii) $g^{-1}(\ucC_\cF(\ucI))=\ucC_{\cF'}(\ucI')$.

(iii) $g^{-1}(\cW(\ucI))=\cW(\ucI')$ and $\mu_\ucI\circ g=\mu_{\ucI'}$.

(iv) If a submonomial Kummer blow up sequence $Y_n\dashrightarrow Y$ is admissible for $\ucI$ (resp. is an order reduction of $\ucI$), then the same is true for the pullback sequence $Y'_n\dashrightarrow Y'$ and $\ucI'$.
\end{lemma}

\section{Relative logarithmic order reduction}\label{princsec}
\addtocontents{toc}
{\noindent We describe invariants of equivalence classes of ideals, sufficient to prove a logarithmic version of Hironaka's trick. We use this to construct  relative order reduction and principalization.}

In this section we will construct the logarithmic relative principalization method and prove Theorem~\ref{Proj:principalization}.

\subsection{Reduction to marked ideals}
By an {\em order reduction method} we mean a method $\cF$, such that given a marked ideals $\ucI$ on a relative logarithmic orbifold $f\:X\to B$, it outputs either an order reduction $\cF(f,\ucI)\:X'\dashrightarrow X$ of $\ucI$ or the empty value. As in the classical case, relative logarithmic principalization of $\cI$ is nothing else but a relative order reduction of the marked ideal $(\cI,1)$. Thus Theorem~\ref{Proj:principalization} is a particular case of the following

\begin{theorem}[Order reduction]\label{mainth}
There exists an order reduction method $\cF$ satisfying the following properties:

(i) Existence: let $f\:X\to B$ be a relative logarithmic orbifold with a marked ideal $\ucI$, and assume that $B\in\bB$ and either $\dim(B)\le 1$ or $f$ has abundance of derivations. Then there exists a blow up $g\:B'\to B$ with the saturated pullback $f'\:X'\to B'$ and $\cI'=\cI\cO_{X'}$ such that $\cF(f',\cI')\neq\emptyset$ and the center of $g$ is monomial over the complement to the closure of the image of $V(\cI)$ in $B$.

(ii) Compatibility with base change: if $\cF(f,\ucI)\neq\emptyset$ and $g\:B'\to B$ is any morphism of logarithmic stacks from $\bB$ with noetherian saturated base changes $f'\:X'\to B'$ and $g'\:X'\to X$, and $\ucI'=g'^{-1}\ucI$, then the sequence $\cF(f',\ucI')$ is obtained from the saturated pullback sequence $\cF(f,\ucI)\times_BB'$ by removing Kummer blowings up with empty centers.

(iii) Functoriality: if $\cF(f,\ucI)\neq\emptyset$ and $\ucI'=g^{-1}\ucI$ for a logarithmically regular morphism $g\:X'\to X$ such that $X'\to B$ is a relative logarithmic orbifold, then $\cF(f',\ucI')$ is obtained from the saturated pullback sequence $\cF(f,\ucI)\times_XX'$ by removing Kummer blowings up with empty centers.

(iv) Dependence on the equivalence class: if $\of\:\oX\to B$ is another relative logarithmic orbifold and $i\:X\into\oX$ is a $B$-suborbifold embedding of a pure codimension, then the pushforward sequence  $i_*(\cF(f,\ucI))$ depends only on $\of$ and the functorial equivalence class $[X,\ucI]$. In particular, if $\ucI = (\cI,1)$ is of weight 1, then $i_*(\cF(f,\ucI))=\cF(\of,i_*\ucI)$.
\end{theorem}

\begin{remark}
Since there are many functorially equivalent marked ideals on $X$, part (iv) provides a non-trivial addendum even when $X=\oX$. However, it is important for our argument to prove (iv) for an arbitrary $i$.
\end{remark}

The rest of Section \ref{princsec} is devoted to proving the theorem.

\subsection{Invariants of functorial equivalence classes}
Similarly to its predecessors, our order reduction method runs by restricting to suborbifolds $i\:H\into X$. In this paper, we prove independence of choices using equivalence of marked ideals on different suborbifolds as defined in \S\ref{equivsubsec}. As a preparation, we are now going to study invariants of equivalence classes.

\subsubsection{A model case}\label{modelsec}
We start with the following simple observation. Assume that $X$ is a smooth scheme with a marked ideal $\ucI=(\cI,a)$ and a closed point $x\in X$. If $b=\ord_x(\cI)$ and $\sigma\:X'\to X$ is the blow up at $x$, then $\cI_E^b$ is the maximal power of $\cI_E$ that divides $\sigma^{-1}(\cI)$. Therefore, the sequence $X'(d)\to X'\to X$ with centers $m_x$ and $\cI_E^d$ is $\ucI$-admissible if and only if $da\le b-a$, that is $d\le \mu_x(\ucI)-1$. In our approach we allow $d$ to be a rational number, hence the equivalence class of $\ucI$ determines the normalized order in a very elementary way: $\mu_x(\ucI) = \max d+1$, the maximum taken over the above small class of admissible sequences. For comparison, in the classical approach one blows up only smooth centers, and the goal of determining an invariant from the collection of allowed transformations is achieved by Hironaka's trick, where one first replaces $X$ by $X\times \bfA^1$ with the logarithmic structure given by $X\times\{0\}$.

In general, we will have to solve a few minor technical problems: $x$ might not be closed, $m_x$ might not be submonomial, and the codimension of $H$ is not determined by $[H,\ucI]$. The following result will be used to deal with the first two issues.

\begin{lemma}\label{submonomiallem}
Let $f\:X\to B$ be a relative logarithmic orbifold and $x\in |X|$ a point. Then there exists a morphisms $B'\to B$ with $B'\in\bB$ and a noetherian $X\times_BB'$, a localized \'etale morphism $X'\to X\times_BB'$ with $X'$ a scheme, and a closed point $x'$ with image $x\in|X|$ such that the ideal $m_{x'}$ is submonomial.
\end{lemma}
\begin{proof}
It is easy to see that there always exists a morphism of logarithmic schemes $g\:B'=\Spec(O)\to B$ such that $O$ is either a field or a DVR, the logarithmic structure is $O\setminus\{0\}$, and $g$ takes the closed point $b'\in B'$ to $b:=f(x)$. Find an \'etale cover $X'\to X\times_BB'$ with $X'$ a scheme, and choose a point $x'\in X'$ over $x$. Then replacing $X'$ by its localization at $x'$ we can assume that $x'$ is closed. Now, the fiber $X'_{b'}$ is a monomial subscheme, and it is easy to see that the logarithmic fiber $S=S_{x'}$ coincides with the logarithmic stratum of $x'$ in $X'_{b'}$, hence $S_{x'}$ is monomial. Since $x'$ is a closed point in $S_{x'}$, the ideal $m_{x'}$ is submonomial.
\end{proof}

\begin{remark}\label{submonomialrem}
One could use other classes of base changes in the lemma. For example, one could replace $B$ by the localization at $B$, increase the logarithmic structure so that $m_b$ becomes monomial, and then apply a monomial blow up so that the base change of $f$ becomes exact at $x$.
\end{remark}

\subsubsection{Operations on equivalence classes}
By definition, functorial equivalence is preserved by base changes of $B$, pullbacks with respect to logarithmically regular morphisms $g\:X'\to X$, and controlled transforms under admissible blow us. Hence all these operations are defined on equivalence classes $[H,\cI]$. For example, $g^{-1}([H,\ucI])$ is defined to be $[H\times_XX',(g|_H)^{-1}\ucI]$.

\subsubsection{Codimensions}
The codimension of $H$ at a point $x\in|X|$ will be denoted $c_x(H)$. Let $\cC$ be a functorial equivalence class of marked ideals on suborbifolds of $X$. By the {\em codimension} $c_x(\cC)$ of $\cC$ at $x\in |X|$ we mean the maximal possible codimension $c_x(H')$, where $g\:X'\to X$ is an \'etale neighborhood of $x$ and $(H',\ucI')$ is a representative of $g^{-1}(\cC)$. Finally, if $\cJ$ is a submonomial Kummer ideal of the form $\cI_H+\cN$, where $\cN$ is Kummer monomial and $H$ is a suborbifold, then the number $c_x(H)$ is easily seen to depend only $\cJ$ and we call it the {\em monomial codimension} $c_x(\cJ)$ of $\cJ$ at $x$. We have the following natural inequality between these codimensions:


\begin{lemma}\label{codimlem}
Assume that $\cC$ is a functorial equivalence class on $X$ and a sequence $\sigma\:X_n\dashrightarrow X_0=X$ with centers $\cJ_i\subseteq\cO_{(X_i)\ket}$ is $\cC$-admissible. Let $i\in\{0\. n-1\}$, $x_i\in\supp(\cJ_i)$ and let $x\in|X|$ be its image. Then $c_{x_i}(\cJ_i)\ge c_{x}(\cC)$.
\end{lemma}
\begin{proof}
Replacing $X$ by an \'etale neighborhood of $x$ and replacing $\sigma$ by its pullback we can assume that $\cC$ contains a representative $(H,\cI)$ with $c_x(H)=c_x(\cC)$. Then the strict transform $H_i\into X_i$ satisfies $c_{x_i}(H_i)=c_x(H)$ and $\cI_{H_i}\subseteq\cJ_i$. Therefore, $c_{x_i}(\cJ_i)\ge c_{x_i}(H_i)=c_x(\cC)$, as required.
\end{proof}

Now we can show that codimensions detect whether a marked ideal is of maximal order.

\begin{theorem}\label{codimth}
Let $X\to B$ be a relative logarithmic orbifold, $i\:H\into X$ a suborbifold, $\ucI$ a marked ideal on $H$ with equivalence class $\cC=[H,\ucI]$, and $x\in\supp(\cI)$ a point. Then $c_x(\cC)>c_x(H)$ if and only if $\ucI$ is of maximal order at $x$.
\end{theorem}
\begin{proof}
Assume that $\ucI$ is of maximal order at $x$. We can replace $X$ by an \'etale neighborhood of $x$, hence by Theorem~\ref{contactexists} we can assume that there exists a maximal contact $H'\into H$ to $\ucI$. By Theorem~\ref{Htheorem}, $\ucC(\ucI)|_{H'}$ is another representative of $\cC$, and its codimension is larger than that of $\ucI$.

Conversely, it suffices to obtain a contradiction to the following assumption: $\ucI=(\cI,a)$ is not of maximal order at $x$, but $c_x(\cC)>c_x(H)$. This situation persists after base changes and pullbacks with respect to \'etale localizations of $X$, hence by Lemma~\ref{submonomiallem} we can assume that $X$ is local, $x$ is closed and $m_{X,x}$ is submonomial. Then $m_{H,x}$ is submonomial too, say, $m_{H,x}=(t,u)$, where $t=(t_1\.t_n)\subset m_{H,x}$ is a family of regular parameters at $x$ and $u=(u_1\.u_r)$ monomials generating the maximal ideal of $\oM_x$.

By our assumption $\logord_{\ucI/B}(x)\ge d:=a+1$. Since $\cI\subseteq (t)^d+(u)\subseteq(t,u^{1/d})^d$, the blow up $\tau\:H'\to H$ along $(t,u^{1/d})$ is $(\cI,d)$-admissible, and hence, $\cI_E$ divides $\tau^{-1}(\cI,a)$. Let $H'(1/a)\to H'$ be the blow up along $\cI_E^{1/a}$. Then the sequence $H'(1/a)\to H'\to H$ is $\ucI$-admissible, and its pushforward $X'(1/a)\to X'\to X$ is $\cC$-admissible. Since the center of $X'(1/a)\to X'$ is of monomial codimension $c_x(H)$, Lemma~\ref{codimlem} yields a contradiction.
\end{proof}

\subsubsection{Normalized invariants}
One can define analogues of normalized invariants for equivalence classes $[H,\ucI]$. Since $H$ is not determined by the class, we take the normalized invariants of $\ucI$ and push them forward to $X$. Namely, by $\mu_{[H,\ucI]}\:|X|\to\NN$ we denote the normalized logarithmic order function $\mu_\ucI\:|H|\to\NN$ extended by zero outside of $|H|$, and by $\cW_{X/B}([H,\ucI])$ we denote the submonomial Kummer ideal $i_*(\cW_{H/B}(\ucI))$. We have slightly abused notation because these invariants certainly depend also on the codimension $c_H$ of $H$. In fact, they depend only on $[H,\ucI]$ and $c_H$, as will be proven soon.
\begin{remark}
It is important to push forward the invariants of $\cI$, instead of taking invariants of $i_*(\ucI)$. Indeed, if $H$ is of positive codimension, then the logarithmic order of $i_*(\cI)$ is $\leq 1$. In fact, $\logord_{i_*(\cI)/B}$ is the characteristic function of $\supp(\cI)$. In addition, $\cW_{X/B}(i_*\ucI)$ is monomial, while $\cW_{X/B}([H,\ucI])$ is only submonomial in general.
\end{remark}

\subsubsection{Hironaka's trick: the logarithmic version}
The following theorem is an analogue of \cite[Theorems 6.1, 6.2]{Bierstone-Milman-funct} and its proof runs along the same line and is even more straightforward.

\begin{theorem}\label{equivth}
Assume that $f\:X\to B$ is a relative logarithmic orbifold, $i\:H\into X$ is a suborbifold of pure codimension $c_H$, and $\ucI=(\cI,a)$ is a marked ideal on $H$. Then the function $\mu_{[H,\ucI]}\:|X|\to\NN$ and the Kummer ideal $\cW_{X/B}([H,\ucI])$ depend only on $\cC=[H,\ucI]$ and $c_H$ rather than on the choice of the pair $(H,\ucI)\in\cC$.
\end{theorem}
\begin{proof}
Since $\cW_{H/B}(\ucI)$ is the maximal monomial $\ucI$-admissible Kummer ideal, $\cW_{X/B}(\cC)$ is the maximal submonomial Kummer ideal which is $\cC$-admissible and of monomial codimension $c_H$. This description only depends on $\cC$ and $c_H$.

The main task is to show that for any $x\in|X|$ the number $\mu=\mu_x(\ucI)$ is determined by $\cC$ and $c_H$. Since the logarithmic order is compatible with regular morphisms and base changes by Lemmas~\ref{functorlogorder1} and \ref{functorlogorder2}, as in the proof Theorem~\ref{codimth} we can use Lemma~\ref{submonomiallem} to reduce to the case when $X$ is a local scheme, $x$ is closed and $m_{H,x}=(t,u)$ is submonomial.

Case 0: $\mu=0$. The vanishing locus of $\mu_{[H,\ucI]}$ is precisely $X\setminus\supp(\ucI)$, and $\mu_x(\ucI)=0$ if and only if $d:=\logord_{\cI/B}(x)<a$ if and only if $m_{H,x}$ is not $\ucI$-admissible if and only if $m_{X,x}$ is not $\cC$-admissible. Thus, $\supp(\ucI)$ is detected by $\cC$ only.

Case 1: $\mu=1$. By Theorem~\ref{codimth} and Case 0, this happens if and only if $c_H>c_x(\cC)$ and $m_{X,x}$ is $\cC$-admissible.

Case 2: $\mu=\infty$. This happens if and only if there exists an $\ucI$-admissible monomial ideal $\cN\subsetneq\cO_H$ with $x\in V(\cN)$ (in fact, $\cN=(u^{1/a})$ will work). The latter happens if and only if $x$ lies in the support of a $\cC$-admissible submonomial ideal of monomial index $c_H$, the property depending on $\cC$ and $c_H$ only.

Case 3: $1<\mu<\infty$. By the above cases, this situation is also detected by $\cC$ and $c_H$. Let us show how to find $\mu$. Note that $a<d<\infty$ and for any $l\ge d$ we have that $\cI\subseteq(t)^d+(u)\subseteq (t,u^{1/l})^d$ and the blow up $\tau_l\:H_l\to H$ along $(t,u^{1/l})$ is $(\cI,d)$-admissible at $x$. For $n\in\QQ_{>0}$ let $\tau_{l,n}\:H_l(n)\to H_l$ denote the blow up along the monomial Kummer ideal $\cI_E^n$, and let $X_l(n)\stackrel{\sigma_{l,n}}\to X_l\stackrel{\sigma_l}\to X$ be the pushforward. By Corollary~\ref{balancedcor} $\tau_l^c(\ucI)$ is balanced with monomial part $\cI_E^{d-a}$, hence the sequence $\sigma$ is $\cC$-admissible if and only if $na\le d-a$. The latter happens if and only if $n\le\mu-1$. Taking $l\gg 0$ this characterizes $\mu$ once we show that the sequence $X_l(\mu-1)\to X_l\to X$ depends only on $\cC$ and $c_H$ (and not on $H$ via $\cI_E$). Clearly, $X_l\to X$ depends only on $l$, and the center of $\sigma_{l,\mu-1}$ is the maximal submonomial $\sigma^c_l(\cC)$-admissible Kummer ideal on $X_l$ of monomial codimension $c_H$, the property which depends only on $\cC$ and $c_H$.
\end{proof}

Combining the above results with Corollary \ref{balancedcor} one easily obtains the following result:

\begin{corollary}\label{equivcor}
Assume that $(H_1,\ucI_1)\approx(H_2,\ucI_2)$ with $\ucI_1$ balanced and $H_1, H_2$ of pure codimension $c$ in $X$. Then $\ucI_2$ is balanced too and $(H_1,\ucI_1^\cln)\approx(H_2,\ucI_2^\cln)$.
\end{corollary}

\subsection{The method}\label{ordersec}
Now, let us construct the method $\cF$ whose existence is asserted by Theorem~\ref{mainth}. In this section, we will only check that $\cF$ satisfies (ii), (iii) and (iv) as this is used to justify the construction. The method itself is a generalization of the absolute logarithmic order reduction in \cite[\S2.11]{ATW-principalization}.

\subsubsection{Induction scheme}\label{indscheme}
The method runs by induction on $n=\logdim(X/B)$, as defined in \S\ref{dimsec}. The induction base is trivial, say $n=-1$, and in what follows we establish the induction step for $n\ge 0$.

The dimension of $\oX$ in (iv) can be arbitrary, but the induction claim is as follows: if $\cC=[X,\ucI]$ denotes the functorial equivalence class in $\oX$, and $(\tilX,\tilucI)$ is another representative of $\cC$ with $\logdim(\tilX/B)\le n$ and the suborbifold embedding $\tili\:\tilX\into\oX$, then $i_*(\cF(f,\ucI))=\tili_*(\cF(f,\tilucI))$.

In (iii) we assume that $\logdim(X'/B)\le n$ and $\logdim(X/B)\le n$.

\subsubsection{Functoriality}
It will be clear from the construction of $\cF$ that it satisfies (ii) and (iii) because all intermediate constructions used in the process, including normalized invariants, Kummer blow ups, transforms, coefficient ideals, and maximal contacts, are compatible with logarithmically regular morphisms and base changes. This was shown in Lemmas~\ref{functorlogorder1}, \ref{functorlogorder2}, \ref{functormarked1}, \ref{functormarked2}, and Theorem~\ref{kummerblowth}. So we will only show how $\cF$ is defined and will check (iv).

\subsubsection{The maximal order case}\label{maxordercase}
First, let us construct $\sigma=\cF(f,\ucI)$ in the particular case when $\ucI$ is of maximal order. The idea is this: \'etale-locally $\ucI$ is equivalent to a marked ideal $\ucI_0$ on a maximal contact $H_0$. By induction, the order reduction of $\ucI_0$ is defined and depends only on the equivalence class of $\ucI_0$. This implies that the construction is independent of choices and hence descends to an order reduction of $\ucI$.

Now, let us work out details. By Theorem~\ref{codimth} there exists an \'etale covering $p\:X_0\to X$ such that the equivalence class $[X_0,p^{-1}(\ucI)]$ contains a representative $(H_0,\ucI_0)$ with $i_0\:H_0\into X_0$ of pure codimension 1. If $g_0=f\circ p\circ i_0$ is the relative logarithmic orbifold $H_0\to B$, then $\cF(g_0,\ucI_0)$ is already defined by induction on $n$. If it is empty, then we set $\cF(f,\ucI)=\emptyset$. Otherwise we define $\sigma_0\:X'_0\dashrightarrow X_0$ to be the pushforward of $\tau_0=\cF(g_0,\ucI_0)\:H'_0\dashrightarrow H_0$, and we claim that $\sigma_0$ descends to a Kummer blow up sequence $\sigma\:X'\dashrightarrow X$. (Note that $H_0$ does not have to descend to a suborbifold of $X$). Once we prove that $\sigma$ exists its independence of the covering $p$ follows: given another covering $p'$ consider a mutual refinement and use functoriality of $\cF$ with respect to \'etale morphisms.

Let $p_i$, $i\in\{1,2\}$ denote the projections of $X_1:=X_0\times_XX_0$ onto $X_0$. To prove that $\sigma_0$ descends it suffices to show that the two pullbacks $\sigma_i=p_i^{-1}\sigma_0$ coincide. Consider the suborbifolds $H_i=p_i^{-1}(H_0)$ of $X_1$ with the sequences $\tau_i\:H'_i\dashrightarrow H_i$ pullbacked from $\tau_0$, morphisms $g_i\:H_i\to B$ and marked ideals $\ucI_i=p_i^{-1}(\ucI_0)$. \'Etale pullback is compatible with pushing forward the blow ups sequence, hence $\sigma_i$ is the pushforward of $\tau_i$ under $H_i\into X_1$. In addition, $\logdim(H_i/B)=\logdim(H_0/B)<n$ hence by induction $\cF$ is functorial with respect to the \'etale morphisms $H_i\to H_0$ and we obtain that $\tau_i=\cF(g_i,\ucI_i)$. Also, by induction the equivalence $(H_0,\ucI_0)\approx(X_0,p^{-1}\ucI)$ is pulled back via $p_i$ to $(H_i,\ucI_i)\approx (X_1,q^{-1}\ucI)$, where $q=p\circ p_i$ is the projection $X_1\to X$. Therefore, $(H_1,\ucI_1)\approx(H_2,\ucI_2)$ and the induction assumption implies that $\sigma_1=\sigma_2$. This proves that $\sigma$ is well defined, and it is an order reduction of $\ucI$ because its pullback $\sigma_0$ to $X_0$ is an order reduction of $p^{-1}\ucI$.

\begin{remark}\label{indeprem}
The above paragraph shows that the local procedure is independent of the (\'etale-local) choice of maximal contact and hence globalizes. It critically uses equivalence on suborbifolds of different codimensions.
\end{remark}

It remains to check (iv) in the maximal order case, so let $\tili\:\tilX\into\oX$ and $(\tilX,\tilucI)\in\cC$ be as in \S\ref{indscheme}. Let $c_X$ and $c_\tilX$ be the codimensions of $X$ and $\tilX$ in $\oX$. By Theorem~\ref{equivth} $\supp(\ucI)=\supp(\tilucI)$. If the supports are empty, there is nothing to prove, so assume that this is not the case and choose a point $x\in\supp(\ucI)$. Then $c_X=\logdim_x(\tilX/B)-\logdim_x(X/B)$ and $c_\tilX=\logdim_x(\tilX/B)-\logdim_x(\tilX/B)$, and hence $c_X\le c_\tilX$.

We should check that $i_*\cF(f,\tilucI)=\tili_*\cF(\tilf,\tilucI)$. This can be done \'etale-locally on $\oX$, and since the \'etale topology of $\oX$ induces the \'etale topologies of $X$ and $\tilX$, passing to a fine enough \'etale covering of $\oX$ we can assume that there exists a logarithmic $B$-suborbifold $H\into X$ of pure codimension one and a marked ideal $\ucJ$ on $H$, such that $(H,\ucJ)\approx(X,\ucI)$ in $X$, and hence also in $\oX$. If $c_X=c_\tilX$, then $\tilucJ$ is also of maximal order, and we can assume in the same way that there exists $\tilH\into\tilX$ of pure codimension one and $\tilucJ$ on $\tilH$, such that $(\tilH,\tilucJ)\approx(\tilX,\tilucI)$. If $c_X<c_\tilX$, then we simply take $\tilH=\tilX$ and $\tilucJ=\tilucI$. Then we have equivalences $(H,\ucJ)\approx(X,\ucI)\approx(\tilX,\tilucI)\approx(\tilH,\tilucJ)$ in $\oX$, hence by the induction assumption the pushforwards of $\cF(H\to B,\ucI_0)$ and $\cF(\tilH\to B,\tilucI_0)$ to $\oX$ coincide. It remains to note that the latter are precisely $i_*\cF(f,\tilucI)$ and $\tili_*\cF(\tilf,\tilucI)$ by the construction of $\cF$ in the maximal order case.

\subsubsection{The general case}\label{generalcase}
Now, we will construct $\cF$ in general, using the maximal order case established above. The method outputs a sequence $$\sigma\:X_0\stackrel{\sigma_0}\to X_{\mu_l}\stackrel{\sigma_{\mu_l}}\dashrightarrow X_{\mu_{l-1}}\stackrel{\sigma_{\mu_{l-2}}}\dashrightarrow\ldots\ldots\stackrel{\sigma_{\mu_0}}\dashrightarrow X_{\mu_0}=X_\infty\stackrel{\sigma_\infty}\to X,$$ which will be  constructed in three steps below. It is convenient to use weights in the numbering because for each $i$ and $\mu=\mu_i$ we will have $\mu=\mu(\ucI^\cln_\mu)$, where $\ucI_{\mu}$ denotes the controlled transform of $\ucI$ to $X_{\mu}$. We will also denote the pushout sequence by $\osigma\:\oX_0\to\oX$ and consider the induced embeddings $i_{\mu}\:X_{\mu}\into\oX_{\mu}$ and equivalence classes $\cC_{\mu}=[X_{\mu},\ucI_{\mu}]$ in $\oX_{\mu}$. 

Step 1. {\em Initial cleaning.} Consider the blow up $\sigma_\infty\:X_\infty\to X$ along the monomial Kummer ideal $\cW_{X/B}(\ucI)$.
\begin{itemize} \item \emph{If the controlled transform $\ucI_\infty$ is not clean, the step outputs ``fail"}.
\item Otherwise, the step outputs $\sigma_\infty$, achieving that the controlled transform is clean, and hence balanced.
\end{itemize}
 The induced blow up $\osigma_\infty\:\oX_\infty\to\oX$ is along $\cW_{\oX/B}(\cC)=i_*\cW_{X/B}(\ucI)$, so it only depends on $\cC$ and $c_H$ by Theorem~\ref{equivth}.

\begin{remark}\label{cleanrem}
(i) By Theorem~\ref{cleaningth}, Step 1 never fails if $\cI$ is logarithmically clean.

(ii) Step 1 may also succeed for other ideals, and, moreover, it may even happen that the same equivalence class contains both logarithmically clean and not logarithmically clean ideals.
\end{remark}

Step 2. {\em Reducing the order of the clean part.} This step composes sequences $\sigma_{\mu}:=\cF(X_{\mu},\ucI_{\mu}^\cln)$, which are defined by \S\ref{maxordercase}. Starting with $\mu_0=\mu(\ucI_\infty)$ and $X_{\mu_0}=X_\infty$, we inductively define $\sigma_{\mu_i}$ by this rule and label its target by $\mu_{i+1}:=\mu((\sigma_{\mu_i}^c\ucI_{\mu_i})^\cln)$.

By induction on $i$ one obtains from Corollary~\ref{balancedcor} that each $\sigma_\mu$ is $\ucI_\mu$-admissible and $\sigma_\mu^c(\ucI_\mu)$ is balanced with the clean part being $\sigma_\mu^c(\ucI_\mu^\cln)$. In particular, the logarithmic order of the clean part drops on each step: $\mu_0>\mu_1>\dots$. Since $\mu_i\in\frac{1}{a}\NN$, after finitely many steps we arrive at $\mu_l=0$, obtaining $\ucI_{\mu_l}$ with a resolved clean part.

Let us show that this step also satisfies property (iv). The class $\cC_\mu^\cln=[X_\mu,\ucI_\mu^\cln]$ in $\oX_\mu$ is determined by $\cC_\mu$ and $c_X$ by Corollary ~\ref{equivcor}. By induction on $i$, the class $\cC_\mu$ is determined by $\cC$ and $c_X$, so it remains to note that $(i_\mu)_*(\sigma_\mu)$ is determined by $\cC_\mu^\cln$ and $c_X$ by \S\ref{maxordercase}.

Step 3. {\em Final cleaning.} At this step, the ideal is balanced with a resolved clean part, so $\sigma_0$ is simply the Kummer blow up of $\cW(\ucI_{\mu_l})$. The same argument as in step 1, shows that the pushforward to $\oX$ only depends on $\cC$ and $c_X$.

\begin{remark}
Note that the notation in the maximal order and general cases are consistent. If $\ucI$ is of maximal order and $\mu=\mu(\ucI)$, then the first and third steps are trivial and on the second step the composition reduces to taking the single sequence $\sigma_\mu=\cF(f, \ucI)$ as defined in \S\ref{maxordercase}. So, in this case the order reduction is the same as was defined in the maximal order case.
\end{remark}

\subsubsection{Addenda}
We conclude with two remarks about the algorithm.

\begin{remark}\label{uniquerem}
In the classical case, intensive study of resolution led to different descriptions of essentially the same algorithm, with the only variations in combinatorial parts of the algorithm. A natural question that interested us before starting this project, is whether the algorithm is indeed essentially unique, or this happened just because of the flow of ideas between different approaches. We expect that the first possibility is true and show that our algorithm 
is essentially unique.

It follows from the construction that our order reduction method is uniquely characterized by the following properties: (i) it only depends on the functorial equivalence class, (ii) it treats $\ucI$ and $\ucI^\cln$ in the same way, and (iii) it starts and finishes with blowing up $\cW(\ucI)$. Property (i) seems to be necessary for functorial algorithms that use maximal contact to induct on dimension. Properties (ii) and (iii) are not necessary, for example, one can first blow up $\cW(\ucI)^{1/2}$ and then the pullback of $\cW(\ucI)^{1/2}$, but it seems that avoiding them could only result in dealing with the monomial parts of ideals  in a less efficient and superficial way.
\end{remark}

\begin{remark}\label{invrem}
Similarly, to \cite[\S2.11.4]{ATW-principalization} one can assign to $\ucI$ an invariant $\inv_\ucI\:|X|\to\Inv$, where $\Inv$ is the set of sequences $(\mu_0\.\mu_n)$ with $\mu_i\in\QQ_{\ge 1}$ for $i<n$ and $\mu_n\in\QQ_{\ge 1}\cup\{0,\infty\}$. Its definition follows loc.cit. without changes: if \'etale-locally over a point $x$, one denotes by $\ucI_i$ the appropriate restriction onto the $i$-th maximal contact $H_i$, then $\mu_i(x)=\mu_x(\ucI_i^\cln)$. In particular, $\inv_\ucI$ only depends on $[X,\ucI]$.
\end{remark}

\subsection{Existence of order reduction}
It remains to establish claim (i) of Theorem~\ref{mainth}. We will do it under the additional assumption that $B\in\bB^\st$. In \S\ref{abscase}, the case $B=\Spec(\QQ)$ of the principalization theorem will be used to prove Theorem~\ref{absdesingth}, which implies that $\bB^\st=\bB$. Since $\Spec(\QQ)\in\bB^\st$, this will imply the assertion of Theorem~\ref{mainth} in full generality.

We say that a blow up $B'\to B$ is {\em permissible} if its center is monomial over the complement to the closure of the image of $V(\cI)$ in $B$. We should prove that $\cF$ does not fail after base change via an appropriate permissible blow up.

We, again, proceed by induction on $n=\logdim(X/B)$. Assume first that $\cF(f,\ucI)$ fails in Step 1 of Section \ref{generalcase}, that is, $\cD^{\infty}_{X/B}(\ucI)$ is not monomial. By the monomialization Theorem~\ref{monomialization} there exists a blow up $B'\to B$ with the base change $g\:X'=X\times_BB'\to X$ such that $g^{-1}\cD^{\infty}_{X/B}(\cI)$ is monomial. Let $f'$ denote the morphism $X'\to B'$ and let $\ucI'=g^{-1}(\ucI)$. Then $g^{-1}\cD^{\infty}_{X/B}(\ucI)=\cD^{\infty}_{X'/B'}(\ucI')$ by Lemmas~\ref{functorlogorder1} and \ref{logcleanlem}, and hence by Remark~\ref{cleanrem}(i) $\cF(f',\ucI')$ does not fail in Step 1 and outputs a Kummer blow up $\sigma'_\infty\:X'_\infty\to X'$.

Let us prove that after an additional permissible blow up $B''\to B'$ the algorithm also does not fail in Step 2 of Section \ref{generalcase}. This will complete the proof since the algorithm blows up an invertible monomial ideal in Step 3, and hence cannot fail there. Recall that in Step 2 the modification $\cF(f',\ucI')$ is the composition of sequences $\sigma'_{\mu_i}$ with $0\le i\le l-1$, and fails if one of those fails. For any $\mu\in\QQ_{\ge 1}$ let $\cF_{\ge\mu}(f',\ucI')$ (resp. $\cF_{>\mu}(f',\ucI')$) be the compositions of $\sigma'_\infty$ and $\sigma'_{\mu_i}$ with $\mu_i\ge \mu$ (resp. $\mu_i>\mu$). The length $l$ is bounded by $\logord_{\cI'_\infty/B'}(X')$, hence by decreasing induction on $\mu$ it suffices to prove that if $\cF_{>\mu}(f',\ucI')$ does not fail, then there exists a permissible blow up $B''\to B'$ such that $\cF_{\ge\mu}(f'',\ucI'')$ does not fail too.

Let $\sigma'_{>\mu}\:X'_{\mu}\dashrightarrow X'$ be the sequence $\cF_{>\mu}(\ucI')$. Then the clean part of $\ucI'_\mu:=(\sigma'_{>\mu})^c(\ucI')$ is of normalized logarithmic order at most $\mu$, and the sequence $\cF_{\ge\mu}(\ucI')$, if exists, is obtained by composing $\sigma'_{>\mu}$ with the order reduction of $(\ucI'_\mu)^\cln$. Moreover, the same is true after any base change $B''\to B'$, since all ingredients of $\cF$, including $\cF_{>\mu}$, are compatible with base changes.

The order reduction of $(\ucI'_\mu)^\cln$ was constructed by pushing forward the order reduction of a coefficient ideal from an \'etale-local maximal contact $H_0$. By induction on $n$, the latter order reduction does not fail after an appropriate permissible blow up $B''\to B'$. Let $g'\:B''\to B$ be the composition. Setting $\ucI''=g'^{-1}(\ucI)$, $\sigma''_{>\mu}=\cF_{>\mu}(\ucI'')$ and $\ucI''_\mu=(\sigma''_{>\mu})^c(\ucI'')$, we have now achieved that the order reduction of $(\ucI''_\mu)^\cln$ does not fail, and hence $\cF_{\ge\mu}(\ucI'')$ does not fail too. This concludes the induction in Step 2 and finishes the proof.

\section{Relative logarithmic desingularization}\label{desingsec}
\addtocontents{toc}
{\noindent We deduce the relative logarithmic desingularization theorems as well as its absolute version.}

This section is devoted to proving Theorem \ref{Proj:toroidalization}. As in the classical case, it is easy to give a local construction based on principalization and the main issue is to prove functoriality, including independence of the embedding. This will be easier than in \cite{ATW-principalization} because we have developed the theory for general logarithmically regular morphisms, hence once an algorithm (depending on choices) is constructed all its properties can be checked on formal completions rather than \'etale locally.

\subsection{The local construction}
As one always does in Hironaka's approach to desingularization, one applies the principalization algorithm and stops it one step before blowing up the strict transform.

\begin{proposition}\label{localdesingprop}
Let $f\:X\to B$ be a relative logarithmic orbifold and let $i\:Z\into X$ be a strict closed immersion of constant codimension such that the morphism $g\:Z\to B$ is generically logarithmically regular. Assume that the relative logarithmic principalization $\cF(f,\cI_Z)$ of the defining ideal $\cI_Z\subseteq\cO_X$ of $Z$ is defined and denote it $\sigma\:X_t\dashrightarrow X$. Then the generic points of $Z$ are blown up at the same stage $\sigma_l\:X_{l+1}\to X_l$, and the strict transform $Z_{l}\into X_l$ of $Z$ is a union of connected components of the center of $\sigma_l$. In particular, $g_l\:Z_l\to B$ is logarithmically regular and $Z_l\to Z$ is a relative logarithmic resolution of $g$ that will be denoted $g_\res\:Z_\res\to B$.
\end{proposition}
\begin{proof}
First, the claim is \'etale local on $X$ and $B$, hence we can assume that they are schemes. Recall that $\sigma$ is the order reduction of $(\cI_Z,1)$ and let $(\cI_i,1)$ denote its controlled transform to $X_i$. By $Z_i\into X_i$ we denote the strict transform of $Z$.

Let $z$ be a generic point of $Z$. Since $g$ is logarithmically regular at $z$, its image is the generic point $b\in B$. Consider first how the algorithm behaves on the localization $Z_z=\Spec(\cO_z)\to B_b=\Spec(\cO_b)$. Let $d$ be the codimension of $Z$ in $X$. Since $z\to b$ is logarithmically regular, $z$ is a suborbifold of $Z_z$ of codimension $d$. Therefore the algorithm simply restricts $d$ times to maximal contacts $H_z^i$, and blows up the $d$-th maximal contact $H_z^d=z$ at the initial cleaning step on $H^d$. In particular, the algorithm behaves similarly at all maximal points of $Z$, and they all are blown up at the same stage $l$. Moreover on this stage one works on a $d$-th maximal contact $H^d\into X_l$ and blows up $H^d$. 

Each generic point of $Z_l$ is a generic point of $H^d$, hence the reduction $H$ of $Z_l$ is the union of the connected components of $H^d$ contained in $Z_l$. In particular, $\cI_l\subseteq\cI_{Z_l}\subseteq\cI_H$. On the other hand, $H^d$ is the iterated maximal contact to $(\cI_l,1)$, hence $\cI_{H^d}\subseteq\cI_l$. It follows that all inclusions become equalities when restricted onto $X_l\setminus(H^d\setminus H)$, and hence $H=Z_l$.
\end{proof}

Here are two complements concerning the proposition.

\begin{remark}
(i) The value of the invariant at step $\sigma_l$ is $(1\.1,\infty)$ with 1 repeated $d$ times, see also \cite[Proof of Theorem~1.2.4]{ATW-principalization}.

(ii) The assumption on codimension in the proposition is essential. The assumption on generic logarithmic regularity can be removed similarly to \cite[\S7.2.7]{ATW-principalization}. We leave this to the interested reader.
\end{remark}

\subsection{Functoriality and independence of the embedding}
Our next task is to prove that $Z_\res$ is functorial. The argument is formal local and we need some preparations.

\subsubsection{Logarithmic rings}
By a local logarithmic ring we mean a local ring $A$ with a homomorphism of monoids $M\to (A,\cdot)$ such that $M^\times=A^\times$. Giving such a datum is equivalent to giving a local scheme $X=\Spec(A)$ with a Zariski logarithmic structure $M\to A$. A homomorphism of monoids $P\to A$ is called a chart of $(A,M)$ if it is a chart of the logarithmic scheme$(X,M)$.

\subsubsection{Minimal presentations}
Let $i\:O\into A$ be an embedding of local logarithmic rings with a complete $A$, and assume that it has a sharp chart $P\into O$, $P\into Q$, $Q\into A$. By a {\em logarithmically regular $O$-presentation} we mean a factorization $O\into C\onto A$ such that $C$ is a complete logarithmic ring, $O\into C$ is logarithmically regular, and $C\onto A$ is strict. Factorizations correspond to strict closed immersions of $\Spec(A)$ into schemes $\Spec(C)$ such that $C$ is a complete local logarithmic ring logarithmically regular over $O$.

\begin{lemma}\label{reemblem}
Let $i\:O\into A$ be as above, let $k=O/m_O$ and $l=A/m_A$, and let $x=(x_1\.x_n)\subset m_A$ be a set whose image form a basis $\ox$ of $m_\oA/m_\oA^2$, where $\oA=A/m_Oq$ and $q=u^{m_Q}A$ is the maximal monomial ideal of $A$. Then any logarithmically regular $O$-presentation of $A$ is of the form $O\into\hatO_P\llbracket Q\rrbracket\llbracket t_1\.t_m\rrbracket\wtimes_kl\stackrel\phi\to A$ with $\phi(t)=(x,0)$ (i.e. $m\ge n$ and $\phi(t_i)=0$ for $i>n$). In particular, for any pair of logarithmically regular $O$-presentations of $A$ one of them factors through the other one.
\end{lemma}
\begin{proof}
Let $O\into C\onto A$ be a logarithmically regular presentation. Notice that $C/m_C=l$ and $\oM_C=Q$, and choose a family of regular parameters $t=(t_1\.t_m)$. By Lemma~\ref{logregchart} $C$ is of the form $\hatO_P\llbracket Q\rrbracket\llbracket t\rrbracket\wtimes_kl$ and it remains to show that we can rechoose $t$ so that $\phi(t)=(x,0)$.

Passing to the logarithmic fibers over $O$, that is, factoring by $m_Ou^{m_Q}$, we obtain a surjection $\ophi\:\oC\onto\oA$, where $\oC=l\llbracket\ot\rrbracket$ and $\ot$ is the image of $t$. Replace $\ot$ by a tuple such that $(\ot_1\.\ot_n)$ is mapped to $\ox$ and $(\ot_{n+1}\.\ot_m)$ lies in $\Ker(\ophi)$ and is mapped to a basis of $\Ker(m_\oC/m^2_\oC\to m_\oA/m^2\oA)$. Then $\ot$ is a basis of $m_\oC/m_\oC^2$, hence a family of coordinates of $\oC$, and any lift $t\in C$ of $\ot$ is a family of regular parameters of $C$ over $O$.

We claim that the map $\phi\:C\to\oC\times_\oA A$ is onto. If $I=\Ker(C\to A)$ and $J=m_Ou^{m_Q}C=\Ker(C\to\oC)$, then $\oA=C/(I+J)$. Dividing the exact sequence $0\to C\to C^2\to C\to 0$ with $c\mapsto (c,c)$ and $(c_1,c_2)\mapsto c_1-c_2$ by the exact subsequence $0\to I\cap J\to I\oplus J\to I+J\to 0$ we obtain an exact sequence $0\to C/(I\cap J)\to A\oplus \oC\to \oA\to 0$, and hence $C\onto C/(I\cap J)\toisom\oC\times_\oA A$. Since $\ophi(\ot)=(\ox,0)$, there exists a lift $t\in C$ such that $\phi(t)=(x,0)$.
\end{proof}

\begin{lemma}\label{liftpreslem}
Let $C\onto A$ be a strict surjective homomorphism of complete local rings. Then for any logarithmically regular homomorphism $A\to A'$ of complete local logarithmic rings can be lifted to a logarithmically regular homomorphism $C\to C'$ of complete local ring. Namely, the homomorphism $C\to A'$ factors as $C\to C'\onto A'$ such that $C\to C'$ is logarithmically regular and $A'=A\otimes_CC'$.
\end{lemma}
\begin{proof}
Let $l=A/m_A$ and $l'=A'/m_{A'}$, and let $Q=\oM_A$ and $Q'=\oM_{A'}$. Then $A'=A_Q\llbracket Q'\rrbracket\llbracket t_1\.t_m\rrbracket\wtimes_ll'$, and one can take $C'=C_Q\llbracket Q'\rrbracket\llbracket t_1\.t_m\rrbracket\wtimes_ll'$ with the natural morphism $C'\to A'$.
\end{proof}

\subsubsection{Functoriality}
Now we can prove the main functoriality result about relative desingularization.

\begin{proposition}\label{indepprop}
Assume that $f_j\:X_j\to B$, $j=1,2$ are two logarithmic orbifolds and $i_j\:Z_j\into X_j$ are strict closed immersions of constant codimensions such that $\cF(f_j,\cI_{Z_j})$ are defined. Then for any pair of logarithmically regular morphisms $Z\to Z_1$ and $Z\to Z_2$ with a common source, the induced relative desingularizations of $Z$ coincide. Namely, there is an isomorphism of $Z$-stacks $(Z_1)_\res\times_{Z_1}Z=(Z_2)_\res\times_{Z_2}Z$.
\end{proposition}
\begin{proof}
By flat descent, it suffices to check the isomorphism of modifications after replacing $Z$ by a flat covering. For example, we can replace $Z$ by its \'etale covering, or we can replace $X_1$ by its \'etale covering $X'_1$ and replace $Z_1$ and $Z$ by their base changes with respect to $X'_1\to X$. In this way one easily reduces the claim to the case when $B,Z_i,X_i$ and $Z$ are schemes. Moreover, it suffices to prove that for any point $z\in Z$ both $\tau_j\:(Z_j)_\res\to Z_j$ are pulled back to the same modification of $\hatZ_z=\Spec(\hatcO_{Z,z})$.

Let $z_j\in Z_j$ be the images of $z$. For shortness we denote $i_j(z_j)$ by $z_j$ and set $\hatZ_j=\Spec(\hatcO_{Z_j,z_j})$, $\hatX_j=\Spec(\hatcO_{X_j,z_j})$. The morphism $\hatX_j\to X$ is regular by the quasi-excellence of $X$ and $\hati_j\:\hatZ_j\into\hatX_j$ is the base change of $i$, hence the principalization of $\hatZ_j$ in $\hatX_j$ is the pullback of the principalization of $Z_j$ in $X_j$ by the functoriality. In particular, the $B$-desingularization of $\hatZ_j$ obtained from $\hati_j$ is the base change of $\tau_j$. This reduces the claim to the particular case when $Z$, $Z_j$ and $X_j$ are spectra of complete local rings. In addition, we can assume that $B=\Spec(\cO_b)$, where $b=g(z)$.

By Lemma~\ref{liftpreslem}, the embedding $\hatZ_j\into\hatX_j$ can be lifted to an embedding $\alpha_j\:\hatZ_z\into Y_j$, where $Y_j$ is logarithmically regular over $\hatX_j$ and hence also over $B$. By functoriality of the principalization, see Theorem~\ref{Proj:principalization}(iii), the pullbacks of $\tau_j$ to $\hatZ_z$ coincide with the $B$-desingularizations of $\hatZ_z$ induced by the logarithmically regular $\cO_b$-presentations $\alpha_j$. It remains to note that the latter desingularizations of $\hatZ_z$ coincide by Lemma~\ref{reemblem} and the re-imbedding principle, see Theorem~\ref{Proj:principalization}(iv).
\end{proof}

As an immediate corollary we obtain

\begin{corollary}\label{indepcor}
The relative desingularization $Z_\res$ defined in Proposition~\ref{localdesingprop} depends only on the morphism $Z\to B$ and is independent of the embedding $i\:Z\into X$. Moreover, if $Z'$ is another logarithmic $B$-orbifold satisfying assumptions of the proposition and $Z'\to Z$ is a logarithmically regular $B$-morphism, then the desingularizations are compatible: $Z'_\res=Z_\res\times_ZZ'$.
\end{corollary}

\subsection{The desingularization method}\label{desingmethodsec}
Now, let us prove Theorem~\ref{Proj:toroidalization}. We say that a morphism $g\:Z\to B$ is {\em locally embeddable} into a relative logarithmic orbifold if there exists an \'etale base change $u\:B'\to B$ and an \'etale covering $v\:Z'\to Z\times_BB'$ such that the morphism $g'\:Z'\to B'$ factors into the composition of a strict closed immersion $Z'\into X'$ and a relative logarithmic orbifold $Z'\to B'$. For any such $g$ we have defined in Proposition~\ref{localdesingprop} a resolution $g'_\res$ of $g'$ and showed in Corollary~\ref{indepcor} that it is independent of the embedding $Z'\into X'$.

We claim that $g'_\res$ descends to a desingularization $g_\res$ of $Z\times_BB'$. Assume first that $B=B'$. The two pullbacks of $g'_\res$ with respect to the projections $Z'\times_ZZ'\to Z'$ induce the same desingularization of the source by Proposition~\ref{indepprop}. Therefore,  $g'_\res$ descends to $g_\res$ by \'etale descent. Descent with respect to a base change $B'\to B$ is done similarly. Functoriality of the desingularization asserted in claim (iii) of Theorem~\ref{Proj:toroidalization} follows from Corollary~\ref{indepcor}. Parts (i) and (ii) about existence and base changes directly follow from the analogous claims in the principalization theorem \ref{Proj:principalization}.

\subsection{The absolute case}\label{abscase}
This section is devoted to proving Theorem \ref{absdesingth}. Analogously to \cite[Section 4]{Temkin-qe}, this will be proved by reducing to the case of schemes over complete local rings via an appropriate localization procedure performed by induction on codimension. The latter case is covered by Theorem~\ref{Proj:schemedesing}.

\subsubsection{Induction on codimension}\label{Sec:toroidalization-induction-on-dimension}
Let $Z$ be as in Theorem \ref{absdesingth}. For a logarithmic scheme $X$ let $X_\sing=X\setminus X_\reg$ denote its logarithmic singularity locus. We will construct a sequence of $(Z_i)_\sing$-supported blow ups $Z_{i+1}\to Z_i$ with $Z_0=Z$, such that the image $T_i\subset Z$ of $(Z_i)_\sing$ contains only points of codimension at least $i+1$. Given such a sequence, the closed sets $T_i$ form a decreasing family, and hence stabilize. By the codimension condition, this implies that $T_i=\emptyset$ for some $i$, and then $Z_i\to Z$ is a desingularization. Moreover, being a composition of $Z_\sing$-supported blow ups, it is a $Z_\sing$-supported blow up itself.

Assume now that a sequence $Z_i\to\dots\to Z_0$ as above is given, and let us construct a blow up $h\:Z_{i+1}\to Z_i$ as required. The closed set $T_i$ contains finitely many points of codimension $i+1$, let us denote them $z_1\.z_n$. Fix for a while $z=z_j$ and let $Z_z=\Spec(\cO_z)$ and $\hatZ=\Spec(\hatcO_z)$ be the localization and the completion with the induced logarithmic structures. The morphisms $\hatZ\to Z_z\to Z$ are strict and regular, hence the same is true for the base changes $\hatZ_i\to (Z_z)_i\to Z_i$, where $\hatZ_i=\hatZ\times_ZZ_i$ and $(Z_z)_i=Z_z\times_Z Z_i$. We will prove in Corollary~\ref{embcor} below that $\hatZ_i$ is locally embeddable into a logarithmic manifold, but let us use it to complete the proof first. Applying Theorem~\ref{Proj:schemedesing} to the morphism $\hatZ_i\to\Spec(\QQ)$, we obtain a blow up $\hatZ_{i+1}\to\hatZ_i$ whose source is logarithmically regular and whose center $\hatV$ is in $(\hatZ_i)_\sing$ and hence is mapped to $(Z_i)_\sing$. Therefore, the image of $(\hatZ_i)_\sing$ in $Z$ lies in the intersection of the localization $Z_z$ and $T_i$, and hence consists of the single point $z$.

We claim that $\hatV$ comes from a closed subscheme $V\into (Z_z)_i$. Indeed, $\hatV$ lies in a closed subscheme defined by $m_z^n\cO_{\hatZ_i}$ for a large enough $n$, and the latter is mapped isomorphically onto the closed subscheme of $(Z_z)_i$ defined by $m_z^n\cO_{(Z_z)_i}$. Since $\hatZ_i\to(Z_z)_i$ is regular, the desingularization $\hatZ_{i+1}\to\hatZ_i$ is the base change of the blow up $h\:(Z_z)_{i+1}\to (Z_z)_i$ with center $V$, and $h$ is a desingularization because $\hatZ_{i+1}\to(Z_z)_{i+1}$ is a logarithmically regular surjective morphism with a logarithmically regular source.

So far, for any $z_j$ we have constructed a desingularizing blow up $h_j\:(Z_{z_j})_{i+1}\to (Z_{z_j})_i$ whose center $V_j$ is mapped to $z_j$ in $Z_{z_j}$. Take $V$ to be the smallest closed subscheme restricting to $V_j$ in $Z_{z_j}$. Namely, we take $V$ to be the schematic image of $\coprod_jV_j$ under the morphism $\coprod_j(Z_{z_j})_i\to Z_i$, and notice that $\coprod_jV_j=V\times_Z\coprod_jZ_{z_j}$ is schematically dense in $V$. (For example, extend first to the open subscheme obtained from $Z_i$ by removing pairwise intersections $\overline{\{z_j\}}\cap\overline{\{z_k\}}$ of Zariski closures.) Since $V_j\subseteq((Z_{Z_j})_i)_\sing\subseteq(Z_i)_\sing$, we have that $V\subseteq(Z_i)_\sing$. We define $h\:Z_{i+1}\to Z_i$ to be the blow up along $V$, and claim that it is as required. Clearly, $h$ is trivial over any point of $Z$ of codimension at most $i+1$ except $z_1\. z_n$. As for each $z_j$, the base change $h\times_ZZ_{z_j}$ is the blow up along $V\times_ZZ_{z_j}=V_j$, hence it is the desingularization $h_j$. This shows that $Z_{i+1}$ is logarithmically regular over all points of $Z$ of codimension at least $i+1$, hence the blow up $h$ is as required.

\subsubsection{Functoriality}
If $h\:Z'\to Z$ is a logarithmically regular morphism, then $Z'_\reg=h^{-1}(Z_\reg)$ by \cite[Corollary~5.1.3(1) and (3)]{Molho-Temkin}. Completions of qe schemes are regular morphisms, hence for a point $z'\in Z'$ with $z=h(z')$ the completion $\hatZ'_{z'}\to\hatZ_z$ is logarithmically regular. Therefore, all ingredients of our construction, including the desingularization provided by Theorem~\ref{Proj:schemedesing}, are functorial with respect to logarithmically regular quasi-saturated morphisms. This proves the functoriality assertion in Theorem~\ref{absdesingth}.

\subsubsection{Local embeddability}
    It remains to prove Corollary \ref{embcor} --- the claim about existence of embeddings used in Section \ref{Sec:toroidalization-induction-on-dimension}. We start with a logarithmic Cohen's structure theorem.

\begin{lemma}\label{Cohen}
Let $Z$ be an fs logarithmic scheme whose underlying scheme is the spectrum of a complete local ring $A$ containing $\QQ$. Then there exists a strict closed immersion $Z\into X$, where $X$ is of the form $\Spec(k\llbracket P\rrbracket\llbracket t_1\. t_n\rrbracket)$ with $k$ a field and the logarithmic structure given by an fs monoid $P$.
\end{lemma}
\begin{proof}
The logarithmic structure on $Z$ possesses a sharp chart $i\:P\to A$. By the classical Cohen's structure theorem there exists a surjection $k\llbracket t_1\. t_n\rrbracket\onto A$. Therefore, the surjection $k\llbracket P\rrbracket\llbracket t_1\. t_n\rrbracket\onto A$ mapping $P$ to $i(P)$ gives rise to a strict closed immersion as required.
\end{proof}

\begin{corollary}\label{embcor}
Let $Z$ be an fs logarithmic scheme whose underlying scheme is the spectrum of a complete local ring $A$ containing $\QQ$. Then any logarithmic scheme $Z'$ of finite type over $Z$ is locally embeddable into a qe logarithmic manifold.
\end{corollary}
\begin{proof}
By Lemma \ref{Cohen}, $Z$ admits a strict closed immersion into a logarithmic scheme $Y=\Spec(C)$, where $C=k\llbracket P\rrbracket\llbracket t_1\. t_n\rrbracket$. It follows easily, that $Z'$ locally admits a strict closed immersion into a logarithmic scheme $Y'$ which is logarithmically smooth over $Y$. It remains to recall that $Y$ is a logarithmic manifold by Lemma~\ref{logmanlem}, and hence $Y'$ is a logarithmic manifold by Theorem~\ref{manifoldth}.
\end{proof}

\section{Extension of main theorems to other categories}
\addtocontents{toc}
{\noindent We extend relative principalization and desingularization to analytic spaces and schemes over valuation rings.}

In this section we will use functoriality to extend Theorems \ref{Proj:principalization}, \ref{Proj:toroidalization} and \ref{Proj:schemedesing} to other settings.

\subsection{Morphisms of finite presentation}\label{fpsec}
In \S\ref{fpsec} we will work with qcqs (quasi-compact and quasi-separated) stacks and schemes of characteristic zero, but they do not have to be noetherian. Our main goal is to desingularize schemes over arbitrary valuation rings, but our arguments apply to general non-noetherian bases as well.

\subsubsection{Approximation}
We briefly recall some facts, the main reference is \cite[${\rm IV}_3$, \S8]{ega}. By \cite[C.9]{ThomasonTrobaugh} any qcqs scheme $B$ of characteristic zero is a filtered limit of a family $\{B_i,f_{ij}\}$ with $B_i$ of finite type over $\QQ$ and $f_{ij}$ affine. If $B$ is integral of characteristic zero, we can take all $B_i$ to be integral of characteristic zero. Furthermore, any $B$-scheme $X$ of finite presentation is the base change of a $B_i$-scheme $X_i$ of finite type for a large enough $i$, and any two choices of such an $X_i$ become isomorphic already after base change to some $B_j$. If $X\to B$ is smooth, \'etale, proper, etc., then the same is true for $X_j\to B_j$ for a large enough $j$. A similar theory exists for morphisms of $B$-schemes, coherent sheaves on them, etc.

The above theory easily extends to logarithmic schemes: a quasi-coherent saturated logarithmic structure $M\to\cO_B$ is a direct colimit of its fs logarithmic substructures $M_\alpha\into M$, and each $M_\alpha$ is obtained by pullback of an fs logarithmic structure $M_{i\alpha}$ on some $X_i$. Therefore, $(X,M_\alpha)$ is the filtered limit of fs logarithmic scheme $(X_j,f_{ij}^*M_{i\alpha})$, and varying $\alpha$ we obtain that $(X,M)$ is the limit of fs logarithmic scheme of finite type over $\ZZ$. In the same way one approximates morphisms of logarithmic schemes, their properties, etc. In particular, using the chart criterion it is easy to see that $X\to B$ is logarithmically smooth if and only if the same is true for approximations $X_i\to B_i$ with a large enough $i$.

\subsubsection{Proof of Theorem \ref{finpres}}
The proof is absolutely the same for $\cF$, $\cR$ and $\cR'$, so we will work with $\cR$ for concreteness. Assume that $g\:Z\to B$ possesses a $\bB$-approximation \'etale locally, that is there exists \'etale covers by schemes $Z_0\to Z$ and $B_0\to B$ and a morphism $g_0\:Z_0\to B_0$ compatible with $g$ and admitting an approximation $\tilg\:\tilZ\to\tilB$ with $B\in\bB$. If $\cR(\tilg)$ fails for any choice of such $\tilg$, then we set $\cR(g)=\cR(g_0)=\emptyset$. Otherwise, choose $\tilg$ with a non-empty $\cR(\tilg)$ and define $\cR(g_0)$ to be the base change of $\cR(\tilg)$. We claim that this is well defined. Indeed, assume that $\tilg'$ is another $\bB$-approximation with $\cR(\tilg')\neq\emptyset$. Since $g_0$ is a filtered limit of approximations, there exist its approximation $\tilg''\:\tilZ''\to\tilB''$ such that both $g_0\to\tilg$ and $g\to\tilg'$ factor through it. Already the pullbacks of $\cR(\tilg)$ and $\cR(\tilg')$ to $\tilZ''$ coincide by compatibility of $\cR$ with base changes.

Next we claim that $\cR(g_0)$ descends to $Z$. Assume first that $B_0=B$. Then the morphism $g_1\:Z_1=Z_0\times_ZZ_0\to B_0$ admits approximations induced from $g_0$. By the above paragraph, both induce the same desingularization $\cR(g_1)$ of $Z_1$, and it follows that the pullbacks of $\cR(g_0)$ to $Z_1$ coincide. Therefore, $\cR(g_0)$ is the pullback of a desingularization of $g$ that we denote $\cR(g)$. Descent with respect to $B_0\to B$ is done similarly. Our construction easily implies that $\cR(g)$ is functorial with respect to logarithmically smooth morphisms and base changes.

It remains to prove the claim about existence, so assume that $g\:Z\to B$ is nice and $B$ is integral as a stack and has a generically trivial logarithmic structure. Assume first that $B$ is a scheme. Then $g$ possesses an approximation $\tilg\:\tilZ\to\tilB$, where $\tilB$ satisfies the same properties. By Theorem~\ref{Proj:toroidalization} there exists a blow up $\tilB'\to\tilB$ with $\tilB'\in\bB$ such that if $\tilg'\:\tilZ'\to\tilB'$ denotes the saturated base change, then $\cR(\tilg\times_{\tilB}\tilB')$ succeeds. By approximation there exists a blow up $B'\to B$ such that the morphism $B'\to \tilB$ factors through $\tilB'$. It satisfies the assertions of the theorem because the saturated pullback $g'\: Z'\to B'$ of $g$ possesses the $\bB$-approximation $\tilg$.

In general, find an \'etale covering by a scheme $B_0\to B$ and let $g_0\:Z_0\to B_0$ be the base change. By the above case, here exists a blow up $B'_0\to B_0$ such that $\cR(g'_0)$ is defined. By \cite[Theorem A]{Rydh} there exists a blow up $B'\to B$ such that $B'\times_BB_0\to B_0$ factors through $B'_0$. It is easy to see that it is a required modification of $B$.

\subsection{Analytic spaces and formal varieties}\label{lastsec}
We conclude the paper with discussing how our results might extend to formal and complex or non-archimedean analytic geometries. First, analogs of logarithmic structures and DM stacks in these categories are defined as their analogs in the theory of schemes. The rest is based on the ideas and results of \cite[\S6.2]{AT2}. For simplicity of notation we assume that the logarithmic structures are trivial.

By a {\em formal variety} $\fX$ we mean a formal scheme locally isomorphic to completion of a variety over a field. So, $\fX$ is covered by opens of the form $\Spf(A)$ with $A=k[x_1\.x_m]\llbracket t_1\.t_n\rrbracket/I$. Let $\fSp$ be a category of analytic spaces or formal schemes as in \cite[\S6.2]{AT2}, where in case (i) -- qe formal schemes, we restrict the category to formal varieties. It is shown in loc.cit. how any object $\fX$ of $\fSp$ is obtained by gluing {\em affinoid} objects $\fX_i$ with excellent rings $A_i=\Gamma(\cO_{\fX_i})$. If $\ff\:\fX\to\fB$ is a regular morphism of affinoid objects with algebras $A$ and $O$, then the morphism $f\:X=\Spec(A)\to B=\Spec(O)$ is regular.

Is $f$ a relative manifold? Probably it is not when $\dim(\fB)>1$. A negative indication is  the following fact communicated to us by Ofer Gabber: if $n\ge 2,m>0$, then the generic fiber of the morphism $$\Spec(k\llbracket x_1\.x_n,y_1\.y_m\rrbracket)\to\Spec(k\llbracket x_1\.x_n\rrbracket)$$ is of dimension $m+n-2$, which is larger than $m$ for $n>2$. Despite this one might still hope for some positive results, as was noticed in Remark~\ref{intromonomrem}.

Assume now that $\dim(\fB)\le 1$. Then our expectation is that $f$ is a relative manifold and this seems plausible especially in light of the special case covered by Lemma~\ref{logmanlem}. This question will be studied in detail elsewhere, but assuming the expectation is true one obtains that principalization of ideals is always possible for $f$, and its analytification or completion induces principalization of ideals for $\ff$. It remains to recall that morphisms of affine schemes corresponding to open immersions in $\fSp$ are regular, hence functoriality of principalization implies that the local principalization we have constructed in $\fSp$ extends to arbitrary regular morphisms. As a corollary one would obtain that analogs of the main Theorems \ref{Proj:principalization}, \ref{Proj:toroidalization}, \ref{Proj:schemedesing} and \ref{Th:semistable} extend to the category $\fSp$.

\appendix

\section{Relative destackification}\label{destackapp}
In \cite[Section 4]{ATW-destackification} we introduced a destackification functor that associates to a simple toroidal orbifold $X$ a blow up $\cD(X)$ such that the coarse moduli space $\cD(X)_\cs$ is also logarithmically regular. This construction was used in \cite[Section~8]{ATW-principalization} to deduce scheme-theoretic resolution of logarithmic varieties from a stack-theoretic one. To resolve morphisms, we need a relative version of this result. Fortunately, the same destackification functor works, but we have to check the additional property that it also preserves logarithmic regularity of certain morphisms. This will require to open the box, and we will have to generalize certain results from \cite{AT1}.

In addition, we will construct a relative version of the destackification $\cD_S(X)_{\cs/S}$ with respect to a given stack $S$. Analogously to defining relative coarse spaces, starting with $\cD$ this will be  done without opening the box by a standard \'etale descent. Using this relativization we obtain projective resolution of morphisms of stacks rather than schemes.

\subsection{\'Etale descent construction}\label{etdessec}

\subsubsection{Assumptions}\label{appassumsec}
Let $X\to S$ be a morphism of logarithmic DM stacks of characteristic zero. We assume that $X$ is logarithmically regular and the relative inertia $I_{X/S}$ is finite diagonalizable and acts trivially on the monoids $\oM_\ox$.

\subsubsection{The relative destackification}
Consider an \'etale presentation $S_1\toto S_0$ of $S$ and set $S_2=S_1\times_{S_0}S_1$ and $X_i=X\times_SS_i$. By \cite[\S2.1.1]{ATW-destackification}, $I_{X_i}=I_{X/S}\times_XX_i$ and hence $X_i$ are simple toroidal orbifolds in the sense of \cite{ATW-destackification} and both morphisms $p_{1,2}\:X_1\toto X_0$ are surjective strict inert \'etale. In particular, the destackification functor from \cite[Theorem~4.1.5]{ATW-destackification} applies to $X_i$ and is compatible with both $p_i$. We will denote it $X_i\mapsto \cD(X_i)$. The same argument applied to the morphisms $X_2\to X_1$ implies that $\cD(X_1)\toto\cD(X_0)$ is a strict \'etale groupoid in DM stacks, hence the quotient stack $[\cD(X_0)/\cD(X_1)]$ is defined by \cite[Lemma~2.1.4]{ATW-destackification}. We denote the latter stack $\cD_S(X)$ and call it the {\em $S$-destackification of $X$}.

\begin{theorem}\label{destackth}
Let $X$ and $S$ be as above, then

(i) The $X$-stack $\cD_S(X)$ depends only on $X\to S$ up to an isomorphism unique up to a unique 2-isomorphism.

(ii) The construction of $\cD_S(X)$ is compatible with representable \'etale morphisms $S'\to S$, that is, $\cD_{S'}(X\times_SS')=\cD_S(X)\times_SS'$.

(iii) The stack $\cD_S(X)$ is logarithmically regular and the $S$-coarsening $\cD_S(X)\to\cD_S(X)_{\cs/S}$ is logarithmically \'etale. In particular, $\cD_S(X)_{\cs/S}$ is logarithmically regular.

(iv) The morphisms $h\:\cD_S(X)\to X$ and $h_0\:\cD_S(X)_{\cs/S}\to X_{\cs/S}$ have a natural blow up structure.

(v) If $X'\to X$ is a surjective logarithmically smooth inert morphism, then the blow ups $\cD_S(X')\to X'$ and $\cD_S(X')_{\cs/S}\to X'_{\cs/S}$ are pullbacks of $h$ and $h_0$.
\end{theorem}
\begin{proof}
To prove (i) we should check that $\cD_S(X)$ is independent of the presentation of $S$. It suffices to consider the case of refinement, that is, $S'_1\toto S'_0$ is another presentation with $S'_0\to S$ factoring through $S_0$. Then the natural equivalence of groupoids $S'_i\to S_i$ gives rise to an equivalence of groupoids $\cD(X'_i)\to\cD(X_i)$, hence to a canonical isomorphism of the quotients.

Let us prove (ii). Choose a presentation $S_1\toto S_0$ of $S$ and set $X_i=X\times_SS_i$, $S'_i=S'\times_SS_i$ and $X'_i=X'\times_SS_i$, where $X'=X\times_SS'$. Then $X'_i\to X_i$ is the base change of the \'etale morphism of schemes $S'_i\to S_i$ and by \cite[Theorem~4.1.5]{ATW-destackification} $\cD(X'_i)=\cD(X_i)\times_{S_i}S'_i=\cD(X_i)\times_SS'$. By descent, passing to the quotients we obtain the asserted isomorphism $\cD_{S'}(X')=\cD_S(X)\times_SS'$.

Properties (iii)--(v) are deduced by \'etale descent from their absolute analogs, which hold for $\cD(X_0)$ by \cite[Theorem~4.1.5]{ATW-destackification}.
\end{proof}

\subsection{Logarithmic regularity}
Now, we are going to strengthen claim (ii) above by showing that $\cD_S(X)_{\cs/S}$ also preserves relative logarithmic regularity.

\begin{theorem}\label{destackth1}
Let $X$ and $S$ be as in \S\ref{appassumsec} and assume that $Z$ is a logarithmically regular logarithmic DM stack and $S\to Z$ is a morphism such that the composition $f\:X\to Z$ is logarithmically regular. Then,

(i) The morphisms $\cD_S(X)\to Z$ and $\cD_S(X)_{\cs/S}\to Z$ are logarithmically regular too.

(ii) If $Z'\to Z$ is a morphism of logarithmically regular logarithmic DM stacks, $S'=S\times_ZZ'$ and $X'=X\times_ZZ'$, then $\cD_{S'}(X')=\cD_S(X)\times_ZZ'$ and $\cD_{S'}(X')_{\cs/S'}=\cD_S(X)_{\cs/S}\times_ZZ'$.
\end{theorem}
\begin{proof}
For shortness we will only check (i). Tracking the same steps one can also see that all ingredients are compatible with the base change $Z'\to Z$, obtaining (ii).

Logarithmic regularity can be checked \'etale locally on the source, hence by Theorem~\ref{destackth}(ii) the claim is \'etale-local on $S$ and we can assume that $S$ is a scheme. This reduces the claim to the absolute case: if $Z$ and $X\to Z$ are logarithmically regular and the map of inertias $I_X\to I_Z$ is trivial (factors through the unit $Z\into I_Z$), then $\cD(X)$ and $\cD(X)_\cs$ are logarithmically regular over $Z$. In the same way, one checks that the problem is \'etale local on $Z$, and hence one can assume that $Z=\Spec(O)$ is affine. Finally, destackification is compatible with inert \'etale morphisms $X'\to X$, and any geometric point $x\to X$ factors through an inert morphism $X'\to X$ whose source is a global $X'=[\Spec(A)/G_x]$. Therefore, we can assume that $X=[\Spec(A)/G]$, where $G$ is an \'etale diagonalizable group acting $Z$-equivariantly, and we can work locally at a $G$-invariant point $x\in X$. In this case, $\cD(X)\to X$ pulls back to the torification $\cT(W)\to W$ of the $G$-action on $W$ (see \cite[\S4.2.1]{ATW-destackification}), and our problem reduces to showing that $\cT(W)$ and $\cT(W)/G$ are logarithmically regular over $Z$.

Recall that torification of balanced actions were constructed in \cite[Theorem~4.6.3 and 5.4.2]{AT1}. To complete the proof we should show that it possesses the relative logarithmic regularity property over $Z$. We assume a familiarity with that paper and briefly indicate the argument.

Since $G$ is a finite group scheme, the localization at $x$ is equivariant, and it suffices to consider the case $X$ and $Z$ are local schemes. Furthermore, by \'etale localization we can assume that both logarithmic structures are Zariski. Fix a sharp chart $Z\to\bfA_P$. By \S\ref{sharpfactorsec} $f$ factors through the sharp factorization $\tilf\:X\to \tilZ=Z_P[\tilP]$. Clearly, $\tilf$ is $G$-equivariant with respect to the trivial action on $\tilZ$. Since, $\tilZ\to Z$ is logarithmically \'etale, it suffices to prove the claim for the morphism $\tilf$. So, in the sequel we also assume that $f$ is sharp.

It is easy to see that $Z\to\bfA_P$ can be extended to a sharp $G$-equivariant chart $Y\to\bfA_Q$. (One can show this directly, or use the first part of the proof of \cite[Proposition~1.2]{Illusie-Temkin}.) The logarithmic fiber $S$ at $x$ is $G$-equivariant, hence we can choose a family of equivariant parameters $s=(s_1\.s_n)$. (The action of $G$ corresponds to a grading, so equivariance simply means that we choose homogeneous parameters that lift a $G$-equivariant basis of the cotangent space.) Also, fix a family of regular parameters $t=(t_1\.t_n)$ at $f(z)$, obtaining a regular morpihsm $Z\to Z_0=\Spec(\QQ[P][t])$. Then $(s,t)$ is the family of regular parameters of $X$ at $x$ and $h\:X\to X_0=\Spec(\QQ[Q][s,t])$ is an inertia preserving regular morphism, where the right hand side is provided with the natural action via the gradings on $Q$ and $s$.
By \cite[Theorem~4.6.3]{AT1}, $\cT(X)=\cT(X_0)\times_{X_0}X$, and this reduces us to the model case claim that $\cT(X_0)$ and $\cT(X_0)/G$ are logarithmically regular over $Z_0$. The explicit computation of $\cT(X_0)$ is given in the proof of \cite[Lemma~4.5.4]{AT1}, and a direct inspection shows that not only $\cT(X_0)$ and $\cT(X_0)/G$ are logarithmically regular, they are also logarithmically regular over $Z_0$. (The parameters $t$ are sort of dummy ones -- they are $G$-invariant, hence do not show up in the formulas for torific ideals and the torification blow up.)
\end{proof}

\begin{remark}
In fact, blowing up the same torific ideal as defined in \cite{AT1} in the absolute case, one obtains a torification for logarithmically regular morphisms $X\to Z$ with an arbitrary logarithmic scheme $Z$, and it is compatible with arbitrary base changes $Z'\to Z$. The latter would then extend to destackification of logarithmically regular stacks over an arbitrary $Z$. However, proving this would require to repeat too many arguments from \cite{AT1} and \cite{ATW-destackification}, and we decided to only consider the case when $Z$ is logarithmically regular, and hence $X$ is logarithmically regular and the absolute torification and destackification has been already defined in the cited papers.
\end{remark}

\section{Regular morphisms}\label{regularapp}

\subsection{The definition}
Recall that a morphism $f\:Y\to Z$ of noetherian schemes is {\em regular} if it is flat and has geometrically regular fibers. This notion is smooth local on $Y$ and fppf local on $Z$, in particular, it extends to morphisms between algebraic stacks. If $f$ is of finite type, then $f$ is regular if and only if it is smooth. Thus, regularity is a natural extension of smoothness to arbitrary morphisms. In fact, the famous theorem of Popescu states that any regular morphism is a filtered limit of smooth ones.

\subsection{Parameters}
Assume that $f\:Y\to Z$ is a regular morphism of schemes, $y\in Y$ is a point, and $S=Y\times_Z\Spec(k(z))$ is the fiber over $z=f(y)$. By a {\em family of regular parameters} of $f$ at $y$ we mean any family $t_1\.t_n\in\cO_y$ whose image is a family of regular parameters of the regular ring $\cO_{S,y}$.

\begin{lemma}\label{regularlem}
Assume that $f\:Y\to Z$ is a morphism of qe stacks of characteristic zero such that $Y$ is a scheme and $f$ is regular at a point $y$.

(i) Let $t_1\.t_l$ be global functions on $Y$. Then the morphism $Y\to Z\times\AA^l$ induced by $f$ and $t_1\.t_l$ is regular at $y$ if and only if $t_1\.t_l$ is a partial family of regular parameters of $f$ at $y$.

(ii) A closed subscheme $X\into Y$ is regular at $y$ over $Z$ if and only if it is given by the vanishing of a partial family $t_1\.t_l$ of regular parameters at $f$ at $y$.
\end{lemma}
\begin{proof}
We will prove only (i), since the other claim is proved similarly. The claim is flat-local on the base and local at $y$, hence we can assume that the schemes are local with closed points $y$ and $z=f(y)$.

Recall that a homomorphism $\phi\:A\to B$ of qe local rings is regular if and only if its completion $\hatphi\:\hatA\to\hatB$ is regular. Indeed, the completion $A\to\hatA$ is regular, so if $\hatphi$ is regular, then $A\to\hatB$ is regular. The latter is the composition of $\phi$ and the regular homomorphism $B\to\hatB$, hence in this case $\phi$ is regular. Conversely, if $\phi$ is regular, then $\hatphi$ is regular by \cite[Corollary~2.4.5]{Temkin-qe}.

Set $C=\hatcO_y$ and $A=\hatcO_z$. By the above paragraph the homomorphism $\psi\:A\to C$ is regular, and we should prove that the homomorphism $\phi\:B=A\llbracket t_1\.t_l\rrbracket\to C$ is regular if and only if $t_1\.t_l$ is a partial family of regular parameters.

Recall that $\phi$ is regular if and only if there exists an isomorphism of $B$-algebras $B\wtimes_kl\llbracket s_1\.s_m\rrbracket\toisom C$, where $k$ and $l$ are compatible fields of definition (e.g., see \cite[Remark~2.2.12]{ATLuna})). Moreover, one can take $s_1\.s_m$ to be any family of regular parameters of $\phi$. If $\phi$ is regular, then choosing such an $s$ we obtain that $(t,s)$ is a family of regular parameters of $\psi$, and hence $t$ is a partial family. Conversely, if $t$ is a partial family, complete it to a full family $(t,s)$ obtaining an isomorphism $A\wtimes_kl\llbracket s,t\rrbracket\toisom C$. Hence $B\wtimes_kl\llbracket s\rrbracket\toisom C$, and $\phi$ is regular.
\end{proof}

\begin{remark}
(i) The same argument applies in any characteristic when the extension $k(y)/k(z)$ is separable (that is, $y$ is a simple regular point). The assertion is false when the extension is inseparable.

(ii) The lemma holds for noetherian schemes without the quasi-excellence assumption, but proving this would require more work. (For example, one could use the characterization of regular homomorphisms in terms of the cotangent complex.)
\end{remark}

\section{Relative Riemann-Zariski spaces}\label{RZappend}
In this appendix we work with arbitrary schemes. Our goal is to prove Theorem \ref{cofinalth}. Its claim is an easy corollary of the flattening theorem of Raynaud-Gruson in the case when $h$ is flat over $Z\setminus T$. However, in our application $h$ can be an arbitrary Kummer cover, so it does not have to be flat, and we have to find another argument. Our proof uses a technique of relative RZ spaces developed in \cite{Temkin-RZ}, which can also be used to prove the flattening theorem.

\subsection{The space $\RZ_U(Z)$}
Recall that to any qcqs (i.e. quasi-compact and quasi-separated) scheme $Z$ with a schematically dense open $U\into Z$ one associates the relative RZ space $\fZ=\RZ_U(Z)$, which is the limit of all $U$-modifications of $Z$ considered in the category of locally ringed spaces. By a strong version of Chow lemma (or by the flattening theorem), the family of $T$-supported blow ups is cofinal, where $T=Z\setminus U$, hence $\fZ$ is also the limit of all $T$-supported blow ups of $Z$.

\begin{lemma}\label{cofinallem}
With the above notation, a family of $U$-modifications $\{Z_i\}_i$ is cofinal if and only if the morphism of locally ringed spaces $p\:\RZ_U(Z)\to\lim_iZ_i$ is an isomorphism.
\end{lemma}
\begin{proof}
Only the opposite direction needs a proof. So, assume that $p$ is an isomorphism and for a given $U$-modification $g\:Z'\to Z$ let us find a finer $U$-modification of the form $Z_i\to Z$. Fix for a while a point $\fz\in\fZ=\RZ_U(Z)$ and let $Z\in Z$, $z'\in Z'$ and $z_i\in Z_i$ be its images. Then $\fZ_\fz=\Spec(\cO_\fz)$ coincides with $\lim_i\Spec(\cO_{z_i})=\lim_{ij}Z_{ij}$, where $\{Z_{ij}\}_j$ is the family of open neighborhoods of $z_i$. We claim that the morphism $\fZ_\fz\to Z'$ factors through some $Z_{ij}$.

By \cite[Theorem~D(i)]{Rydh} there exists a closed immersion $Z'\into\tilZ'$ with $\tilZ'$ of finite presentation over $Z$. (For our argument, we could also replace $Z'$ by a small enough affine neighborhood of $z'$, and then existence of such $\tilZ'$ is obvious.) Let $Z'$ be given by $\cI\subset\cO_{\tilZ'}$. This ideal is finitely generated on the open subscheme $\tilU'=U\times_Z\tilZ'$ because $U\into Z$ is finitely presented. By \cite[Lemma~6.9.10.1]{egaI} there exists a finitely generated ideal $\cI'\subseteq\cI$ such that $\cI'|_U=\cI|_U$. Replacing $Z'$ by $V(\cI')$ we can also achieve that $\tilU'=U$, and then $U$ is an open subscheme of $\tilZ'$ whose schematic closure is $Z'$. By \cite[${\rm IV}_3$, Theorem~8.14.2]{ega}, $\fZ_\fz\to\tilZ'$ factors through some $Z_{ij}$. Since $U\times_ZZ_{ij}$ is schematically dense in $Z_{ij}$ and is mapped to $U\subseteq\tilZ'$, the morphism $Z_{ij}\to\tilZ'$ factors through $Z'$.

Now, for any $\fz\in\fZ$ fix an appropriate $Z_{ij}$ that we denote $Z_{\fz}\subseteq Z_{i(\fz)}$. Their preimages form an open covering of $\fZ$, and since the latter is quasi-compact by \cite[Proposition~3.3.1]{Temkin-stable}, there exist $\fz_1\.\fz_n$ such that the preimages of $Z_{\fz_1}\.Z_{\fz_n}$ cover $\fZ$. Taking $i$ which dominates $i(\fz_1)\.i(\fz_n)$ and replacing $Z_{\fz_j}$ by their preimages in $Z_i$ we can assume that $i=i_1=\ldots=i_n$. The map $\fZ\to Z_i$ is surjective (for example, this follows from the valuative description of $\fZ=\RZ_U(Z_i)$ below), hence $Z_i=\cup_{j=1}^nZ_{\fz_j}$. By the construction the morphisms $Z_{\fz_j}\to Z$ factor through $Z'$, and they agree on the intersections because all maps are uniquely determined by their behaviour on the preimage of $U$. This proves that $Z_i\to Z$ factors through $Z'$, as claimed.
\end{proof}

\subsection{The space $\Val_U(Z)$}
Our next goal is to describe the RZ spaces explicitly in terms of semivaluations.

\subsubsection{Recollections}
The set $\fZ=\RZ_U(Z)$ is described in \cite[Proposition 2.2.1 and Corollary 3.4.7]{Temkin-RZ} by constructing an explicit locally ringed space $\Val_U(Z)$ with a homeomorphism $\Val_U(Z)\toisom\fZ$. Points of $\Val_U(Z)$ are minimal semivaluations on $Z$ with kernel on $U$, which can also be described as follows: a point is a pair $u\in U$, $\psi\:\Spec(R)\to Z$, where $R$ is a valuation ring of $k(u)$ and $\psi$ extends the morphism $u\to Z$, such that $u\to U$ cannot be extended to an open subscheme of $\Spec(R)$ containing at least two points. In particular, $u$ is uniquely determined by $\fz$ (and this gives a retraction $\fZ\to U$, which sends each to point to its minimal generization contained in $U$). In addition, $\cO_\fz$ is the preimage of $R$ under $\cO_u\onto k(u)$.

\subsubsection{Semivaluation rings}
The ring $\cO_\fZ$ composed from $\cO_u$ and $R$ is an example of the notion of a {\em semivaluation ring $O$ with a semifraction ring $A$} introduced in \cite[\S2.1]{Temkin-RZ}: it consists of a local ring $(A,m)$ and a subring $O\subseteq A$ such that $m\subseteq O$ and $R=O/m$ is a valuation ring with fraction field $A/m$. Equivalently, $O$ is the preimage of a valuation ring $R$ of $A/m$.

\begin{remark}
(i) The diagram formed by $A,O,R$ and $A/m$ is bicartesian, and the same is true for the dual diagram of affine schemes. In a sense, $\Spec(O)$ is obtained by gluing $\Spec(A)$ and $\Spec(R)$ along the point $\Spec(A/m)$, which is closed in $\Spec(A)$ and generic in $\Spec(R)$.

(ii) The notion is chosen due to the fact that this datum induces a semiavaluation on $A$ with kernel $m$, and any semivaluation gives rise to an appropriate semivaluation ring.
\end{remark}

The following result generalizes the classical fact that if $A$ is a valuation ring and $B$ is a local normal $A$-algebra with $\Frac(B)/\Frac(A)$ algebraic, then $B$ is a valuation ring.

\begin{lemma}\label{semivallem}
Assume that $O$ is a semivaluation ring with semifraction ring $(A,m)$ and $O'$ a normal local $O$-algebra such that the fiber over $m$ is non-empty and for any prime of $O'$ over $m$ the extension of residue fields is algebraic. Then $O'$ is a semivaluation ring with semifraction field $A'=O'\otimes_OA$.
\end{lemma}
\begin{proof}
Set $m'=\sqrt{mO'}$ and $S=O\setminus m$. Then $A=O_S$ and $m_S=m$, and hence $A'=O'_S$ and $m'=m'_S$. We claim that the reduced ring $R'=O'/m'$ is normal. If not, then there exists $\oa,\ob\in R'$ such that $\ob$ is regular, $\oa\notin\ob R'$ and $\oa^n+\sum_{i=0}^{n-1}\oc_i\oa^i\ob^{n-i}=0$ for some $n$ and $\oc_i\in R'$. Choosing arbitrary lifts $a,b,c_i\in O'$ we obtain that $x=a^n+\sum_{i=0}^{n-1}c_ia^ib^{n-i}\in m'$. Since $\ob$ is regular, $b$ is not contained in any prime ideal of $O'$ over $m$, and hence $b$ becomes a unit in $O'_S$. In particular, $b^{-n}x\in S^{-1}m'=m'$ and replacing $c_0$ by $c_0-b^{-n}x$ we achieve that $a^n+\sum_{i=0}^{n-1}c_ia^ib^{n-i}=0$. Thus $a/b\in O'$ by the normality of $O'$, and reducing modulo $m'$ we obtain that $\oa/\ob\in R'$. A contradiction.

Thus, $R'$ is a local normal ring, hence a domain. In particular, the ideal $m'$ is prime and it is the whole fiber over $m$. Furthermore, $R'$ contains the valuation ring $R$ and $\Frac(O'/m')$ is algebraic over $\Frac(O/m)$ by the assumptions. So, by the classical theory, $R'$ is a valuation ring. Finally, $A'=O'_S$ is a local ring with the maximal ideal $m'_S=m'$ and field of fractions $\Frac(R')$. So, $O'$ is a semivaluation ring with semifraction field $A'$.
\end{proof}

\subsection{A cofinality theorem}
Now we can prove the main result of the appendix. Let us call a morphism of schemes $V\to U$ {\em quasi-integral} if for any point $v\in V$ with $u=h(v)$ the extension $k(v)/k(u)$ is algebraic. This class of morphisms contains quasi-finite and integral ones, but is wider than compositions of open immersions and integral morphisms. For example, it also contains localizations.

\begin{theorem}\label{cofinalth}
Let $Z$ be a qcqs scheme with a schematically dense normal open subscheme $U$ and $T=Z\setminus U$. Assume that $h\:Y\to Z$ is a dominant morphism such that $V=U\times_ZY$ is normal, quasi-integral over $U$ and schematically dense in $Y$. Let $\{Z_i\}_i$ be the family of $T$-supported blow ups of $Z$, and for each $i$ let $Y_i$ be the schematic closure of $V$ in $Z_i\times_ZY$. Then the family of all finite $T$-modifications of the schemes $Y_i$ is cofinal in the family of all $T$-modifications of $Y$.
\end{theorem}
\begin{proof}
If $\{Y_{ij}\}_j$ denotes the family of all finite $T$-modifications of $Y_i$, then $\tilY_i:=\lim_j Y_{ij}$ is normal (it corresponds to the integral closure of $\cO_{Y_i}$ in $\cO_V$). Set $\fZ=\RZ_U(Z)$, $\fY=\RZ_V(Y)$ and $\fY':=\lim_i\tilY_i$. By Lemma~\ref{cofinallem} it suffices to prove that the map $p\:\fY\to\fY'$ is an isomorphism. Choose any point $\fy'\in\fY'$ and let $\fz\in\fZ$, $y_i\in\tilY_i$ and $z_i\in Z_i$ be its images. Since $\cO_{\fy'}=\colim_i\cO_{y_i}$ is a filtered colimit of normal rings, it is a normal $\cO_\fz$-algebra. Set $\fZ_\fz=\Spec(\cO_\fz)$ and $\fY'_{\fy'}=\Spec(\cO_{\fy'})$. By the valuative description of RZ spaces, $\cO_\fz$ is a semivaluation ring with semifraction field $\cO_u$ for some $u\in U$ and $U_u=\Spec(\cO_u)$ is isomorphic to $\fZ_\fz\times_ZU$. Therefore $\fY'_{\fy'}\times_{\fZ_\fz}U_u$ is isomorphic to $\lim_i\Spec(\cO_{y_i})\times_ZU$. It follows that $\fY'_{\fy'}\times_{\fZ_\fz}U_u$ is the limit of localizations of $V$, in particular, the map $\fY'_{\fy'}\times_{\fZ_\fz}U_u\to U_u$ is quasi-integral, and the fiber over $u$ is quasi-integral. Thus, Lemma~\ref{semivallem} applies to the extension $\cO_\fz\subset\cO_{\fy'}$ and we obtain that $\cO_{\fy'}$ is a semivaluation ring with semifraction ring $\cO_{\fy'}\otimes_{\cO_\fz}\cO_u$. Spectrum of the latter is a limit of open subschemes of $V$, hence it is of the form $\Spec(\cO_{v'})$ for a point $v'\in V$.

By the valuative description of $\fY$, giving a point $\fy$ above $\fy'$ is equivalent to giving a local homomorphism $\cO_{\fy'}\to\cO_\fy$, where the target is a semivaluation ring composed from its semifraction ring $\cO_v$ and a valuation ring $R\subseteq k(v)$, where $v\in V$ and $\cO_{v'}\to\cO_v$ is a specialization. Since the morphism $u\to Y$ cannot be extended to a non-trivial part of $\Spec(R)$, it follows easily that $v=v'$, and then by locality $\cO_{\fy'}=\cO_\fy$. This shows that the fiber over $\fy'$ in $\fY$ consists of a single point with the same local ring, and hence $\fY=\fY'$, as required.
\end{proof}

\bibliographystyle{amsalpha}
\bibliography{principalization}

\end{document}